\documentclass[reqno,12pt,letterpaper]{amsart}
\usepackage{amsmath,amssymb,amsthm,graphicx,mathrsfs,url}\usepackage[usenames,dvipsnames]{color}
\usepackage[colorlinks=true,linkcolor=black,citecolor=black]{hyperref}
\usepackage{amsxtra}
\usepackage{subcaption}
\usepackage{wasysym}
\usepackage[T1]{fontenc}
\usepackage[utf8]{inputenc}
\usepackage{mathtools}
\DeclareMathOperator{\const}{const.}
\def\?[#1]{\textbf{[#1]}\marginpar{\Large{\textbf{??}}}} 

\let\epsilon=\varepsilon % sorry Knuth

\newcommand{\RR}{{\mathbb R}}
\newcommand{\NN}{{\mathbb N}}
\newcommand{\CC}{{\mathbb C}}
\newcommand{\ch}{{\operatorname{c}}}
\newcommand{\ach}{{\operatorname{ac}}}

\newcommand{\TT}{{\mathbb T}}
\newcommand{\ZZ}{{\mathbb Z}}
\newcommand{\id}{\operatorname{id}}

\newtheorem{prop}{Proposition}[section]	
\newtheorem{thm}[prop]{Theorem}
\usepackage{tikz}

\newtheorem{lemm}[prop]{Lemma}
\newtheorem{corr}[prop]{Corollary}

\theoremstyle{definition}
\newtheorem{defi}[prop]{Definition}
\newtheorem{rem}[prop]{Remark}

\newtheorem*{rmk}{Remark}
\newtheorem*{rmks}{Remarks}

\numberwithin{equation}{section}

\DeclareMathOperator{\Spec}{Spec}

\let\Im=\Imag

\let\Re=\Real

\DeclareMathOperator{\supp}{supp}

\DeclareMathOperator{\WF}{WF}

\DeclareMathOperator{\diag}{diag}
\DeclareMathOperator{\ran}{Ran}

\DeclareMathOperator{\Span}{span}

\usepackage{scalerel}

\newcommand\reallywidehat[1]{\arraycolsep=0pt\relax%
\begin{array}{c}
\stretchto{
  \scaleto{
    \scalerel*[\widthof{\ensuremath{#1}}]{\kern-.5pt\bigwedge\kern-.5pt}
    {\rule[-\textheight/2]{1ex}{\textheight}} %WIDTH-LIMITED BIG WEDGE
  }{\textheight} % 
}{0.5ex}\\           % THIS SQUEEZES THE WEDGE TO 0.5ex HEIGHT
#1\\                 % THIS STACKS THE WEDGE ATOP THE ARGUMENT
\rule{-1ex}{0ex}
\end{array}
}

\author{Simon Becker}
\address[Simon Becker]{Department of Mathematics, ETH Zurich, 8004, Zurich, Switzerland.}
\email{simon.becker@ethz.ch}

\author{Jens Wittsten}
\address[Jens Wittsten]{Department of Engineering, University of Bor{\aa}s, SE-501 90 Bor{\aa}s, Sweden}
\email{jens.wittsten@hb.se}

\title[Semiclassical quantization conditions in strained moir\'e lattices]{Semiclassical quantization conditions in strained moir\'e lattices}

\begin{document}

%\subjclass[2020]{81Q20 (primary), 34L40 (secondary)}

%\keywords{Bohr-Sommerfeld rule, WKB expansion, Dirac operator, Harper model, strained honeycomb structure}

\begin{abstract}
In this article we generalize the Bohr-Sommerfeld rule for scalar symbols at a potential well to matrix-valued symbols having eigenvalues that may coalesce precisely at the bottom of the well.  
As an application, we study the existence of approximately flat bands in moir\'e heterostructures such as strained two-dimensional honeycomb lattices in a model recently introduced by Timmel and Mele.
\end{abstract}
\maketitle

\section{Introduction}

In recent decades, scientists have mastered the creation of two-dimensional crystals, consisting of single atomic or molecular layers. When these crystals are stacked with slight offsets or rotations, they produce a large-scale interference pattern known as a {\it moiré pattern}. In such moir\'e materials, the electronic states align with the periodicity of the moiré pattern rather than that of the original crystal, exerting a profound influence on the material's electronic properties.

Twisted bilayer graphene (TBG), where two layers of graphene are stacked with a slight twist, serves as a prime example of this phenomenon. Graphene is a two-dimensional crystal formed by a single layer of carbon atoms arranged in a honeycomb lattice. When stacked at specific twist angles, known as {\it magic angles}, TBG exhibits remarkable properties including unconventional superconductivity and a distinctive flat electronic band structure at low energies. Tarnopolsky, Kruchkov, and Vishwanath \cite{magic} introduced a chiral continuum model for TBG which captures this fundamental nature of TBG's magic angles by showcasing perfectly flat Bloch-Floquet bands precisely at the magic angles. In \cite{becker2021spectral,becker2022mathematics} it was shown that as the twisting angle is very small, virtually every band close to zero energy is essentially flat for this model.

In this article, we study an analogue of the above-mentioned chiral model in one dimension, introduced by Timmel and Mele \cite{TM20}, where
a moir\'e-type structure appears in one dimension through the application of physical strain. 
While this model does not exhibit perfectly flat bands, we show that there exist 
approximate eigenvalues of infinite multiplicity 
in the limit of very large moir\'e cells. 
These approximate eigenvalues of infinite multiplicity are of the form  \eqref{eq:quasimode_constr} and correspond to {\it almost flat bands}, see \S\ref{ss:flatbands}.
The infinite multiplicity is obtained from an infinite number of disjoint wells which give rise to an infinite number of almost orthonormal quasimodes.
The model actually doesn't have any exact eigenvalues so the physical relevance of the quasimodes is their implication of approximately flat band spectrum.

The key ingredient in obtaining these approximate eigenvalues, and 
our main mathematical contribution, is a generalization of the Bohr-Sommerfeld quantization condition at potential wells to fairly general matrix-valued symbols (see Theorem \ref{thm:expansionsintro} and Corollary \ref{cor:periodicquasimodesnormalform}). This was previously only known for operators that are essentially scalar, and the semiclassical techniques using symplectic changes of coordinates to reduce the symbol to a harmonic oscillator do not generalize to systems. While our technique generalizes to higher dimensions, we restrict ourselves to the one-dimensional case in this work.

\begin{figure}
\includegraphics[scale=0.39]{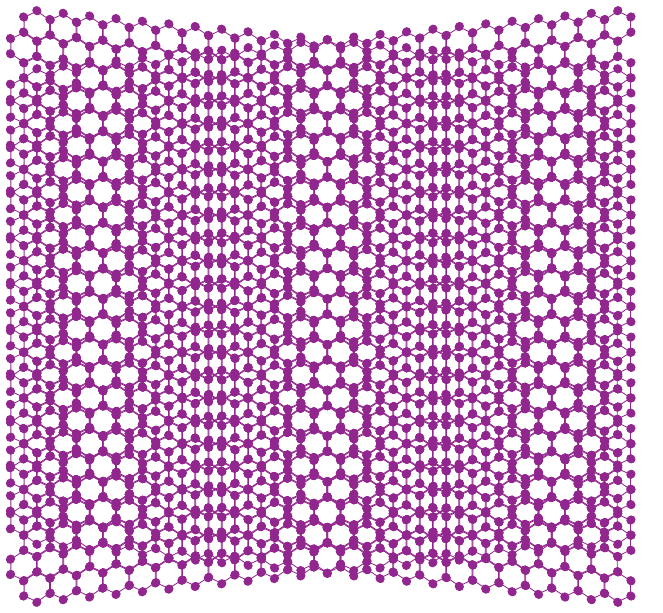} \quad
\includegraphics[scale=0.38]{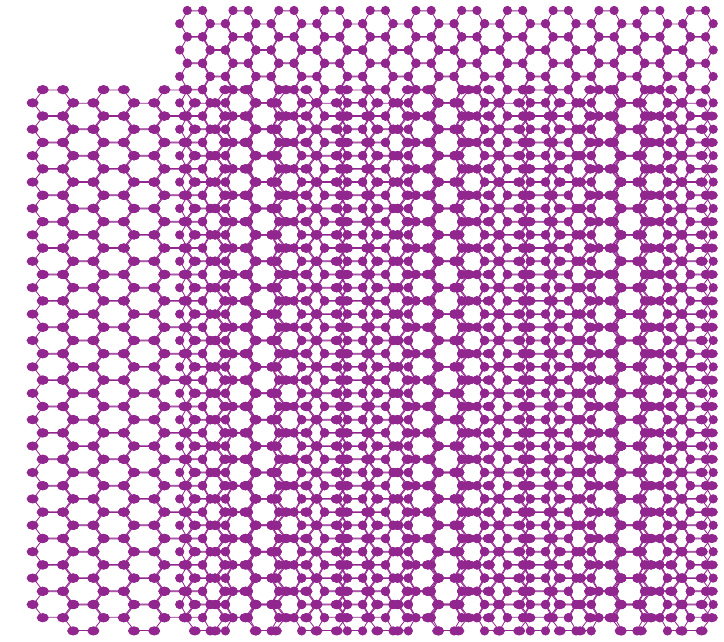} \quad
\includegraphics[scale=0.395]{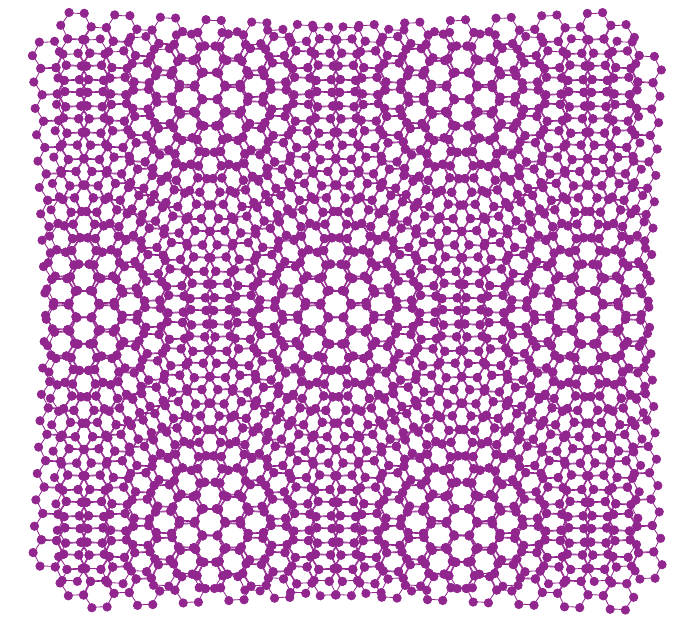}
\caption{Left: superposition of two honeycomb lattices subject to antisymmetric shear strain. Middle: superposition of two honeycomb lattices where one lattice is subject to uniaxial strain in horizontal direction. Both superpositions create 1D moir\'e patterns in the horizontal direction. For comparison, the right panel shows a 2D moir\'e pattern created by a superposition of twisted honeycomb lattices without strain.\label{fig:one}}
\end{figure}

The setting that we will be concerned with is illustrated in Figure \ref{fig:one}, where we describe the superposition of two honeycomb lattices under strain. The left superposition is subjected to antisymmetric shear strain, while the middle superposition is subjected to uniaxial strain in the horizontal direction. Contrary to the case of twisted bilayer graphene in which the moir\'e pattern is a two-dimensional structure (see the right panel), the left and middle moir\'e patterns in Figure \ref{fig:one} are essentially one-dimensional.

Apart from modeling graphene sheets under mechanical strain, the model we analyze in this work has also been considered for low-energy electron diffraction (LEED) studies in surface reconstructions of metals. More precisely, metals such as iridium, platinum, and gold are known to exhibit columns of honeycomb lattice structures on their surface with pits in between them. This phenomenon where the crystal structure of the metal is broken up on the surface is known as \emph{surface reconstruction} \cite{H12,LEED81}. In addition, the existence of one-dimensional flat bands in twisted one-dimensional Germanium selenide lattices has been recently discovered in \cite{GeSe20}.

In order to understand such emerging physical phenomena, we develop a new spectral analysis of operators with matrix-valued symbols exhibiting a potential well. For Schrödinger operators this has been studied by Barry Simon \cite{simon1983semiclassical} and has been generalized to pseudodifferential operators by Helffer-Robert \cite{Helffer-well} (see also \cite{DimSjo}) to symbols that, via a linear symplectic transformation, can be locally reduced to a principal symbol with non-degenerate minimum at $(x,\xi)=(0,0)$ such that
\begin{equation}\label{eq:princ_symb}
 p_0(x,\xi) = \frac{\lambda}{2}(\xi^2+x^2) + \mathcal{O}( (x,\xi)^3)\text{ with }\lambda>0.
 \end{equation}
The Bohr-Sommerfeld quantization condition near a potential well then allows for the following asymptotic spectral description, see also \cite{CdV2005bohr,ifa2018bohr}. Indeed, we can parametrize periodic orbits $p_0^{-1}(\tau)$ by the bicharacteristic flow $I\ni t \mapsto (x(t),\xi(t))$ at the energy level $\tau.$ We shall then write $\int_{p_0^{-1}(\tau)}f \ dt := \int_I f(x(t),\xi(t)) \ dt.$ The following theorem is discussed in \cite[Theorem 14.9]{DimSjo}, where we used the convention of \cite{CdV2005bohr,ifa2018bohr} instead for the characterization of the approximate eigenvalues.

\begin{thm}
\label{theo:DiSJ}
Let $p(x,\xi;h) \in S(1)$ with asymptotic expansion $p\sim\sum_{k\ge0}h^kp_k$ and assume that the principal symbol $p_0(x,\xi)\ge 0$ with $p_0(x,\xi)=0$ only at $(x,\xi)=0,$ as in \eqref{eq:princ_symb}, and also $\liminf_{\vert (x,\xi) \vert \to \infty} p(x,\xi)>0$. Let $\gamma_{\tau}:=p_0^{-1}(\tau)$. Then there exists a smooth function $F(\tau,h) \sim \sum_{n=0}^{\infty} h^n F_n(\tau)$, with $F_0(\tau)=\frac{1}{2\pi} \int_{\gamma_{\tau}} \xi dx=\tau/\lambda+\mathcal O(\tau^2)$, $F_1(\tau) =\frac{1}{2}-\int_{\gamma_{\tau}} p_1 dt $, and $$F_2(\tau)=\frac{1}{4\pi}\partial_{\tau} \int_{\gamma_{\tau}} \frac{1}{12} \operatorname{det}\begin{pmatrix} \partial_{x \xi} p_{0} &  \partial_{x x} p_{0}\\  -\partial_{\xi \xi} p_{0} & -\partial_{x \xi} p_{0} \end{pmatrix} -p_1^2  dt -\frac{1}{2\pi} \int_{\gamma_\tau} p_2  dt ,$$ such that for every fixed $\delta>0$, the eigenvalues of $P(h)= p^w(x,hD;h)$ in $(-\infty, h^{\delta})$ are defined by the implicit equation $F\left(\lambda_k(h),h\right) + \mathcal{O}(h^{\infty})=kh$ with $k \in \mathbb N_0$. 
\end{thm}

\begin{rmk}
The proof of Theorem \ref{theo:DiSJ} relies on what happens for the harmonic oscillator $p(x,\xi)= p_0(x,\xi)= \lambda (x^2+\xi^2)/2$. In this case, at energy level $\tau$, the bicharacteristic flow is $(x(t),\xi(t))= \sqrt{2\tau /\lambda}(\cos(2\pi t),\sin(2\pi t))$ for $t \in [0,1].$ Then $F_0(\tau) = \tau/\lambda$, $F_1(\tau)=1/2$ and $F_2(\tau)=0$, as the determinant in the theorem is a constant independent of $\tau$. Since $F_0$ is positive, we find for $\lambda_k(h) \ge 0$ that $$F_0(\lambda_n(h)) = \frac{\lambda_n(h)}{\lambda}+\mathcal O(h^3)= (n-1/2)h.$$
By the positivity constraint, eigenvalues start only at $n = 1$ and thus $\lambda_n(h)=\lambda (n-1/2)h+\mathcal O(h^3)$ for $n \ge 1.$ The vanishing of all higher order terms in case of the quantum harmonic oscillator is discussed in \cite{CdV2005bohr,ifa2018bohr}.
 \end{rmk}

In $F_1$, the value $\frac{1}{2}$ corresponds to the Maslov index, whereas $\int_{\gamma_{\tau}} p_1 \vert dt \vert$ is associated with Berry's phase, see also \cite{carmier2008berry}. 
Here the symbol class $S(1)$ is the set of all smooth functions with bounded derivatives of any order, and $p^w(x,hD)$ is the semiclassical Weyl quantization of $p$, see Section \ref{sec:pdo}. In the study of the Harper operator, this result has been generalized by Helffer-Sjöstrand \cite[Corollary 3.1.2]{helffer1990analyse} to matrix-valued symbols that can be block-diagonalized to a scalar symbol exhibiting a potential well and a possibly matrix-valued remainder that is spectrally gapped from the well. In fact, let $M \in S(1)$ be a self-adjoint matrix-valued symbol with $M(x,\xi)\in \mathbb C^{n \times n}$ and with one eigenvalue $\mu \in S(1)$, of algebraic multiplicity one, such that 
\begin{equation}
\label{eq:gap_condition}
 \inf_{(x,\xi) \in T^*\mathbb R} d(\Spec(M(x,\xi))\backslash \mu(x,\xi),\mu(x,\xi))>0.
 \end{equation}
Then there exists a unitary pseudodifferential operator $U(x,hD)$ such that 
\[ U^*(x,hD) M(x,hD) U(x,hD) = \operatorname{diag}( \tilde \mu(x,hD), \tilde M(x,hD) )+\mathcal O(h^{\infty}),\]
where $\tilde \mu$ has principal symbol $\mu.$

Our objective in this work is to study self-adjoint operators $P^w(x,hD)$ for which the gap-condition \eqref{eq:gap_condition} fails due to the existence of a degenerate potential well in the following sense:

\begin{defi}\label{def:degeneratewell}
Let $P_0(x,\xi)\in \CC^{2\times2}$ and assume that there is a constant $C\ge0$ such that $P_0(x,\xi)+C\operatorname{id}_{\CC^2\times\CC^2}$ is positive semi-definite for all $(x,\xi)$.
If 
\begin{equation}\label{eq:degeneratewell}
P_0(x,\xi)=(a(\xi-\xi_0)^2+b (x-x_0)^2)\id_2+\mathcal{O}(\lvert (x,\xi)-(x_0,\xi_0)\rvert^3)
\end{equation}
for some $a,b>0$, then we say that $P_0$ has a potential well at $(x_0,\xi_0)$. 
If there is a neighborhood of $(x_0,\xi_0)$ in which $P_0$ has only one distinct eigenvalue of constant multiplicity 2 then we say that the well is non-degenerate, otherwise it is said to be degenerate. We say that the system $P^w(x,hD)$ has a (degenerate or non-degenerate) well 
at $(x_0,\xi_0)$ 
if the principal symbol $P_0$ of $P^w(x,hD)$ has a (degenerate or non-degenerate) well at $(x_0,\xi_0)$.
\end{defi}

Note that if $P_0$ has a degenerate potential well at $(x_0,\xi_0)$ then the eigenvalues necessarily coalesce at $(x_0,\xi_0)$ since $P_0(x_0,\xi_0)$ has only the eigenvalue $\lambda(x_0,\xi_0)=0$ with algebraic multiplicity two. In particular, \eqref{eq:gap_condition} is not satisfied.
The main mathematical contribution of this article is the construction of quasimodes for such symbols in Theorem \ref{thm:expansionsintro} and Corollary \ref{cor:periodicquasimodesnormalform}, corresponding to operators on the real line and on the circle, respectively. Note that \eqref{eq:gap_condition} is also not satisfied when $P_0$ has a non-degenerate potential well since then $P_0$ has a double instead of simple eigenvalue there. Theorem \ref{thm:expansionsintro} and Corollary \ref{cor:periodicquasimodesnormalform} apply to this case too. However, when $P^w$ is self-adjoint with a non-degenerate potential well then $P_0$ is essentially scalar near the well so scalar methods can be used instead. For a degenerate well this is not possible since the multiplicity changes from one to two precisely at the bottom of the well; this is the reason we choose to call that case degenerate.

Before stating these results we briefly describe the setting.
In Theorem \ref{thm:expansionsintro}
we assume that the full symbol of $P^w$ has either an asymptotic expansion in $S(1)$ which is relevant for tight-binding operators, or in $S^2(T^*\mathbb R)$ which allows for second order differential operators including the square of a Dirac operator 
(see Section \ref{sec:pdo} for definitions of symbol classes). 
By multiplying the principal symbol $P_0$ of $P^w$ by a scalar if necessary we may assume that $a=1$ in \eqref{eq:degeneratewell}. We will also assume that $x_0=0$, but keep $\xi_0$ to illustrate its effects. 
Thus, $P\sim\sum h^jP_j$ where $P_0$ has a potential well at $(0,\xi_0)$. 
We may without loss of generality assume that the subleading symbol satisfies $P_1(0,\xi_0)=\diag(\mu_1,\mu_2)$ with $\mu_1,\mu_2\in\RR$. Indeed, since $P_1(x,\xi)$ is Hermitian when $P^w$ is self-adjoint, we may conjugate $P_1(0,\xi_0)$ with a unitary constant matrix to obtain this form, and this preserves \eqref{eq:degeneratewell} which proves the claim.
We note that if $P_0\in S(1)$ is positive semi-definite then the sharp G\aa rding inequality for systems (which for example follows from H\"ormander's Weyl calculus \cite[Theorem 6.8]{hormander1979weyl}, see Lemma \ref{lem:Gårding} below) gives $(P_0^w(x,hD)u,u)\ge-Ch\lVert u\rVert^2$ for $h$ small which implies that also 
\begin{equation*}
(P^w(x,hD)u,u)\ge-Ch\lVert u\rVert^2
\end{equation*}
 for $h$ small.
If $u$ is microlocally small in a neighborhood of the well at $(0,\xi_0)$, this will lead to a sufficiently good positive lower bound for our applications. 
We will be able to argue in a similar way when $P\in S^2(T^*\mathbb R)$ by assuming, in addition, that if $P\in S^2(T^*\mathbb R)$ then
\begin{equation}\label{eq:lowerboundS2}
(P^w(x,hD)u,u)\ge  (Vu,u)-Ch(u,u)
\end{equation}
where $V\in C^\infty(\RR;\RR^{2\times 2})$ is a matrix-valued function such that for some $c_1,c_2>0$, $V(x)+c_1\operatorname{id}_{\CC^{2\times 2}}$ is positive semi-definite for all $x$, and $V(x)=c_2x^2\operatorname{id}_{\CC^{2\times 2}}+\mathcal O(x^3)$ near $x=0$.
This assumption is tailored towards the low-energy model we study in \S\ref{ss:lowenergy}, and is natural for the square of a Dirac operator with a potential well, see Lemma \ref{lem:loweroperatorbound}.

We show that under these conditions, we can use the standard rescaling $y=h^{-\frac12}x$ to write $P^w$ as 
\begin{equation}\label{eq:normalformpullbackintro}
P^w(x,hD)=\gamma^*\circ\mathcal T\circ(\gamma^{-1})^*,\qquad \gamma(x)=h^{-\frac12}x,
\end{equation}
where $\mathcal T$ is an operator of the form
\begin{equation}\label{eq:normalformoperatorintro}
\mathcal Tv(y)= e^{i\xi_0 y/\sqrt h}T^w(y,D)(e^{-i\xi_0\bullet/\sqrt h}v)(y),
\end{equation}
see Proposition \ref{prop:normalform}.
Here $T^w=hT_0^w+h^{3/2}R_h^w$, with the leading symbol $T_0$ given as the direct sum of two harmonic oscillators 
\begin{equation*}
 T_0(y,\eta) = \operatorname{diag}(\eta^2 + \omega^2 y^2 + \mu_1, \eta^2 + \omega^2 y^2+ \mu_2),
\end{equation*}
where $\omega>0$ is determined from $a,b$ in Definition \ref{def:degeneratewell}. (Assuming $a=1$ we have $\omega=\sqrt b$.)
The symbol expansion of $R_h^w$ is described in Proposition \ref{prop:normalform}. 
These subleading symbols couple the two harmonic oscillators in a non-trivial way, preventing us from resorting to scalar methods. Instead, we use the fact that $R_h^w\varphi=\mathcal O(1)$ for eigenvectors $\varphi$ of $T_0^w$ to perform a perturbative analysis and obtain the following theorem, which is the main mathematical result of the paper. For the statement we recall the harmonic oscillator basis functions 
\begin{equation}\label{eq:harmoscbasis}
\begin{split}
 \varphi_{n,\omega}(y) &= \tfrac{1}{\sqrt{2^n n!}} \left(\tfrac{\omega}{\pi}\right)^{1/4}  H_n\left(\sqrt{\omega} y\right)e^{-\tfrac{\omega y^2}{2}}=:\phi_{n, \omega}(y) e^{-\tfrac{\omega y^2}{2}}  \quad n \in \mathbb N_0,\quad \omega>0,
 \end{split}
 \end{equation}
satisfying $(D_y^2 + \omega^2 y^2 ) \varphi_{n,\omega} = (2n+1) \omega \varphi_{n,\omega}$, and normalized in $L^2(\RR)$. 
Here $H_n$ is the $n$:th Hermite polynomial.

\begin{thm}\label{thm:expansionsintro}
Let $P \in C^\infty(T^*\RR)$ have asymptotic expansion $P\sim\sum_{k\ge0}^\infty h^kP_k$, where either $P_k\in S(1)$ for $k\ge0$, or $P_k\in S^{2-k}(T^*\mathbb R)$ for $k\ge0$, so that either $P\in S(1)$ or $P\in S^2(T^*\RR)$. Assume \eqref{eq:lowerboundS2} if $P\in S^2(T^*\mathbb R)$. Assume that $P_0$ has a potential well at $(0,\xi_0)$ in the sense of Definition \ref{def:degeneratewell} and that $P_1(0,\xi_0)=\diag(\mu_1,\mu_2)$ with $\mu_1,\mu_2\in\RR$.
Then $P^w$ can be written as in \eqref{eq:normalformpullbackintro}--\eqref{eq:normalformoperatorintro}, with $T^w$ in \eqref{eq:normalformoperatorintro} satisfying the following: For any $n\in\NN_0$ there is an $h_0>0$ together with quasimodes $v^{(j)}(n)$ and approximate eigenvalues $\lambda^{(j)}(n)$, $j=1,2$, such that for any $\ell\in\NN_0$ 
\begin{equation*}
(T^w(y,D)-\lambda^{(j)}(n))v^{(j)}(n,y)  = \mathcal{O}_\mathscr{S}(h^{\ell+\frac32}),\quad 0<h<h_0,
\end{equation*}
where $ v^{(j)}(n,y) =\sum_{i =0}^{2\ell}h^{i/2} v^{(j)}_i(n,y) e^{-\omega y^2/2}$ mod $\mathcal O_{\mathscr S}(h^{\ell+1/2})$ 
for some polynomials $v_{i}^{(j)}$, and where $\lambda^{(j)}(n) = h \sum_{i =0}^{2\ell} h^{i/2} \lambda_{i}^{(j)}(n)$ mod $\mathcal O(h^{\ell+3/2})$ with $\lambda^{(j)}_0(n) = (2n+1)\omega+\mu_j$. The leading order amplitudes are $v^{(1)}_0(n)=(\phi_{n,\omega},0)$ and $v_0^{(2)}(n)=(0,\phi_{n,\omega})$. 
\end{thm}

The proof of Theorem \ref{thm:expansionsintro} takes up the bulk of Section \ref{sec:wellsnormalform}.
We remark that the presence of $\mu_1$ and $\mu_2$ in $T_0$ has a nontrivial influence on the available methods for constructing quasimodes. (Note that resonance occurs when $\mu_1-\mu_2\in2\ZZ\omega$, see the discussion preceding \eqref{eq:eigenvalueordering} below.) To illustrate this, we include an explicit WKB-type construction of approximate eigenvalues and quasimodes in Proposition \ref{thm:quasimodes2} as a special case of Theorem \ref{thm:expansionsintro}, which relies on one of the following constraints being fulfilled: either $\mu_1-\mu_2 \notin 2\ZZ\omega$, or $\mu_1-\mu_2\in4\ZZ\omega$ but the expansion of $R_h$ satisfies an alternating block symmetry assumption (diagonal and off-diagonal, alternating). 
In particular, for the construction to work it is crucial that $\mu_1-\mu_2\notin (4\ZZ+2)\omega$, and as we will see below, this assumption is violated in the models of one-dimensional strained moir\'e lattices introduced by Timmel and Mele \cite{TM20}.

To treat the general case, we design a new phase space version of a quasimode construction used by Barry Simon \cite{simon1983semiclassical}. This includes a phase space version of the IMS localization formula in Lemma \ref{lem:IMS} which may be of independent interest. To isolate the spectral contribution of the degenerate potential well at $(0,\xi_0)$, we use the notion of a {\it massive Weyl operator} by adding to $P^w(x,hD)$ the operator $(1-\chi^w(x,hD))\operatorname{id}_{\CC^{2\times 2}}$ where $\chi$ is a symbol having small support such that $\chi\equiv1$ near the well, see \eqref{eq:massive}. Since the quasimodes have semiclassical wavefront set confined to $\{(0,\xi_0)\}$, the difference between $P^w(x,hD)$ and its massive counterpart acting on the quasimodes is $\mathcal O(h^\infty)$ as $h\to0$.

If the symbol $P(x,\xi)\in C^\infty(T^*\RR)$ in Theorem \ref{thm:expansionsintro} in addition is 1-periodic in $x$, then $P^w$ preserves periodicity in the sense that if $u(x)\in C^\infty(\RR)$ is 1-periodic then so is $P^w(x,hD)u(x)$.
Moreover, $P$ can be identified with a symbol on $T^*\TT$, where $\TT:=\RR/\ZZ$ is the one-dimensional torus.
By identifying 1-periodic functions on $\RR$ with functions on $\TT$, pseudodifferential operators on $\TT$ can in this way be identified with operators having symbols that are 1-periodic in $x$, acting on 1-periodic functions, see \S\ref{subsec:pseudo} for details. As a consequence of Theorem \ref{thm:expansionsintro} we therefore obtain the following quasimode construction for $P^w(x,hD)$ on $L^2(\mathbb T;\CC^2)$ near a potential well.
We give the proof in Subsection \ref{ss:quasimodes}.

\begin{corr}\label{cor:periodicquasimodesnormalform}
Let $P$ satisfy the assumptions in Theorem \ref{thm:expansionsintro}. Assume in addition that $P(x,\xi)$ is 1-periodic in $x$, and regard $P^w(x,hD)$ as a pseudodifferential operator on $\TT$.
Let $\{v^{(j)}(n)\}_{n\in\NN_0}$ and $\{\lambda^{(j)}(n)\}_{n\in\NN_0}$, $j=1,2$, be given by Theorem \ref{thm:expansionsintro}. 
Then the family of quasimodes $\{u^{(j)}(n)\}_{n\in\NN_0}$ in $C^\infty(\mathbb T;\CC^2)$ given by 
$$
u^{(j)}(n,x)=h^{-\frac14}\sum_{k\in\ZZ} e^{i\xi_0 (x-k) /h}v^{(j)}(n,(x-k)h^{-\frac12})
$$
satisfies
$\WF_h(u^{(j)})=\{(0,\xi_0)\}\subset T^*\mathbb T$, $\lVert u^{(j)}\rVert_{L^2(\mathbb T;\CC^2)}=1+\mathcal{O}(h^\frac12)$ and
\begin{equation}
\label{eq:quasimode_constr}
(P^w(x,hD)-\lambda^{(j)}(n))u^{(j)}(n)=\mathcal{O}_{L^2(\mathbb T;\CC^2)}(h^{\ell+\frac32})
\end{equation}
for $0<h<h_0(n)$ and all $\ell\in\NN_0$. 
\end{corr}

In the second part of the paper we then apply Corollary \ref{cor:periodicquasimodesnormalform} to the models of Timmel and Mele \cite{TM20} of one-dimensional moir\'e structures discussed above. To describe our results in more detail we first need to introduce the models. 

\subsection{Model Hamiltonian}
We start by introducing the tight-binding (i.e., discrete) model of one-dimensional moir\'e structures. 
Here, an effectively one-dimensional moir\'e pattern (visible along the horizontal direction in Figure \ref{fig:one}) has been formed due to periodic strain-modulation. Using periodicity in the orthogonal direction, Floquet theory is used to obtain a family of Hamiltonians, depending on a quasimomentum $k_\perp\in\RR$  perpendicular to the moir\'e direction, that approximates the dynamics of strained bilayer graphene in the direction of the moir\'e pattern, and takes the following form: 
Let $\psi=(\psi_n)_{n=-\infty}^\infty$ be a vector in $\ell^2(\ZZ;\CC^4)$. The Harper model for strained bilayer graphene \cite{TM20} is defined as the action $H_{\operatorname{TB}}(w)\psi=((H_{\operatorname{TB}}(w)\psi)_n)_{n=-\infty}^\infty$ where
\begin{equation}\label{eq:Harper}
 (H_{\operatorname{TB}}(w)\psi)_n:=\mathbf t(k_{\perp}) \psi_{n+1} +\mathbf t(k_{\perp}) \psi_{n-1}  + (\mathbf t_0 + V_w(n/L))\psi_n.
\end{equation}
Here, $L$ is the length of a unit period of the moir\'e pattern (i.e., the length of the fundamental cell of the pattern's one-dimensional lattice structure); we call $L$ the moir\'e length. $L$ is related to the strength of the strain, and $L\to\infty$ as the strength of the strain tends to zero.

Denote the Pauli matrices by $\sigma_i$ for $i \in \{1,2,3\}$.
Then the kinetic part is the discrete Dirac operator which, for $\gamma_{15} = \operatorname{diag}(\sigma_1,\sigma_1)$ and $\gamma_{25}= \operatorname{diag}(\sigma_2,\sigma_2)$,  is defined as 
 \[ (D_{\operatorname{kin}}(k_{\perp})\psi)_n = \mathbf t(k_{\perp}) \psi_{n+1} + \mathbf t(k_{\perp}) \psi_{n-1} +\mathbf t_0 \psi_n,\]
 where 
\begin{equation}
\mathbf t(k_{\perp}) =  (\cos(2\pi k_{\perp}) \gamma_{15}+ \sin(2\pi k_{\perp}) \gamma_{25}),\qquad\qquad \mathbf t_0 = \gamma_{15}.
\end{equation} 
This parameter configuration corresponds to the armchair configuration in \cite{TM20} resulting from antisymmetric shear strain (shown in the left panel of Figure \ref{fig:one}). Uniaxial strain leads to a similar model.
The honeycomb lattice consists of two types of atoms per fundamental cell, denoted A and B, respectively. The full potential $V_w$ is defined in terms of an anti-chiral potential (ac), describing the interaction between atoms A/A, B/B of the honeycomb lattice, and a chiral potential (c), describing the interaction between atoms A/B, B/A of the honeycomb lattice, of the form
 \[ V_{\operatorname{ac}}= U_{\operatorname{ac}} \begin{pmatrix} 0 & \operatorname{id} \\ \operatorname{id} & 0 \end{pmatrix}, \ V_\ch = \begin{pmatrix} 0 & W \\ W^* & 0 \end{pmatrix} \text{ with } W= \begin{pmatrix} 0 & U_\ch^- \\ U_\ch^+ & 0 \end{pmatrix} ,\]
 where $U_\mathrm{ac}(x)=1+2\cos(2\pi x)$ and $U^\pm_\mathrm{c}(x) = 1- \cos(2\pi x) \pm \sqrt 3 \sin(2\pi x)$.  
Then
 \[ V_w(x) = w_0 V_{\operatorname{ac}}(x) + w_1 V_{\operatorname{c}}(x),\quad w=(w_0,w_1)\in\RR_+^2. \]
 Interaction between atoms A/A, B/B where the two honeycomb lattices are aligned is called AA stacking, while interaction between atoms A/B, B/A are called AB and BA stacking, respectively. The stacking configurations are illustrated in Figure \ref{fig:config} for the two types of 1D moiré patterns shown in Figure \ref{fig:one}. 
 \begin{figure}
\includegraphics[width=0.98\textwidth]{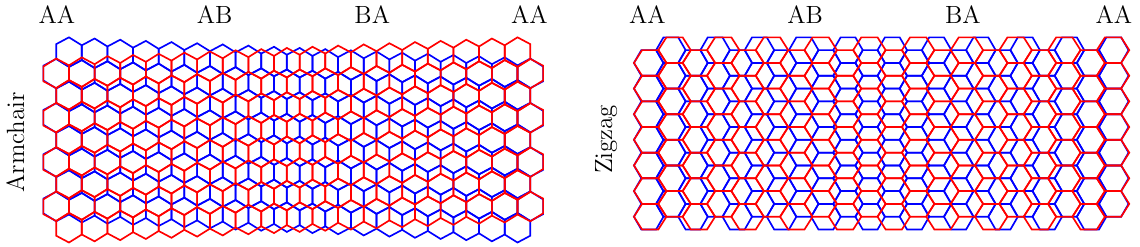}
\caption{\label{fig:config} Different stacking configurations (AA, AB, BA) for superpositions of honeycomb lattices under strain. On the left, both the red and the blue lattice is subject to antisymmetric shear strain, while on the right, the red lattice is subject to uniaxial strain in the horizontal direction.}
\end{figure}
In their paper on strained lattices, Timmel and Mele \cite{TM20} consider the full model and assume that $w_0=w_1$. We shall mainly be concerned with the chiral limit, obtained by setting $w_0=0$. While this may seem like a drastic simplification, the resulting model reproduces the spectrum close to zero of the full model quite well, see \S\ref{ss:flatbands} below and in particular Figure \ref{fig:comparison}.

To use the general framework developed for operators with degenerate potential wells we note that the discrete operator \eqref{eq:Harper} is unitarily equivalent to the pseudodifferential operator 
\begin{equation}\label{eq:HarperPDO}
 H_{\operatorname{\Psi DO}}(w)u(x):=(2\mathbf t(k_{\perp}) \cos(2\pi x)+\mathbf t_0  +V_w(hD_x))u(x)
 \end{equation}
acting on $L^2(\mathbb T;\CC^4)$, $\mathbb T=\RR/\ZZ$,
see Lemma \ref{lem:PsiDO}. Here, the semiclassical parameter is defined in terms of the moir\'e length via $h=(2\pi L)^{-1}$, i.e., we are concerned with the limit of large moir\'e lengths $L \gg 1$. In the chiral limit ($w_0=0$), conjugating $H_{\operatorname{\Psi DO}}(w)$ by a unitary matrix leads to a system on off-diagonal block form, thus effectively reducing the spectral analysis to a $2\times2$ system, see Lemma \ref{lem:reduction0Harper}. We show that the degenerate wells only appear for quasimomenta $k_\perp=0$ or $k_\perp=\frac12$. After describing the precise location of the wells in Proposition \ref{prop:degwellHarper}, we then use Corollary \ref{cor:periodicquasimodesnormalform} to obtain approximate eigenvalues and quasimodes to any order. We can do this for each of the degenerate wells, which appear periodically with period 1 in the fiber direction of phase space. Since the approximate eigenvalues are independent of the choice of well, this gives approximate eigenvalues of infinite multiplicity:
 
 \begin{thm}\label{cor:periodicquasimodes2}
Let $ H_{\operatorname{\Psi DO}}(w)$ be given by \eqref{eq:HarperPDO} and consider the chiral limit $w=(0,w_1)$ and quasimomentum $k_\perp=0$ or $k_\perp=\frac12$. For each $\xi_0=\pm \frac13(\frac12)^{2k_\perp} +n_0$, $n_0\in\ZZ$, and $j=1,2$ there are quasimodes $\{\psi_\pm^{(\xi_0,j)}(n)\}_{n\in\NN_0}\subset C^\infty(\mathbb T;\CC^4)$, normalized in $L^2(\mathbb T;\CC^4)$ and with $\WF_h(\psi_\pm^{(\xi_0,j)}(n))=\{(\xi_0,0)\}$, together with approximate eigenvalues $\{\pm\sqrt{\lambda^{(\xi_0,j)}(n)}\}_{n\in\NN_0}$ that are independent of $n_0\in\ZZ$, such that if $0<h<h_0(n)$ then
$$
\bigg( H_{\operatorname{\Psi DO}}(0,w_1)\mp\sqrt{\lambda^{(\xi_0,j)}(n)}\bigg) \psi_\pm^{(\xi_0,j)}(n)=\mathcal{O}(h^{\ell+1})
$$
in $L^2(\mathbb T;\CC^4)$ for any $\ell\in\NN_0$, where $\lambda^{(\xi_0,j)}(n) = h \sum_{i = 0}^{2\ell} h^{i/2} \lambda_{i}^{(\xi_0,j)}(n)$ mod $\mathcal O(h^{\ell+3/2})$ with $h^{(\xi_0,j)}_0(n)=(2n+1\pm(-1)^{2k_\perp+j-1})12\pi^2 w_1$. 
In particular, near zero energy $H_{\operatorname{\Psi DO}}(0,w_1)$ has approximate eigenvalues of infinite multiplicity given by
$$
\big\{\pm  \sqrt{24\pi^2w_1hn}+\mathcal{O}(h)\big\}_{n\in\NN_0}.
$$
\end{thm}

\subsection{Effective Hamiltonian}\label{ss:effective}
We shall also consider an effective low-energy model of \eqref{eq:Harper} introduced in \cite{TM20} for a moir\'e superlattice with antisymmetric shear strain. After a rescaling $x/L\mapsto x$, where $L$ is the moir\'e length, the model is described by the semiclassical operator  
\begin{equation}\label{eq:LEintro}
H^w(x,hD_x) = \begin{pmatrix} 0 & hD_x -i k_{\perp} & w_0 U(x) & w_1  U^-(x) \\ hD_x +i k_{\perp} & 0 & w_1 U^+(x) & w_0U(x)  \\ w_0 U(x) & w_1 U^+(x) & 0 & hD_x -i k_{\perp}   \\  w_1 U^-(x) & w_0 U(x) & h D_x +i k_{\perp} & 0 \end{pmatrix},
\end{equation}
with semiclassical parameter $h=1/L$, acting on $L^2(\RR;\CC^4)$. Here $k_\perp$ is the quasimomentum in the orthogonal periodic direction, and $U(x)=U_\mathrm{ac}(x)=1+2\cos(2\pi x)$ and $U^{\pm}(x)=U^\pm_\mathrm{c}(x) = 1- \cos(2\pi x) \pm \sqrt 3 \sin(2\pi x)$ as before. The kinetic differential operator is essentially a linearization in $\xi$ of a symbol associated with the discrete model \eqref{eq:Harper}. 
We denote the symbol of the chiral Hamiltonian, when $w_0=0$, by $H_\ch$ and the symbol of the anti-chiral Hamiltonian, when $w_1=0$, by $H_{\operatorname{ac}}$, respectively.

Since the potential in $H$ is 1-periodic in $x$, we can use the standard Bloch-Floquet transform to equivalently study the spectrum of $H^w(k_x)=H^w(x,hD;k_x)$ on $L^2(\mathbb T;\CC^4)$, where
\begin{equation}\label{eq:tm20}
  H^w(k_x) = \begin{pmatrix} 0 & hD_x+k_x -i k_{\perp} & w_0 U(x) & w_1  U^-(x) \\ hD_x+k_x +i k_{\perp} & 0 & w_1 U^+(x) & w_0U(x)  \\ w_0 U(x) & w_1 U^+(x) & 0 & hD_x+k_x -i k_{\perp}   \\  w_1 U^-(x) & w_0 U(x) & hD_x+k_x +i k_{\perp} & 0 \end{pmatrix}
\end{equation}
and
\begin{equation}\label{eq:BFunion}
\Spec(H^w) = \bigcup_{k_x \in [0,2\pi h]} \Spec(H^w(k_x)).
\end{equation}
Then both $H_{\Psi\mathrm{DO}}$ and $H^w(k_x)$ act on $L^2(\mathbb T;\CC^4)$ which allows for a unified treatment of both models. Note that $k_x=\mathcal O(h)$ so it does not contribute to the principal symbol of $H^w(k_x)$.

Similar to $H_{\operatorname{\Psi DO}}(w)$ we find in the chiral limit $w_0=0$ that conjugating $H_\ch^w(k_x)$ by a unitary matrix leads to a system on off-diagonal block form, see Lemma \ref{lem:reduction0}. We locate the degenerate potential wells in Proposition \ref{prop:degwellH}, which only appear for quasimomentum $k_\perp=0$, and apply Corollary \ref{cor:periodicquasimodesnormalform} to obtain approximate eigenvalues and quasimodes to any order. The proof (given at the end of \S\ref{ss:lowenergy}) shows that the obtained approximate eigenvalues of $H_\ch^w(k_x)$ are independent of $k_x$,
which by \eqref{eq:BFunion} leads to approximate eigenvalues of infinite multiplicity for $H_\ch^w$:

\begin{thm}\label{cor:periodicquasimodes}
Let $H^w$ and $H^w(k_x)$ be given by \eqref{eq:LEintro} and \eqref{eq:tm20}, respectively, and consider the chiral limit $w=(0,w_1)$ and quasimomentum $k_\perp=0$.
For each $j=1,2$ there are quasimodes $\{\psi_\pm^{(j)}(n)\}_{n\in\NN_0}\subset C^\infty(\mathbb T;\CC^4)$, normalized in $L^2(\mathbb T;\CC^4)$ and with $\WF_h(\psi_\pm^{(j)}(n))=\{(0,0)\}$, together with approximate eigenvalues $\{\pm\sqrt{\lambda^{(j)}(n)}\}_{n\in\NN_0}$, such that if $0<h<h_0(n)$ then
$$
\bigg( H_\ch^w(k_x)\mp\sqrt{\lambda^{(j)}(n)}\bigg) \psi_\pm^{(j)}(n)=\mathcal{O}(h^{\ell+1})
$$
in $L^2(\mathbb T;\CC^4)$ for any $\ell\in\NN_0$, where $\lambda^{(j)}(n) = h \sum_{i = 0}^{2\ell} h^{i/2} \lambda_{i}^{(j)}(n)$ mod $\mathcal O(h^{\ell+3/2})$ with $h^{(j)}_0(n)=((2n+1)+(-1)^j)2\pi\sqrt 3 w_1$. In particular, for quasimomentum $k_\perp=0$, the chiral limit of $H^w$ in \eqref{eq:LEintro} has approximate eigenvalues of infinite multiplicity given by
$$
\Big\{\pm  \sqrt{4\pi\sqrt 3w_1 hn}+\mathcal{O}(h)\Big\}_{n\in\NN_0}.
$$
\end{thm}

\subsection{Almost flat bands}\label{ss:flatbands}

A band of $H^w$ (as the word has been used above) refers to an eigenvalue of $H^w(k_x)$ as a function of $k_x$. This is justified by \eqref{eq:BFunion}.
It is a classical result that a flat band corresponds to an eigenvalue of infinite multiplicity \cite[Corollary 6.11]{kuchment2016overview}. The converse can be deduced in the same way if the operator is self-adjoint and depends analytically on the quasimomentum $k$ (Rellich's theorem). Existence of approximate eigenvalues for the chiral limit $H_\ch^w(k_x)$, as described in the previous subsection, then results in the bands close to zero energy being almost flat. This is illustrated in the left panel of Figure \ref{fig:bands}. For comparison, the right panel shows the bands of the anti-chiral limit, none of which are almost flat (see Appendix \ref{sec:antichiral}).

\begin{figure}
{\begin{tikzpicture}
\node at (-2,0) {\includegraphics[width=7.6cm]{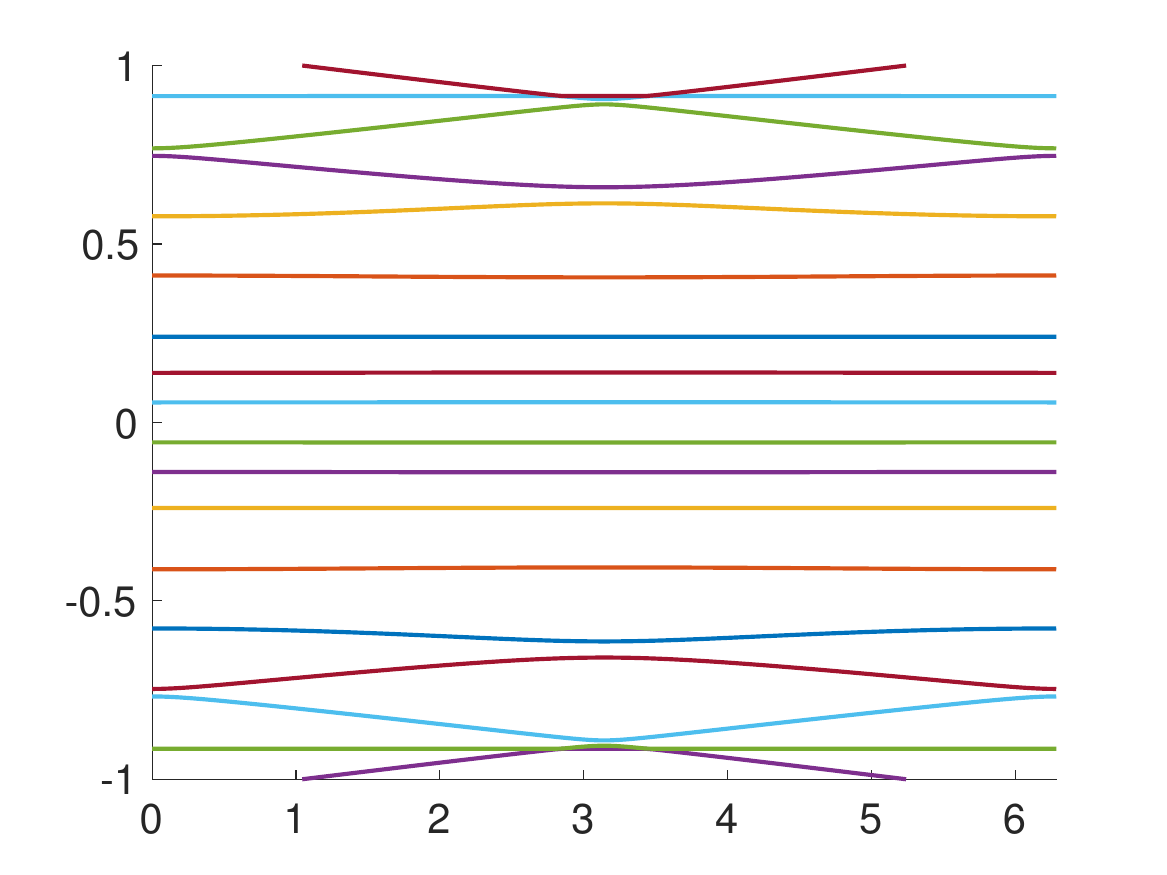}};
\node at (7.6-2,0) {\includegraphics[width=7.6cm]{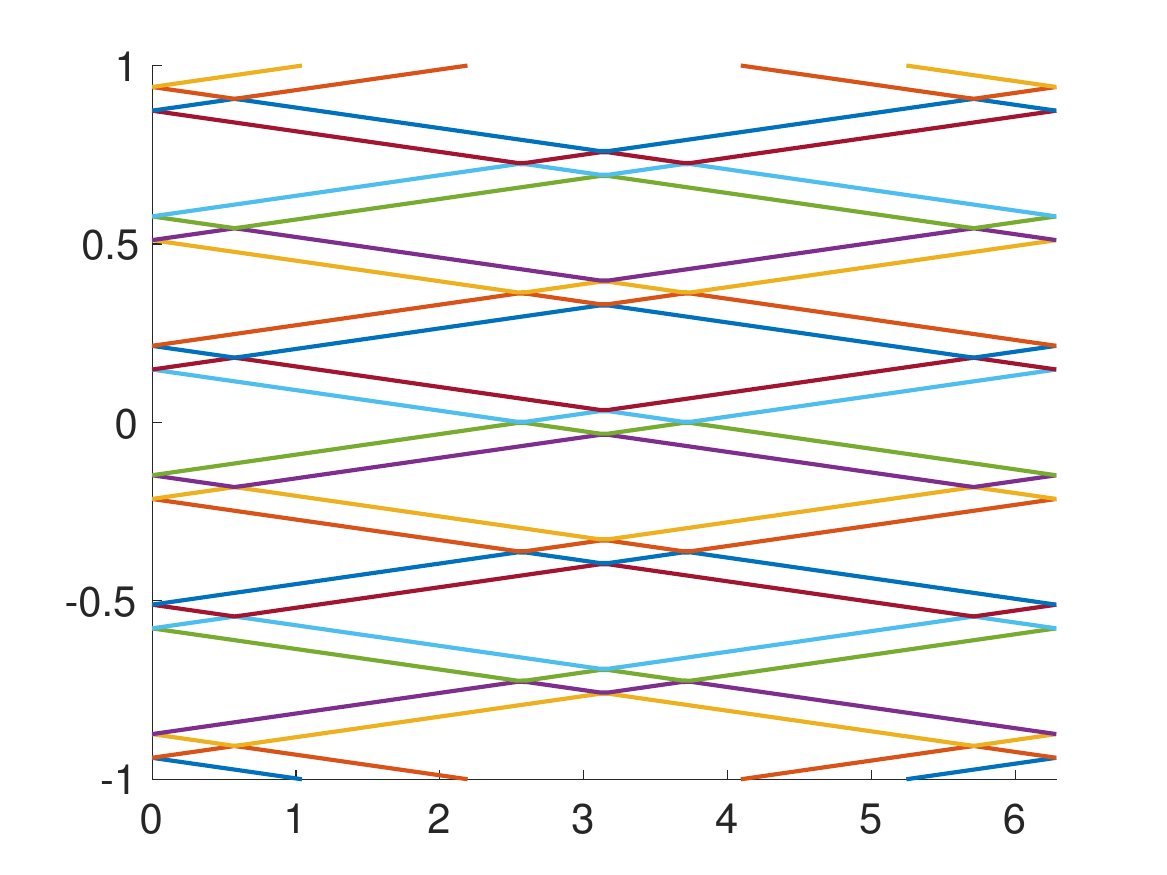}};
\node at (-2,-3.0) {$k_x/h$};
\node at (-2,3.0) {Chiral model};
\node at (-4-1.7,0) {$E$};
\node at (-2+7.6,3.0) {Anti-chiral model};
\node at (-4-1.7+7.6,0) {$E$};
\node at (7.6-2,-3.0) {$k_x/h$};
\end{tikzpicture}}
 \caption{Bands closest to zero energy for the low-energy model \eqref{eq:tm20} in the chiral limit (left) $(w_0,w_1)=(0,1)$ and anti-chiral limit (right) $(w_0,w_1)=(1,0)$ with $h=1/10$. The figure plots the quasimomentum $k_x/h$ on the $x$-axis and the respective eigenvalues of $H^w(k_x)$ on the $y$-axis.  The chiral model exhibits almost flat bands.\label{fig:bands}}
 \end{figure}
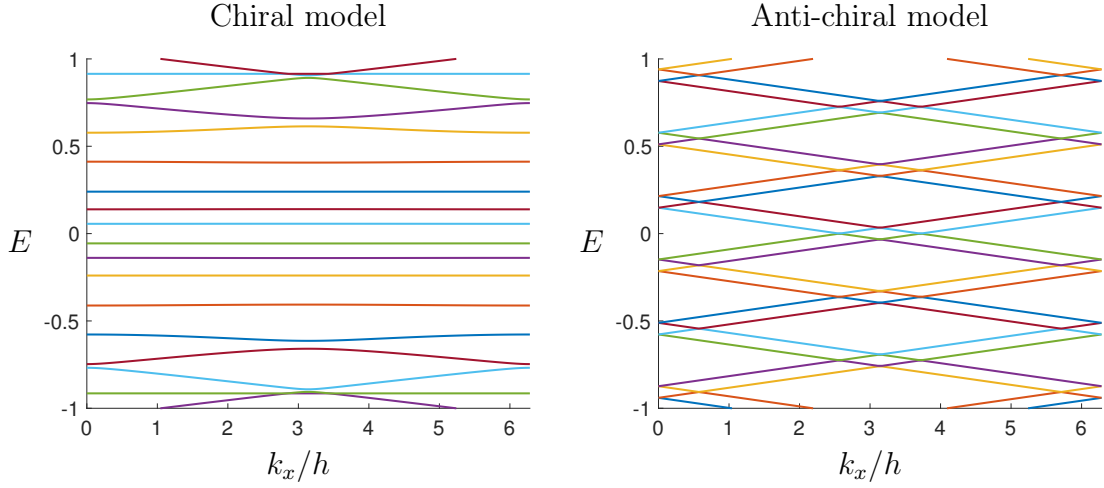

The corresponding notion of bands also exists for the discrete Hamiltonian \eqref{eq:Harper}. To see this, note that there is a version of Bloch-Floquet theory also for this model. Indeed, let $L=q/p$ for positive integers $p$ and $q$.  Then $H_{\operatorname{TB}}$ commutes with translations $\tau_q\psi_n=\psi_{n-q}$. The Bloch transform $\ell^2(\ZZ; \mathbb C^4) \ni (\psi_n)_{n\in\ZZ}\to (k_x \mapsto (\phi_n(k_x))_{n \in \ZZ/q\ZZ} \in L^2(\mathbb R/(2\pi \ZZ);\mathbb C^4) \otimes \mathbb C^{q} $ is then defined for $k_x \in \RR/(2\pi /q)\ZZ$ by
\[ \phi_n(k_x) := \sum_{m \in \ZZ} \psi_{n-qm}e^{iqmk_x} \text{ with } n \in \ZZ/q\ZZ\]
and the Floquet transformed Hamiltonian $H_{\text{TB}}(k_x)$ takes the form
\[ H_{\text{TB}}(k_x)= \mathbf t(k_{\perp}) \otimes (J(k_x)+J(k_x)^*) + (\mathbf t_0  \otimes I_q +\mathcal V_w ),\]
where 
$$
J(k_x)=\begin{pmatrix} 0 & 0 & 0& e^{iqk_x} \\
1& 0 & 0 &0 \\
0& \ddots& 0 & 0 \\
0 & 0& 1 & 0 \end{pmatrix} \in \mathbb C^{q \times q}
$$
and the potential 
\[ \mathcal V_w=\begin{pmatrix} 0 & w_0  I_2 \\ w_0 I_2 & 0 \end{pmatrix} \otimes U_q+w_1 \begin{pmatrix} 0 & \tau^t \\ \tau & 0 \end{pmatrix} \otimes U_q^{+}+w_1 \begin{pmatrix} 0 & \tau \\ \tau^t & 0 \end{pmatrix} \otimes U_q^{-}\]
with matrices $U^{\pm}_q=\operatorname{diag}(U^{\pm}(jp/q))_{1\le j \le q}$, $U_{q}=\operatorname{diag}(U(j p/q))_{1\le j \le q}$, and $\tau=\begin{pmatrix} 0 & 1 \\ 0 & 0 \end{pmatrix}$. In particular, 
\begin{equation*}
\Spec(H_{\text{TB}}) = \bigcup_{k_x \in [0,2\pi/q ]} \Spec(H_{\text{TB}}(k_x)).
\end{equation*}
In the left panel of Figure \ref{fig:bands2} the corresponding bands close to zero energy are shown, demonstrating that they are indeed almost flat. The right panel shows the bands of the anti-chiral limit, discussed in Appendix \ref{sec:antichiral}.

\begin{figure}
{\begin{tikzpicture}
\node at (-2,0) {\includegraphics[width=7.6cm]{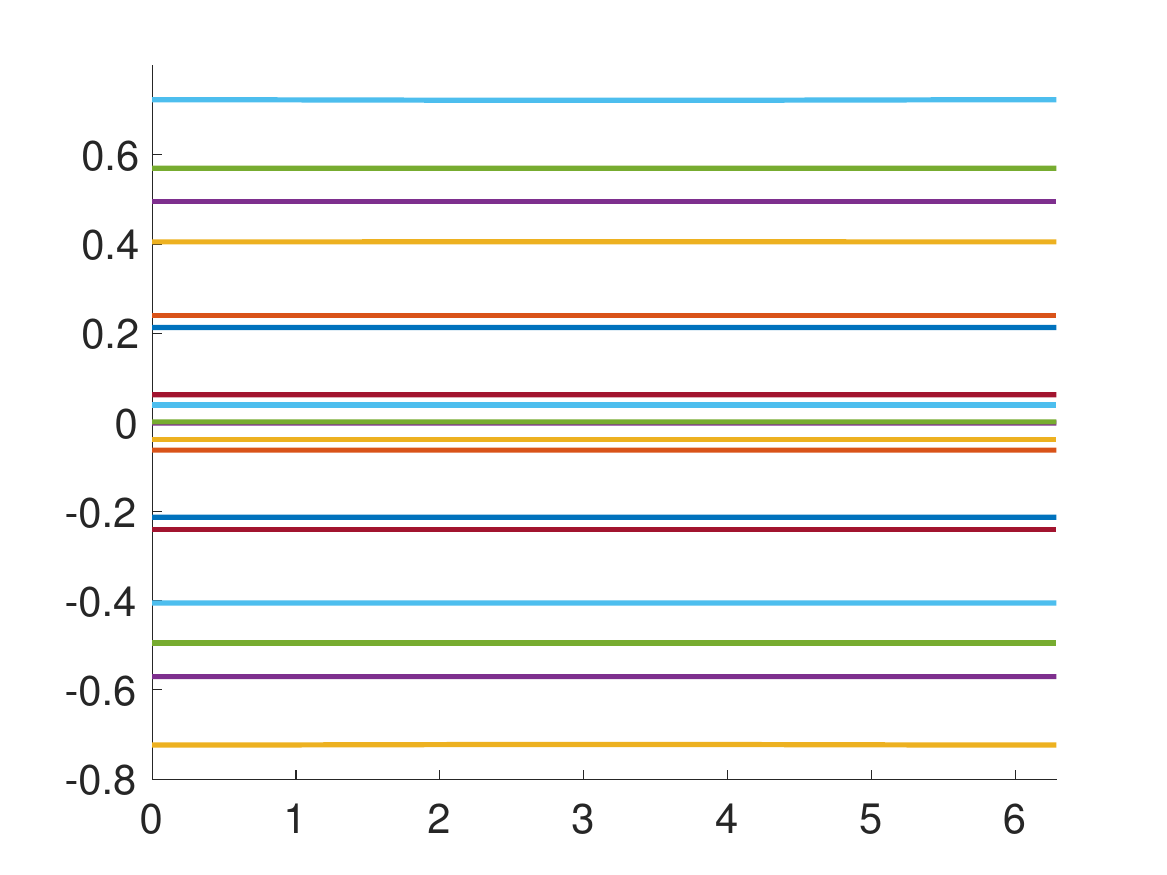}};
\node at (7.6-2,0) {\includegraphics[width=7.6cm]{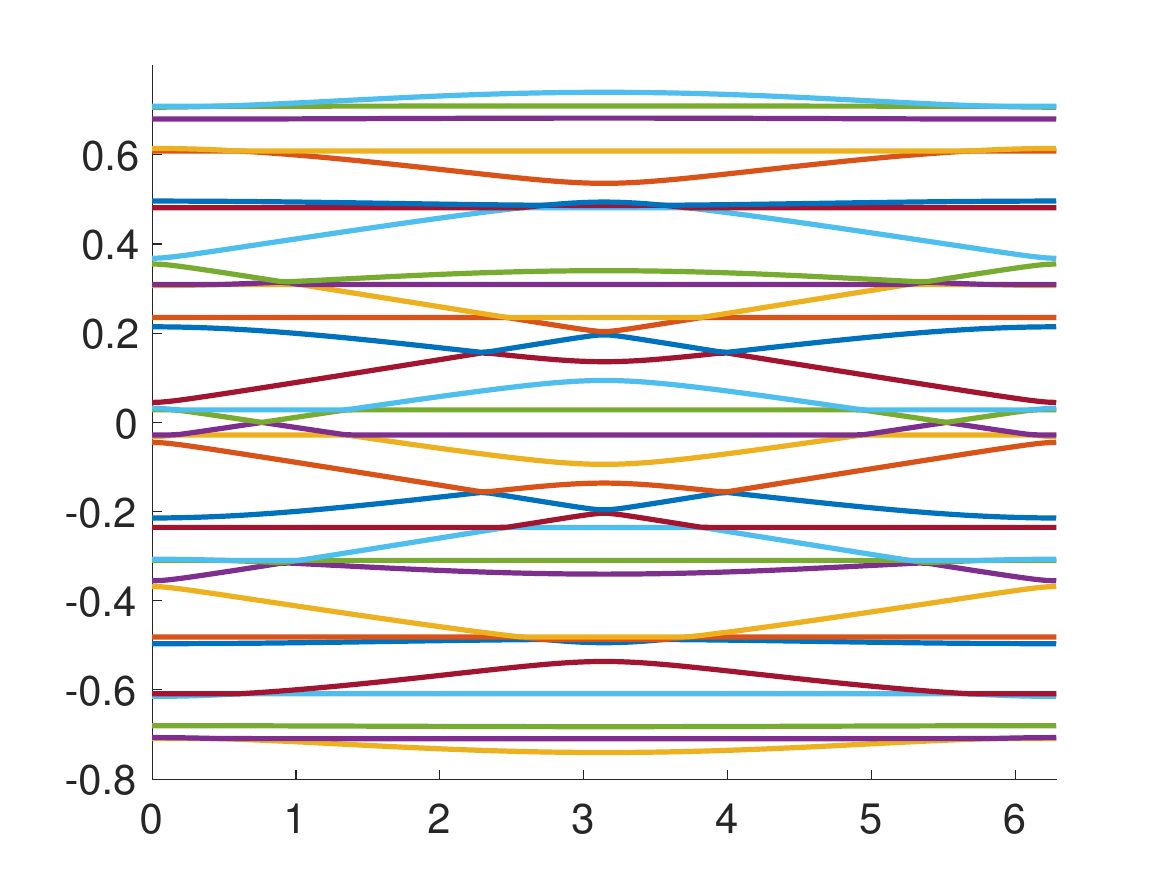}};
\node at (-2,-3.0) {$k_x\cdot q$};
\node at (-2,3.0) {Chiral model};
\node at (-4-1.7,0) {$E$};
\node at (-2+7.6,3.0) {Anti-chiral model};
\node at (-4-1.7+7.6,0) {$E$};
\node at (7.6-2,-3.0) {$k_x\cdot q$};
\end{tikzpicture}}
 \caption{Bands close to zero energy for the discrete model \eqref{eq:Harper} in the chiral limit (left) $(w_0,w_1)=(0,1)$ and anti-chiral limit (right) $(w_0,w_1)=(1,0)$ with $q=30,p=1$ and thus $L=30$. The figure plots the quasimomentum $k_x\cdot q$ on the $x$-axis and the respective eigenvalues of $H_{\operatorname{TB}}(k_x)$ on the $y$-axis. As in Figure \ref{fig:bands}, the chiral model exhibits almost flat bands.\label{fig:bands2}}
 \end{figure}
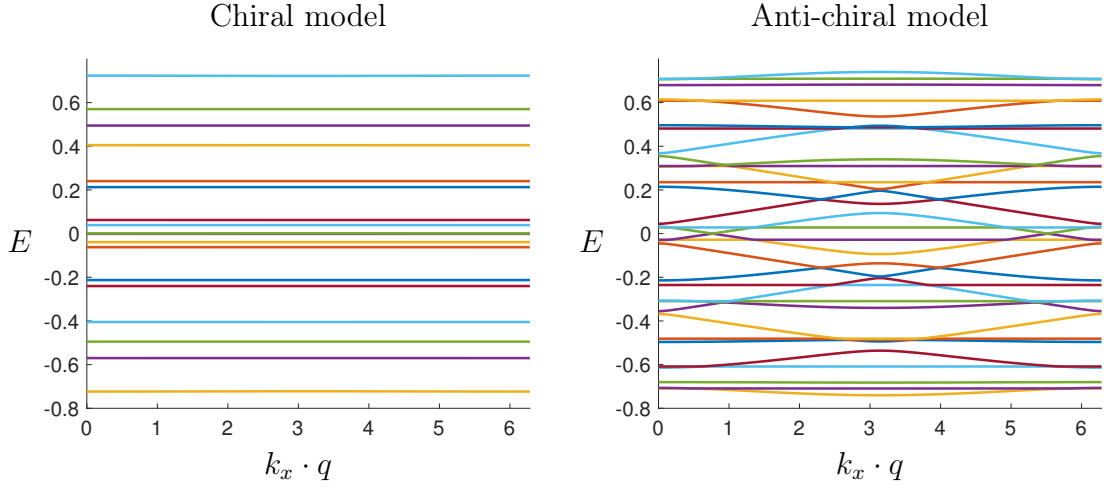

 The approximately flat bands exhibited by the chiral limits of $H^w$ and $H_{\mathrm{TB}}$ (cf.~the left panel of Figures \ref{fig:bands} and \ref{fig:bands2}) remain nearly flat for the full models for physically reasonable choices of hopping parameters $w=(w_0,w_1)$. For twisted bilayer graphene, it has been experimentally verified that the hopping strength ratio is $w_0/w_1\approx 0.7$-$0.8$ for small twisting angles. Using a similar ratio, Figure \ref{fig:comparison} displays the bands of the full models, showing that the spectrum close to zero energy is accurately reproduced by the corresponding chiral limits.

\begin{figure}
{\begin{tikzpicture}
\node at (-2,0) {\includegraphics[width=7.6cm]{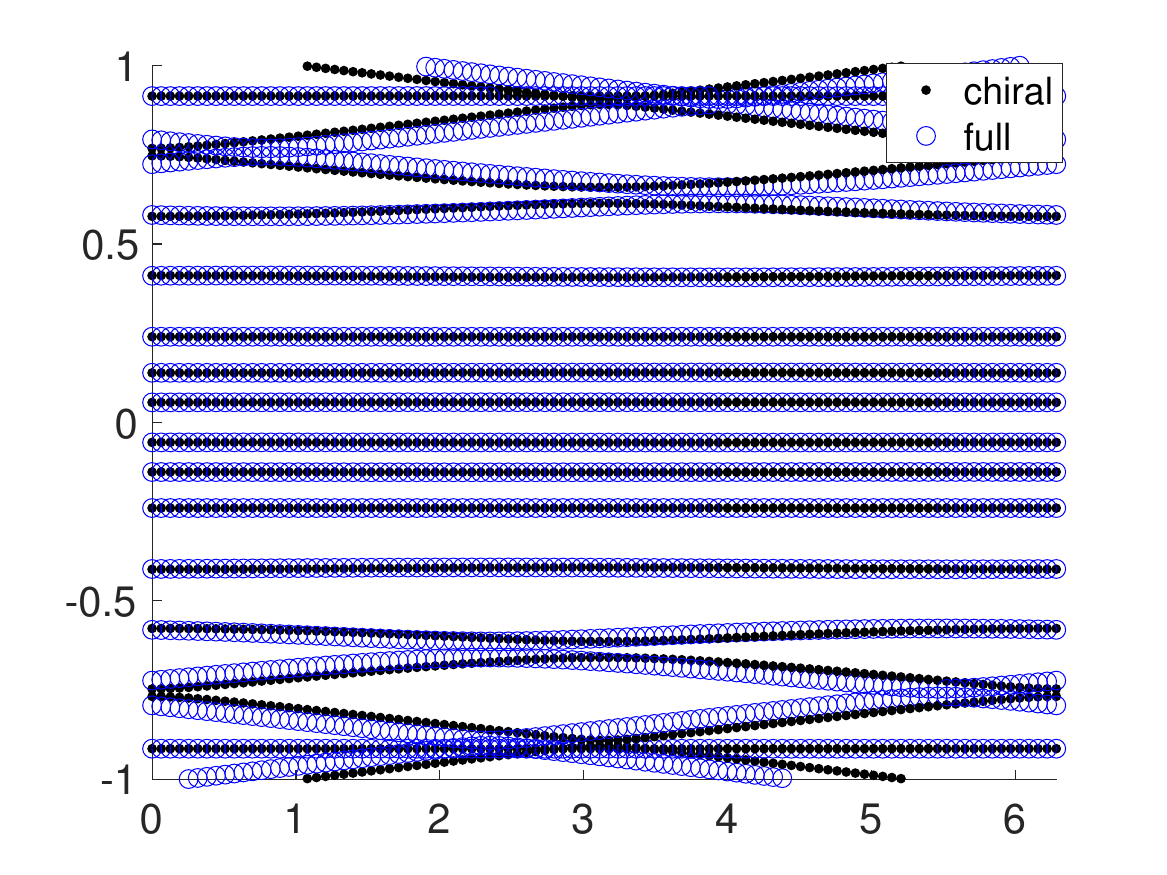}};
\node at (7.6-2,0) {\includegraphics[width=7.6cm]{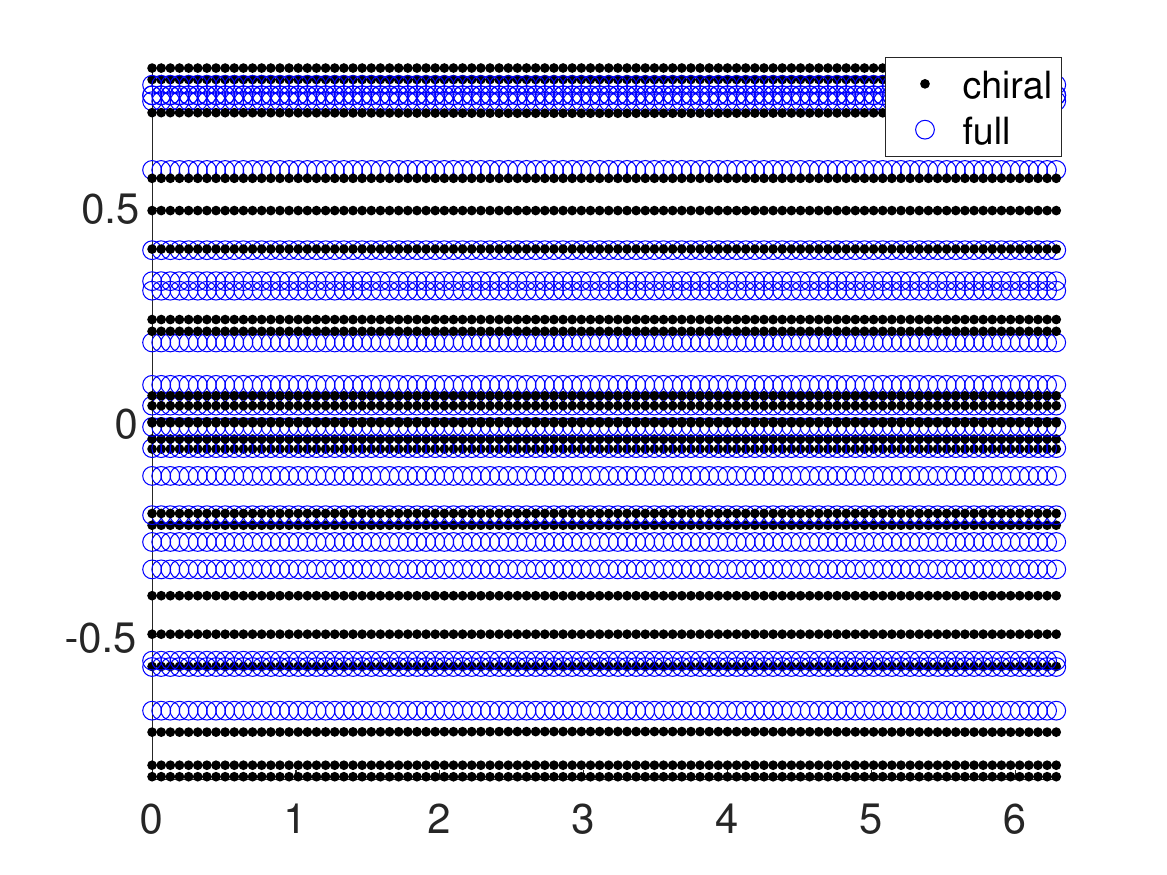}};
\node at (-2,-3.0) {$k_x/h$};
\node at (-2,3.0) {Chiral vs.~full -- low energy model};
\node at (-4-1.7,0) {$E$};
\node at (-4-1.7+7.6,0) {$E$};
\node at (-2+7.6,3.0) {Chiral vs.~full -- tight-binding model};
\node at (7.6-2,-3.0) {$k_x\cdot q$};
\end{tikzpicture}}
 \caption{Comparison of bands close to zero energy for the chiral limit $(w_0,w_1)=(0,1)$ of the low-energy model \eqref{eq:tm20} and the full low energy model $(w_0,w_1)=(0.8,1)$ (left) with $h=1/10$ showing an impressive agreement. Chiral limit of the tight-binding model and full tight-binding model for $L=30$ with $p=1$ and $q=30$  (right). \label{fig:comparison}}
 \end{figure}
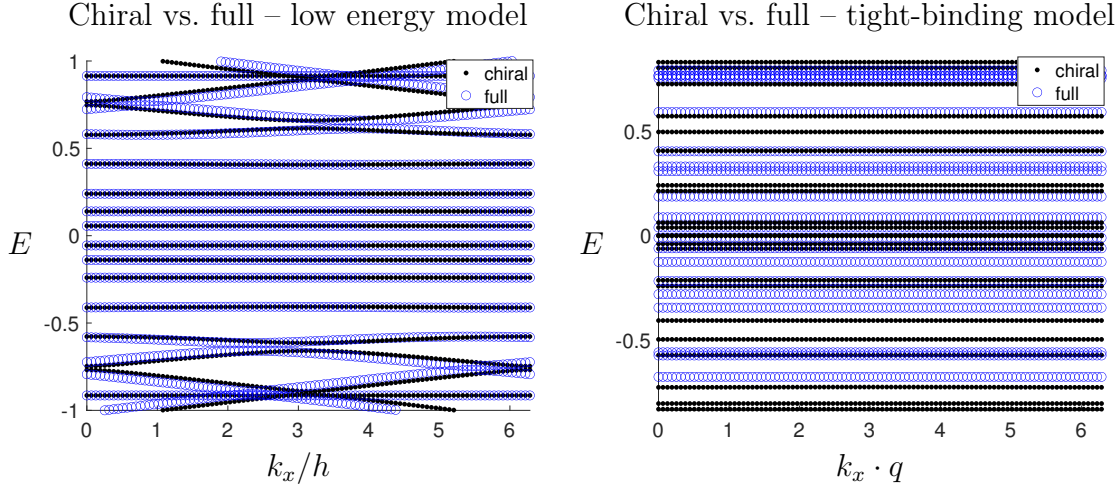

\subsection{Analysis close to rational moir\'e lengths and open questions}
The previous semiclassical description provides a representation of the operator model essentially for large moir\'e lengths. We can also find a semiclassical description close to arbitrary commensurable length scales. To translate this spectral problem near any commensurable $h$, the key proposition is

\begin{prop}\label{prop:commensurable}
Let $h = \frac{p}{q} +h'$, then there is a unitary transformation $\mathcal U_q: L^2(\mathbb T;\mathbb C^4) \rightarrow L^2(\mathbb T;\mathbb C^4) \otimes \mathbb C^q$ such that the Hamiltonian in \eqref{eq:HarperPDO} satisfies
$$(\mathcal U_q H_{\operatorname{\Psi DO}} \mathcal U_q^*)(x,h'D_x) = \mathbf t(k_{\perp})\otimes (e^{2\pi ix}J_{q,p}^*+e^{-2\pi ix}J_{q,p})+\mathbf t_0 \otimes I_q  + \tilde{V}_w(h'D_x),$$
where $J_{q,p}:=J_q^p$ with
\[J_q=\operatorname{diag}(1,\gamma,\ldots, \gamma^{q-1})\text{ with }\gamma=e^{2\pi i /q},\text{ and } (K_q)_{jk} =\begin{cases} 1& \text{if  $k\equiv j+1 \ \mathrm{mod}$ $q$,} \\
0& \text{otherwise.} \end{cases}\]
Here, $\hat{V}_w$ is defined as $V_w$ but with matrix-valued self-adjoint potentials
$$\hat{U}(h'D_x) = I_q+ e^{2\pi i h' D_x} K_q + e^{-2\pi i h' D_x} K_q^*$$
$$\hat{U}^{\pm}(h'D_x) = I_q-\frac{e^{2\pi i h' D_x}K_q +e^{-2\pi i h' D_x}K_q^*}{2}  \pm \sqrt 3 \left(\frac{e^{2\pi i h' D_x}K_q -e^{-2\pi i h' D_x}K_q^*}{2i} \right).$$
\end{prop}
\begin{proof}
Consider the unitary map $\mathcal U_q: L^2(\mathbb T ;\mathbb C^4) \rightarrow L^2(\mathbb T;\mathbb C^4) \otimes \mathbb C^q$, defined by
\[ (\mathcal U_q u) = \operatorname{diag}(u,T^1 u ,\ldots,T^{q-1} u)\] 
with $(T u)(x):= u(x-\tfrac{ p}{q}).$ For $u(x)=e^{-2\pi ix}$, this map satisfies
\begin{equation*}
\begin{split}
\mathcal U_qu \mathcal U_q^* = u J_q^p\quad\text{ and }\quad
\mathcal U_q e^{-2\pi i hD_x} \mathcal U_q^* = e^{-2\pi i h'D_{x}} K_q^{*}.
\end{split}
\end{equation*}
The result then follows immediately, as the operator consists of such primitive Fourier modes.
 \end{proof}
A more detailed analysis of this model close to commensurable moir\'e lengths is an open problem and should be compared to the magnetic case \cite{helffer1990analyse}. Using the results of Proposition \ref{prop:commensurable}, it is possible to show that for example for $p=1,q=2$, the chiral Hamiltonian also exhibits a potential well at zero energy.

\subsection*{Outline of the article}
In Section \ref{sec:pdo}, we briefly recall relevant background on semiclassical pseudodifferential operators. Section \ref{sec:wellsnormalform} contains our analysis of the spectral asymptotics for systems exhibiting a potential well. In Section \ref{sec:chiral}, we then apply the spectral asymptotics derived in the previous section to the chiral Hamiltonian of the pseudodifferential Harper model \eqref{eq:HarperPDO} and of the low-energy model \eqref{eq:tm20}.
The article also contains an appendix which consists of Section \ref{sec:aux_results} where we prove auxiliary results used in the proofs of Section \ref{sec:wellsnormalform},
and Section \ref{sec:antichiral} where we  for comparison discuss the anti-chiral limits of models \eqref{eq:HarperPDO} and \eqref{eq:tm20}. In the former model, there are various quasimodes at potential wells located at different energy levels, but not necessarily at zero, see  Theorem \ref{thm:Harperantichiral}. The gap-condition \eqref{eq:gap_condition} fails, but the operator is diagonalizable so the scalar results of Theorem \ref{theo:DiSJ} apply directly. In the latter model, there are no wells at all, see Remark \ref{rem:achLEM}. The bands near zero energy of the anti-chiral limits of each corresponding model are shown in the right panels of Figures \ref{fig:bands} and \ref{fig:bands2} for comparison.
Spectral aspects of these models will be discussed in the forthcoming article \cite{becker2022hofstadter}.

\section{Semiclassical pseudodifferential operators}\label{sec:pdo}

\subsection*{Notation}
We denote by $H^m(\RR^n)$ the Sobolev space of order $m$.
The Pauli matrices are denoted by $\sigma_i$ for $i \in \{1,2,3\}.$ Recall that the Kohn-Nirenberg symbol class $S^m(\RR\times\RR)$ is the set of all $a\in C^\infty(\RR\times\RR)$ such that
\begin{equation*}
\lvert \partial_\xi^j\partial_x^k a(x,\xi)\rvert\le C_{jk}(1+\lvert\xi\rvert)^{m-j},\quad j,k\ge0.
\end{equation*}
By $S_\delta^{m,k}$ we denote the class of symbols $p$ such that
$$
\lvert \partial_x^\alpha\partial_\xi^\beta p(x,\xi)\rvert\le C_{\alpha\beta} h^{-(\alpha+\beta)\delta} h^{-m}\langle\xi\rangle^{k-\beta},\quad \alpha,\beta\in\NN_0.
$$
We let $\Psi^{m}(\RR)$ and $\Psi_{\delta}^{m,k}(\RR)$ denote the corresponding class of semiclassical operators, and recall that if $a^w\in \Psi_{\delta}^{m,k}$ and $b^w\in \Psi_{\delta}^{m'\!,k'}$ then $a^wb^w\in \Psi_{\delta}^{m+m'\!,k+k'}$.

We identify functions on $T^*\TT=\TT\times\RR$ with functions on $T^*\RR=\RR\times\RR$ that are 1-periodic in the base variable $x$. The class $S^m(T^\ast\mathbb T)$ is identified with the subset of $S^m(T^*\RR)$ consisting of the symbols that are 1-periodic in $x$. Similarly, we identify symbols in $C^\infty(T^*\TT)$ belonging to $S(1)$ with symbols in $C^\infty(T^*\RR)$ that are 1-periodic in $x$ and belong to $S(1)$. 
We shall also need the symbol classes $S(m)$ where $m$ is an order function of the type
\begin{equation*}
m(y,\eta)=(1+\lvert y\rvert^2+\lvert \eta\rvert^2)^{\nu/2}
\end{equation*}
for some $\nu\ge0$, consisting of $a\in C^\infty(T^*\RR)$ such that
$\lvert\partial_\eta^\alpha\partial_y^\beta a(y,\eta)\rvert\le C_{\alpha\beta} m(y,\eta)$ for all $\alpha,\beta\in\NN_0$. For such $m$ we usually just write $S(\langle (y,\eta)\rangle^\nu)$, where we use the notation $\langle t\rangle=(1+\lvert t\rvert^2)^\frac12$ for $t\in\RR$.
All these symbol classes generalize in the natural way to $n\times m$ systems $a\in C^\infty(\RR\times\RR;\CC^{n\times m})$ and we shall not emphasize the size $\CC^{n\times m}$ in the notation. Usually we will also simply write e.g.~$L^2(\mathbb T)$ instead of $L^2(\mathbb T;\CC^d)$ when the vector dimension is clear from context.

\subsection{Pseudodifferential calculus on $\mathbb T$}
\label{subsec:pseudo}
In this subsection, we provide the relevant background on semiclassical pseudodifferential operators on $\mathbb T$ (see \cite[Section 5.3]{zworski} for a detailed exposition).
Let $\mathbb T=\RR/\ZZ$ and identify $\mathbb T$ with the fundamental domain $[0,1)$. Functions $u\in L^2(\mathbb T)$ are identified with periodic functions on $\RR$ with period 1. Symbols $a$ on $T^*\TT$ are identified with symbols $a(x,\xi)$ on $T^*\RR$ that are 1-periodic in $x$. The standard quantization of such a symbol is the semiclassical operator $A(h)=a(x,hD)$
acting on 1-periodic functions via
$$
A(h)u(x)=\frac{1}{2\pi h}\int_{T^*\RR} e^{i(x-y)\xi/h}a(x,\xi)u(y)\,dy\,d\xi
=\frac{1}{2\pi }\int_{T^*\RR} e^{i(x-y)\xi}a(x,h\xi)u(y)\,dy\,d\xi,
$$
interpreted in the weak sense. Note that a periodic function $u$ identified with $u\in L^2(\TT)$ belongs to the space $\mathscr S'(\RR)$ of tempered distributions, so this action is well-defined. It is easy to check that if $u(x)$ is 1-periodic then
\begin{equation}\label{eq:per}
A(h)u(x+k)=(a(x,hD)u(\bullet+k))(x)=A(h)u(x),\quad k\in\ZZ,
\end{equation}
so $A(h)$ preserves periodicity, and thus defines an operator on $\TT$. If, say, $a\in S(1)$, then $A(h):L^2(\mathbb T)\to L^2(\mathbb T)$ is bounded, see \cite[Theorem 5.5]{zworski}.

Using the standard quantization above, we may express the action of $A(h)$ in terms of Fourier coefficients: If $u$ is 1-periodic, write $u(y)=\sum_{n\in\ZZ} e^{2\pi iny}u_n$ where $u_n=\int_0^1 e^{-2\pi iny}u(y)\,dy$. Inserting this into the definition of $A(h)u(x)$ we obtain
$$
A(h)u(x)=\sum_{n\in\ZZ} u_n\int_\RR e^{ix\xi/h}a(x,\xi)\bigg(\frac{1}{2\pi h}\int_\RR e^{iy(2\pi nh-\xi)/h}\,dy\bigg)d\xi,
$$
where $\frac{1}{2\pi h}\int_\RR e^{iy(2\pi nh-\xi)/h}\,dy=\delta(2\pi nh-\xi)$ is the Dirac mass at $2\pi nh-\xi$. Hence,
\begin{equation}\label{eq:actionFC}
A(h)u(x)=
\sum_{n\in\ZZ} e^{2\pi inx}a(x,2\pi nh)u_n.
\end{equation}

\begin{rmks}
1. If we instead use the Weyl quantization, defined for a symbol $a(x,\xi)$ as
$$
a^w(x,hD)u(x)=\frac{1}{2\pi h}\int e^{i(x-y)\xi/h}a((x+y)/2,\xi)u(y)\,dy\,d\xi,
$$
then $a^w(x,hD)u(x)$ is still $1$-periodic, but formula \eqref{eq:actionFC} needs to be altered and becomes more involved in general. However, in the special case that $a$ is a linear combination of functions that only depend on either $x$ or $\xi$, so that $a(x,\xi)=a_0(x)+a_1(\xi)$, it is easy to see that we similarly get
\begin{equation*}
a^w(x,hD)u(x)=\sum_{n\in\ZZ} e^{2\pi inx}a(x,2\pi nh)u_n
\end{equation*}
for such operators. (Using the correspondence between different quantizations we have that
$a^w(x,hD)u(x)=\sum_{n\in\ZZ}e^{2\pi i nx}(e^{\frac{i}2hD_xD_\xi}a)(x,2\pi nh)u_n$ in the general case.)

2. \label{rmk2} Let $a$ be symbol on $T^*\TT$, identified with a symbol $a(x,\xi)$ on $T^*\RR$ that is 1-periodic in $x$, and suppose that $a(x,\xi)\in S(1)$. Then $a^w$ defines an operator on $L^2(\TT)$, understood as acting on the space of 1-periodic functions equipped with the norm in $L^2([0,1))$. Naturally, symbols in $S(1)$ also give rise to operators on $L^2(\RR)$, so we may also view $a^w$ as an operator
$a^w:L^2(\RR)\to L^2(\RR)$
by changing the domain. The same is true if $a\in S^m$ as long as we interpret $a^w$ on $L^2$ as a densely defined operator when $m>0$. This can be said to be the viewpoint in Theorem 1.3, should the symbol $P(x,\xi)$ in the statement happen to be 1-periodic in $x$.
\end{rmks}

We now show that $H_{\operatorname{TB}}$ is unitarily equivalent to the semiclassical pseudodifferential operator $H_{\operatorname{\Psi DO}}(w)$ in \eqref{eq:HarperPDO}.

\begin{lemm}\label{lem:PsiDO}
Let $H_{\operatorname{\Psi DO}}(w)$ be as in \eqref{eq:HarperPDO} with $h=(2\pi L)^{-1}$ and set
\begin{equation}
\label{eq:a}
a(x,\xi)=2\mathbf t(k_{\perp}) \cos(2\pi x)+\mathbf t_0+V_w(\xi).
\end{equation}
Then the discrete operator \eqref{eq:Harper} is unitarily equivalent to the pseudodifferential operator $H_{\operatorname{\Psi DO}}(w)=a^w(x,hD):L^2(\mathbb T)\to L^2(\mathbb T)$.
\end{lemm}

\begin{proof}
Let $\psi=(\psi_n)_{n=-\infty}^\infty\in \ell^2(\ZZ)$ and set $\Psi(x)=\sum_{n\in\ZZ} e^{2\pi inx}\psi_n$ so that $\psi_n$ is the $n$:th Fourier coefficient of $\Psi$. With $\psi=(\psi_k)_{k\in\ZZ}$, \eqref{eq:Harper} then gives rise to the action
\begin{align*}
A(w)\Psi(x)&=\sum_{n\in\ZZ}e^{2\pi inx} (H_\mathrm{TB}(w)\psi)_n
\\&=\sum_{n\in\ZZ}e^{2\pi inx} (\mathbf t(k_{\perp}) \psi_{n+1} +\mathbf t(k_{\perp}) \psi_{n-1}  + (\mathbf t_0 + V_w(n/L))\psi_n)
\end{align*}
which we rewrite as
$$
A(w)\Psi(x)=\sum_{n\in\ZZ}e^{2\pi inx}  (2\mathbf t(k_{\perp}) \cos(2\pi x)+\mathbf t_0+V_w(2\pi n h))\psi_n,
$$
where $h=(2\pi L)^{-1}$. In view of the remark above we may interpret this as the action of the semiclassical operator defined as the Weyl quantization $A(w)=a^w(x,hD)$ of the symbol $a$ given by \eqref{eq:a}, where we have suppressed the dependence on $w=(w_0,w_1)$ for simplicity. By the definition of $V_w$ we see that $a$ is a bounded smooth function which implies that $a^w(x,hD)$ is bounded on $L^2(\mathbb T)$, and the lemma follows.
\end{proof}

\section{Quasimodes near degenerate wells}\label{sec:wellsnormalform}

The purpose of this section is to study quasimodes of $P^w$ when $P^w$ has a 
potential well in the
sense of Definition \ref{def:degeneratewell}, and prove Theorem \ref{thm:expansionsintro} and Corollary \ref{cor:periodicquasimodesnormalform}.
We therefore let $P^w(x,hD)$ be a $2\times 2$ system of self-adjoint semiclassical operators with matrix valued symbol $P\in C^\infty(T^*\mathbb R)$. We shall assume that $P$ has an expansion $P\sim \sum_{k=0}^\infty h^kP_k$ where either $P_k\in S(1)$ for $k\ge0$ or $P_k\in S^{2-k}(T^*\mathbb R)$ for $k\ge0$. (The former implies that $P\in S(1)$ and the latter that $P\in S^2(T^*\mathbb R)$.) In both cases $P-\sum_{k=0}^Nh^kP_k=\mathcal O_{S(1)}(h^{N+1})$ for all $N\ge1$.
As explained in the introduction, we will assume that $a=1$ and $x_0=0$ in \eqref{eq:degeneratewell}, so that $P^w$ has a 
potential well at $(0,\xi_0)$. When proving Corollary \ref{cor:periodicquasimodesnormalform} we will in addition assume that $P(x,\xi)$ is 1-periodic in $x$, and identify $P$ with a symbol on $T^*\mathbb T$.

\subsection{Normal form}
If $P_0$ has a degenerate potential well then the gap condition \eqref{eq:gap_condition} is clearly violated. We shall therefore have to study the spectrum of $P^w$ by another approach, the first step of which is to obtain a suitable normal form. 

We begin with a general discussion and first recall the standard rescaling, so suppose that $p\in S^m(T^\ast\mathbb R)$ and make the change of variables $y=h^{-\frac12}x$. Then $p^w(x,hD_x)u(x)=p_h^w(y,D_y)v(y)$ where $u(x)=v(y)$ and the Weyl quantization of $p_h(y,\eta)=p(h^\frac12 y,h^\frac12\eta)$ is understood to be non-semiclassical, i.e.,
$$
p_h^w(y,D_y)v(y)=\frac{1}{2\pi }\int e^{i(y-t)\eta} p_h((y+t)/2,\eta) v(t)\,dt\,d\eta.
$$
In other words, if $\gamma^\ast v=v\circ \gamma$ denotes pullback by $\gamma(x)=h^{-\frac12}x$, and $P(h)=p^w(x,hD_x)$, then $p_h^w(y,D_y)=(\gamma^{-1})^*\circ P(h)\circ\gamma^\ast $. 
If $0<h<1$ then
$$
\lvert \partial_\eta^j\partial_y^k p_h(y,\eta)\rvert
\le C_{jk} h^{j/2}(1+\lvert h^\frac12\eta\rvert)^{m-j},\quad j,k\ge0,
$$
and if $m\ge0$ then standard calculus shows that the right-hand side is bounded by $C_{jk}(1+\lvert\eta\rvert)^{m-j}$ so 
\begin{equation*}\label{eq:hboundedset}
p(x,\xi)\in S^m(\RR_x\times\RR_\xi)\quad\Longrightarrow\quad p_h(y,\eta)\in S^m(\RR_y\times\RR_\eta),\quad 0<h<1,
\end{equation*}
uniformly.

If $p\in S(1)$ then the same calculations show that $p_h\in S(1)$ uniformly for $0<h<1$. We then have the following normal form.

\begin{prop}\label{prop:normalform}
Let $P\in C^\infty(T^*\RR)$ with $P(x,\xi)\sim\sum_{j=0}^\infty h^jP_j(x,\xi)$, where either $P_j\in S(1)$, $j\ge0$ or $P_j\in S^{2-j}(T^*\mathbb R)$, $j\ge0$.
Assume that $P_0$ has a potential well at $(0,\xi_0)$ in the sense of Definition \ref{def:degeneratewell} and that $P_1(0,\xi_0)=\diag(\mu_1,\mu_2)$ with $\mu_1,\mu_2\in\RR$. Then
\begin{equation}\label{eq:normalformpullback}
P^w(x,hD)=\gamma^*\circ\mathcal T\circ(\gamma^{-1})^*,\qquad \gamma(x)=h^{-\frac12}x,
\end{equation}
where $\mathcal T$ is an operator of the form
\begin{equation}\label{eq:normalformoperator}
\mathcal Tv(y)= e^{i\xi_0 y/\sqrt h}T^w(y,D)(e^{-i\xi_0\bullet/\sqrt h}v)(y)
\end{equation}
and $T^w$ is a $2\times 2$ self-adjoint matrix-valued system with expansion
\begin{equation}\label{eq:Texpansion}
T(y,\eta)=\sum_{j=0}^kh^{(j+2)/2}T_j(y,\eta)+h^{(k+3)/2}R_k(y,\eta;h),
\end{equation}
where
\begin{equation}\label{eq:T0}
T_0(y,\eta) = \diag(\eta^2+\omega^2y^2+\mu_1,\eta^2+\omega^2y^2+\mu_2)\quad\text{ with } \omega>0, 
\end{equation}
and where $T_j\in S(\langle(y,\eta)\rangle^{j+2})$ for $j\ge0$ and $R_k(h)\in S(\langle(y,\eta)\rangle^{k+3})$ uniformly for $0<h<1$. Moreover, if $P(x,\xi)$ is periodic in $x$ with period 1 then $T(y,\eta)$ is periodic in $y$ with period $h^{-\frac12}$.
\end{prop}

Note in particular that $T_0$ is just the symbol of a direct sum of harmonic oscillators which are perturbed and coupled to one another through terms appearing in $T_j,R_k$.

\begin{proof}
First note that for a symbol $p(x,\xi)$ we have after the symplectic change of variables $\zeta=\xi-\xi_0$ that
\begin{align*}
p^w(x,hD)u(x)
&=\frac{1}{2\pi h}\int e^{i(x-y)\zeta/h} e^{ix \xi_0 /h} p((x+y)/2,\zeta+\xi_0) e^{-i \xi_0 \bullet  /h}u(y)\,dy\, d\zeta\\
&= e^{ix \xi_0 /h} q^w(x,hD)( e^{-i \xi_0 \bullet  /h}u)(x)
\end{align*}
where $q(x,\zeta)=p(x,\zeta+\xi_0 )$. With $y=h^{-\frac12}x$ we have as above that $q^w(x,hD)u(x)=q_h^w(y,D)v(y)$ where $q_h(y,\eta)=p(h^{\frac12}y,h^{\frac12}\eta+\xi_0 )$ and $u(x)=v(y)$. 
Hence, 
$$
p^w(x,hD)u(x)=e^{iy \xi_0 /\sqrt h}q_h^w(y,D)(e^{-i \xi_0 \bullet /\sqrt h}v)(y)
$$
where $u(x)=v(y)$.

Applying this to $P^w(x,hD)$ we find that
$$
P^w(x,hD)u(x)=e^{iy \xi_0 /\sqrt h}T^w(y,D)(e^{-i \xi_0 \bullet /\sqrt h}v)(y)
$$
where $T(y,\eta)=P(h^{\frac12}y,h^{\frac12}\eta+\xi_0 )$ and $u(x)=v(y)$. Clearly, if $P(x,\xi)$ is 1-periodic in $x$ then $T(y,\eta)$ is periodic in $y$ with period $h^{-\frac12}$.

Now Taylor expand $T(y,\eta)$ near $y=0$, $\eta=0$. In view of Definition \ref{def:degeneratewell} we get, both when $P\in S(1)$ and when $P\in S^2(T^*\mathbb R)$,
$$
P_0(h^\frac12 y,h^\frac12\eta+\xi_0)=h(\eta^2+\omega^2y^2)\id_2+\sum_{j=3}^{k+1}h^{j/2}p_j^{(0)}(y,\eta) +h^{(k+1)/2}r_{k+1}^{(0)}(y,\eta;h)
$$
where $p_j^{(0)}\in S(\langle(y,\eta)\rangle^j)$ and $r_k^{(0)}(h)\in S(\langle(y,\eta)\rangle^{k+1})$ uniformly in $0<h<1$. Using Taylor's formula also on $P_1\in S(1)$ gives 
$$
hP_1(h^\frac12 y,h^\frac12\eta+\xi_0)=h\diag(\mu_1,\mu_2)+\sum_{j=1}^kh^{(j+2)/2}p_j^{(1)}(y,\eta) +h^{(k+3)/2}r_k^{(1)}(y,\eta;h)
$$
where $\mu_1,\mu_2\in\RR$, $p_j^{(1)}\in S(\langle(y,\eta)\rangle^j)$ and $r_k^{(1)}(h)\in S(\langle(y,\eta)\rangle^{k+1})$ uniformly in $0<h<1$. 
We then continue in this way to Taylor expand $h^2P_2,\ldots,h^{k+1}P_{k+1}$, and since $P-\sum_{j=0}^{k+1}h^jP_j=\mathcal O_{S(1)}(h^{k+2})$, the result follows by combining the expansions.
\end{proof}

\subsection{Quasimodes on $\TT$}\label{ss:quasimodes}

The main objective of Section \ref{sec:wellsnormalform} is to prove Theorem \ref{thm:expansionsintro} and thereby obtain approximate eigenvalues and quasimodes for $T^w(y,D)$ on $L^2(\RR)$. However, we first prove that this leads to the existence of approximate eigenvalues and quasimodes for $P^w(x,hD)$ on $L^2(\mathbb T)$:

\begin{proof}[Proof of Corollary \ref{cor:periodicquasimodesnormalform}]
First note that $P(x,\xi)$ is 1-periodic in $x$ by assumption. In view of \eqref{eq:per} we may then regard $P^w(x,hD)$ as an operator on $L^2(\TT;\CC^2)$ (densely defined when $P\in S^2$) by having it act on 1-periodic functions. Similarly, since $T(y,\eta)$ is periodic in $y$ with period $h^{-\frac12}$ by Proposition \ref{prop:normalform}, we can 
identify $\RR/h^{-\frac12}\ZZ$ with the fundamental domain $I_h=[-\frac12 h^{-\frac12},\frac12 h^{-\frac12})$ and view $T^w(y,D)$ and $\mathcal T$ in \eqref{eq:normalformoperator} as densely defined operators on $L^2(I_h)=L^2(I_h;\CC^2)$ by having them act on $h^{-\frac12}$-periodic functions. Note that $v\mapsto e^{i\xi_0\bullet/\sqrt h}v$ is not a unitary transformation on $L^2(I_h)$ since $e^{i\xi_0\bullet/\sqrt h}$ is not periodic with period $h^{-\frac12}$ in general, but $\mathcal T$ preserves periodicity. Indeed,
\begin{equation}\label{eq:preservesperiodicity}
\begin{aligned}
\mathcal Tv(y+h^{-\frac12})&=e^{i\xi_0(y+h^{-\frac12})/\sqrt h}T^w(y+h^{-\frac12},D)(e^{-i\xi_0 \bullet/\sqrt h}v)(y+h^{-\frac12})\\&=e^{i\xi_0(y+h^{-\frac12})/\sqrt h}T^w(y,D)(e^{-i\xi_0 (\bullet+h^{-\frac12})/\sqrt h}v(\bullet+h^{-\frac12}))(y)
\\&=e^{i\xi_0y/\sqrt h}T^w(y,D)(e^{-i\xi_0 \bullet/\sqrt h}v(\bullet+h^{-\frac12}))(y)
\end{aligned}
\end{equation}
where the second identity follows from $T(y,\eta)$ being $h^{-\frac12}$-periodic in $y$. Hence, if $v$ is periodic with period $h^{-\frac12}$ then $\mathcal Tv(y+h^{-\frac12})=\mathcal Tv(y)$.

Fix $\ell\in\NN_0$ and let $\{v^{(j)}(n)\}_{n\in\NN_0}$ and $\{\lambda^{(j)}(n)\}_{n\in\NN_0}$, $j=1,2$, be the quasimodes and approximate eigenvalues of $T^w(y,D)$ on $\RR$ given by Theorem \ref{thm:expansionsintro}. Then $T^w v^{(j)}(n)=\lambda v^{(j)}(n)+h^{\ell+\frac32}r^{(j)}(n)$ where the remainder $r^{(j)}(n)=r^{(j)}_{2\ell}(n;h)\in\mathscr S(\RR)$ has seminorms in $\mathscr S$ bounded  uniformly in $0<h<1$.
Let us fix $j$ and $n$ and drop them from the notation. Set $w(y)=e^{i\xi_0 y /\sqrt h}v(y)$ so that $\mathcal Tw=\lambda w+e^{i\xi_0 \bullet /\sqrt h}h^{\ell+\frac32}r$ and define
$$
\widetilde w(y)=\sum_{k\in\ZZ} w(y-kh^{-\frac12})=\sum_{k\in\ZZ} e^{i\xi_0 (y-kh^{-\frac12}) /\sqrt h}v(y-kh^{-\frac12}).
$$
By \eqref{eq:preservesperiodicity} it follows that if $k\in\ZZ$ then
\begin{align*}
\mathcal T(w(\bullet-kh^{-\frac12}))(y)&=(\mathcal T w)(y-kh^{-\frac12})\\&=\lambda w(y-kh^{-\frac12})+h^{\ell+\frac32}e^{i\xi_0 (y-kh^{-\frac12}) /\sqrt h}r(y-kh^{-\frac12}).
\end{align*}
Since $v\in\mathscr S$ it follows that $\widetilde w\in C^\infty$ is periodic with period $h^{-\frac12}$, and
\begin{equation}\label{eq:periodicqm}
(\mathcal T-\lambda)\widetilde w(y)=h^{\ell+\frac32}\sum_{k\in\ZZ}e^{i\xi_0 (y-kh^{-\frac12}) /\sqrt h} r(y-kh^{-\frac12}). 
\end{equation}
The weighted pullback $u=h^{-\frac14}\gamma^*(\widetilde w)$ is 1-periodic and $\WF_h(u)=\{(0,\xi_0)\}\subset T^*\mathbb T$,  and we shall show that it also has the other properties stated in the corollary.

By assumption we have $\lVert v\rVert_{L^2(\RR;\CC^2)}=\lVert \varphi_{n,\omega}\rVert_{L^2(\RR;\CC)}+\mathcal{O}(h^\frac12)=1+\mathcal{O}(h^\frac12)$, and we claim that $\lVert u\rVert_{L^2(\mathbb T)}=1+\mathcal{O}(h^\frac12)$ as well. Indeed,
\begin{align*}
\lVert u\rVert_{L^2(\mathbb T)}^2&=\int_{\mathbb T}\lvert h^{-\frac14}\widetilde w(h^{-\frac12}x)\rvert^2\,dx
=\int_{I_h}\bigg\lvert \sum_{k\in\ZZ}e^{i\xi_0 (y-kh^{-\frac12}) /\sqrt h}v(y-kh^{-\frac12})\bigg\rvert^2dy\\&
=\lVert v\rVert_{L^2(\RR)}^2+\sum_{k\in\ZZ}\sum_{\lvert j-k\rvert\ge1}e^{i\xi_0(j-k)/h}\int_{I_h}v(y-kh^{-\frac12})\overline{v(y-jh^{-\frac12})}\,dy.
\end{align*}
Since $v\in\mathscr S$ there is for any $N>0$ a constant $C>0$ such that
\begin{align*}
\lvert v(y-kh^{-\frac12})v(y-jh^{-\frac12})\rvert&\le C(1+\lvert y-kh^{-\frac12}\rvert^2)^{-2N}(1+\lvert y-jh^{-\frac12}\rvert^2)^{-N}\\&\le 2^NC(1+\lvert y-kh^{-\frac12}\rvert^2)^{-N}(1+\lvert (j-k)h^{-\frac12}\rvert^2)^{-N}
\end{align*}
where the second estimate follows from Peetre's inequality. It follows that
$$
\sum_{k\in\ZZ}\sum_{\lvert j-k\rvert\ge1}\int_{I_h}\lvert v(y-kh^{-\frac12})v(y-jh^{-\frac12})\rvert\,dy\le 2^NC \int_{-\infty}^\infty \frac{dy}{(1+y^2)^N}\sum_{n=1}^\infty \frac{2}{(1+n^2/h)^N}
$$
and since the right-hand side is $\mathcal{O}(h^N )$ we get $\lVert u\rVert_{L^2(\mathbb T)}^2=\lVert v\rVert_{L^2(\RR)}^2+\mathcal{O}(h^N)=1+\mathcal{O}(h^\frac12)$ by the triangle inequality, and thus $\lVert u\rVert_{L^2(\mathbb T)}=1+\mathcal{O}(h^\frac12)$.

Since $r\in\mathscr S$ uniformly in $0<h<1$ we can apply the same arguments to the right-hand side of \eqref{eq:periodicqm} and, in view of \eqref{eq:normalformpullback}, obtain
\begin{align*}
\lVert (P^w-\lambda)u\rVert_{L^2(\mathbb T)}&=h^{\ell+\frac32}\bigg(\int_{I_h}\bigg\lvert \sum_{k\in\ZZ}e^{i\xi_0 (y-kh^{-\frac12}) /\sqrt h} r(y-kh^{-\frac12})\bigg\rvert^2dy\bigg)^\frac12\\&=h^{\ell+\frac32}(\lVert r\rVert_{L^2(\RR)}+\mathcal{O}(h^N))
\end{align*}
for $N>0$, showing that $\lVert (P^w(x,hD)-\lambda) u\rVert_{L^2(\mathbb T)}=\mathcal{O}(h^{\ell+\frac32})$. 
\end{proof}

\subsection{Explicit WKB construction}

Before proving Theorem \ref{thm:expansionsintro} in full generality we first discuss a special case for which there exists a rather explicit WKB construction. Recall from \eqref{eq:harmoscbasis} the harmonic oscillator basis functions 
\begin{equation*}
\begin{split}
 \varphi_{n,\omega}(y) &=\phi_{n, \omega}(y) e^{-\tfrac{\omega y^2}{2}} = \tfrac{1}{\sqrt{2^n n!}} \left(\tfrac{\omega}{\pi}\right)^{1/4}  H_n\left(\sqrt{\omega} y\right)e^{-\tfrac{\omega y^2}{2}}, \quad n \in \mathbb N_0,\quad \omega>0,
 \end{split}
 \end{equation*}
where $H_n$ is the $n$:th Hermite polynomial, which is even (odd) when $n$ is even (odd).

Let $C_d$ denote the module of homogeneous polynomials of degree $d$ in the ring  $C=\CC[y,\eta]$. Let the polynomials of even and odd degree in $\CC[y,\eta]$ be denoted by
\begin{align*}
\mathcal P_{\mathrm{even}}&=\left\{ \sum_{0\le j\le n}p_{2j}:   p_{k}\in C_k,\ n\in\NN_0\right\},\quad \mathcal P_{\mathrm{odd}}=\left\{ \sum_{0\le j\le n}p_{2j+1}:   p_{k}\in C_k,\ n\in\NN_0\right\}.
\end{align*}

\begin{lemm}\label{lem:parity}
Let $v\in C^\infty$ be either even or odd, and let $p\in\CC[y,\eta]$. Then $p^w(y,D)v$ has the same parity as $v$ when $p\in \mathcal P_{\mathrm{even}}$, and opposite parity when $p\in \mathcal P_{\mathrm{odd}}$.
\end{lemm}

\begin{proof}
By linearity it suffices to consider the case of a homogeneous polynomial $p(y,\eta)=y^{n}\eta^{k}$, where $n+k$ is either even or odd, depending on if $p\in \mathcal P_{\mathrm{even}}$ or $p\in \mathcal P_{\mathrm{odd}}$. Since $$p((y+s)/2,\eta)=2^{-n}\sum_{j=0}^n\binom{n}{j}y^{n-j}\eta^k s^j$$ we find that $p^w(y,D)v(y)$ is a linear combination of terms $y^{n-j}D_y^k(y^jv(y))$ where $0\le j\le n$. 
Since multiplication by $y^m$ changes parity if and only if $m$ is odd, and differentiation $D_y^k$ changes parity if and only if $k$ is odd, it follows that $y^{n-j}D_y^k(y^jv(y))$ will have the same parity as $v$ when $n+k$ is even and opposite parity to $v$ when $n+k$ is odd.
\end{proof}

Introduce the sets
\begin{equation*}
\begin{split}
 \mathcal A_{\operatorname{sym}}&:=\left\{ \begin{pmatrix} p_{11} & p_{12} \\ \bar p_{12} & p_{22} \end{pmatrix}: p_{jk}\in\mathcal P_{\mathrm{even}}
 \right\} \text{ and }
  \mathcal A_{\operatorname{asym}}:=\left\{ \begin{pmatrix} p_{11} & p_{12} \\ \bar p_{12} & p_{22} \end{pmatrix}:  p_{jk}  \in \mathcal P_{\mathrm{odd}}\right\}. 
 \end{split}
\end{equation*}
Let $T=\sum_{j=0}^kh^{(j+1)/2}T_j+h^{(k+3)/2}R_k$ be the symbol given by Proposition \ref{prop:normalform}. Recalling that this is a Taylor expansion we have $T_{2j-1}\in\mathcal A_\mathrm{asym}$ and $T_{2j}\in\mathcal A_\mathrm{sym}$ for $j\ge1$.
We then have the following result.

\begin{prop}\label{thm:quasimodes2}
Assume that for all $\ell\in\NN_0$, the Weyl symbol of $T^w(y,D)$ has an expansion $T = h\sum_{i=0}^{2\ell} h^{i/2} T_i+h^{\ell+\frac32}R_{2\ell}$
such that
\[ T_0(y,\eta) = \operatorname{diag}(\eta^2 + \omega^2 y^2 + \mu_1, \eta^2 + \omega^2 y^2+ \mu_2), \quad \omega>0, \quad \mu_1-\mu_2 \notin (4\mathbb Z +2)\omega\]
where $T_{2i+1} \in \mathcal A_{\operatorname{asym}}$ for $i \in \mathbb N_0$ and $T_{2i} \in  \mathcal A_{\operatorname{sym}}$ for $i\in \mathbb N_0$. In addition, for $\mu_1-\mu_2 \in 4\omega\ZZ$ we assume that $T_{2i}$ is a diagonal matrix and $T_{2i+1}$ has only off-diagonal entries.  Assume also that there is a $\nu\ge0$ such that $T_i,R_{2\ell}\in S(\langle (y,\eta)\rangle^{\nu})$ for $0\le i\le 2\ell$, uniformly in $0<h<1$.
Then there exist approximate eigenvalues
\[\lambda^{(j)}(n) = h \sum_{i \ge 0} h^{i/2} \lambda_{i}^{(j)}(n) \]
with $\lambda^{(j)}_0(n) = (2n+1)\omega+\mu_j$ for some $n \in \mathbb N_0$, $j\in \{1,2\}$, $\lambda^{(j)}_i=0$ for $i \in 2 \mathbb N_0+1$ and quasimodes 
\[ u^{(j)}(n,y) =\sum_{i \ge 0}h^{i/2} u^{(j)}_i(y) e^{-\tfrac{\omega y^2}{2}}\]
where $u_{i}^{(j)}$ are polynomials such that 
\begin{equation}\label{eq:WKB2}
\Big(T^w(y,D)-h\sum_{i =0}^{2\ell} h^{i/2} \lambda^{(j)}_i(n)\Big)\sum_{i=0}^{2\ell} h^{i/2} u^{(j)}_i(y) e^{-\frac{\omega y^2}{2}} = O_\mathscr{S}(h^{\ell+\frac32}).
\end{equation}
The leading order amplitudes are $u^{(1)}_0=(\phi_{n,\omega},0)$ and $u_0^{(2)}=(0,\phi_{n,\omega})$.
\end{prop}

\begin{proof}
The eigensystem for $T_0^w(y,D)$ is given by
\begin{equation*}
\begin{split}
 T_0^w(y,D)\begin{pmatrix} \phi_{n,\omega},0 \end{pmatrix}^t&= ((2n+1) \omega + \mu_1)\begin{pmatrix} \phi_{n,\omega},0 \end{pmatrix}^t \text{ for }n \in \mathbb N_0,  \\
  T_0^w(y,D)\begin{pmatrix} 0,\phi_{m,\omega} \end{pmatrix}^t&= ((2m+1) \omega + \mu_2)\begin{pmatrix} 0,\phi_{m,\omega} \end{pmatrix}^t \text{ for }m \in \mathbb N_0,
 \end{split}
 \end{equation*}
with eigenvectors $\begin{pmatrix} \phi_{n,\omega},0 \end{pmatrix}^t \text{ and }\begin{pmatrix} 0,\phi_{m,\omega} \end{pmatrix}^t$.

We therefore make the approximate eigenvalue and quasi-mode ansatz 
\begin{equation*}
\begin{split}
\lambda &=
\sum_{i \ge 0} h^{i/2} \lambda_{i} \quad\text{ and }\quad
u(y) =\sum_{i \ge 0}h^{i/2} u_i(y) e^{-\tfrac{\omega y^2}{2}},
\end{split}
\end{equation*}
where we present the construction without loss of generality for $ u_0^{(1)}(y):=(\phi_{n,\omega}(y),0)^t$ rather than $(0,\phi_{m,\omega}(y))^t$ and choose $\lambda_0 = (2n+1)\omega+\mu_1$.
Recall that the Hermite polynomial $H_n$ is an even polynomial if $n$ is even and odd polynomial if $n$ is odd. We may assume without loss of generality that $n$ is odd.
Iteratively constructing a WKB solution satisfying \eqref{eq:WKB2} is equivalent to successively solving 
\begin{equation*}
\sum_{i=0}^{k} (T_i^w-\lambda_i) u_{k-i}e^{-\tfrac{\omega y^2}{2}} =0
\end{equation*}
for $k=0,\ldots,2\ell$. In fact, as we will show the $u_j$ are polynomials so $u_je^{-\omega y^2/2}\in\mathscr S$. Since the remaining terms in \eqref{eq:WKB2} are 
$h^{1+(i+j)/2}T_i^w u_je^{-\omega y^2/2}$ for $2\ell <i+j\le 4\ell$ together with $h^{\ell+\frac32} R_{2\ell}^w(\sum_{j=0}^{2\ell} h^{j/2}u_je^{-\omega y^2/2})$, and since $T_i,R_{2\ell} \in S(\langle (y,\eta)\rangle^\nu)$, these terms are all in $h^{\ell+\frac32}\mathscr S$ uniformly in $0<h<1$ by assumption which gives an error of order $O_\mathscr{S}(h^{\ell+\frac32})$.

\emph{Step 1:} The first step is to note that 
\[ (T_0^w -\lambda_0) u_0 e^{-\tfrac{\omega y^2}{2}} =0\]
holds by assumption.

\emph{Step 2:} We shall argue by induction as $k$ runs through two consecutive integers. We first notice for $k=1$ that
\[ (T_0^w -\lambda_0) u_1 e^{-\tfrac{\omega y^2}{2}} =-(T_1^w-\lambda_1) u_0e^{-\tfrac{\omega y^2}{2}}=\begin{pmatrix} \lambda_1-(T_1^w)_{11} \\
-(T_1^w)_{21} \end{pmatrix}\phi_{n,\omega}(y).\]
To solve this for $u_1$, the right-hand side must be orthogonal to $\ker(T_0^w-\lambda_0)$. 
We then observe that according to Lemma \ref{lem:parity}, applying $(T_1^w)_{j1}$ to $\phi_{n,\omega}$ changes the parity of that function, which implies that  
$$\left\langle \phi_{n,\omega},(T_1^w)_{j1}\phi_{n,\omega}\right\rangle=0,\quad j=1,2.$$
This gives $\lambda_1=0$ and since $((T_0^w)_{11}-\lambda_0)^{-1}$ is bounded on the orthogonal complement of $\ker ((T_0^w)_{11}-\lambda_0)=\Span\{\phi_{n,\omega}\}$ we can take
\[ (u_1)_1 :=- e^{\tfrac{\omega y^2}{2}}((T_0^w)_{11} -\lambda_0)^{-1} (T_1^w)_{11}\phi_{n,\omega}.\]
For $(u_1)_2$ there are two cases to consider: if $\mu_1-\mu_2 = 4\omega l \in 4\omega \ZZ$ then $\ker((T_0)_{22}^w-\lambda_0)=\Span\{\phi_{n+2l,\omega}\}$, and otherwise if $\mu_1-\mu_2\notin 2\omega\mathbb Z  $ then $\ker((T_0)_{22}^w-\lambda_0)=\{0\}$. From the assumptions on $T_1$ we see that in either case we can take
\[ (u_1)_2 :=- e^{\tfrac{\omega y^2}{2}}((T_0^w)_{22} -\lambda_0)^{-1} (T_1^w)_{21}\phi_{n,\omega}.\]
As mentioned, $(T_1^w)_{j1}\phi_{n,\omega}$ is even since $n$ is odd. Multiplying by $e^{-\omega y^2/2}$ and applying $(T_0^w)_{jj}-\lambda_0$ to both sides does not change the parity, so $(u_1)_j$ must also be even. It is easy to check that $(u_1)_j$ is a polynomial function for $j=1,2$, so $u_1e^{-\omega y^2/2}\in\mathscr S$. This follows as 
\[ \left\{\sum_{i=0}^{N} a_i y^{2i} e^{-\omega y^2/2}; a_i \in \CC\right\} \text{ and }  \left\{\sum_{i=0}^{N} a_i y^{2i+1} e^{-\omega y^2/2}; a_i \in \CC\right\}\]
are invariant subspaces of $T_0$ for any $N \in \mathbb N$; see the last paragraph in the proof of Proposition \ref{thm:simple} below.

For $k=2$ we get
\begin{align*} (T_0^w -\lambda_0) u_2 e^{-\tfrac{\omega y^2}{2}} &=-(T_1^w-\lambda_1) u_1e^{-\tfrac{\omega y^2}{2}}-(T_2^w-\lambda_2) 
u_0e^{-\tfrac{\omega y^2}{2}}
\end{align*}
where $u_0=(\phi_{n,\omega},0)^t$ and $\lambda_1=0$ by the previous steps.
To solve the equation for $(u_2)_1$, the right-hand side must be orthogonal to $\ker ((T_0^w)_{11}-\lambda_0)=\Span\{\phi_{n,\omega}\}$, which means that $\lambda_2$ must satisfy
$$
\lambda_2=\bigg\langle (T_2^w)_{11}\phi_{n,\omega}+\Big(T_1^wu_1 e^{-\tfrac{\omega y^2}{2}}\Big)_1,\phi_{n,\omega}\bigg\rangle.
$$
With this choice of $\lambda_2$ we then get the solution
$$
(u_2)_1=-e^{\tfrac{\omega y^2}{2}}((T_0^w)_{11}-\lambda_0)^{-1}\bigg[ \Big(T_1^wu_1 e^{-\tfrac{\omega y^2}{2}}\Big)_1+((T_2^w)_{11}-\lambda_2)\phi_{n,\omega}\bigg].
$$
Note that the expression in brackets is an odd function by Lemma \ref{lem:parity}, and as above we find that $(u_2)_1$ is an odd polynomial.

For $(u_2)_2$ we get the equation
\begin{align*} ((T_0^w)_{22} -\lambda_0) (u_2)_2 e^{-\tfrac{\omega y^2}{2}} &=-\Big(T_1^w u_1e^{-\tfrac{\omega y^2}{2}}\Big)_2-(T_2^w)_{21}\phi_{n,\omega}
\end{align*}
so by similar reasoning as above we get
\[ (u_2)_2 :=- e^{\tfrac{\omega y^2}{2}}((T_0^w)_{22} -\lambda_0)^{-1} \bigg[ \Big(T_1^w u_1e^{-\tfrac{\omega y^2}{2}}\Big)_2+(T_2^w)_{21}\phi_{n,\omega}\bigg]
\]
where $(u_2)_2$ is an odd polynomial. Observe that for $\mu_1-\mu_2 \in 4\omega\ZZ$ the term in the bracket vanishes.

\emph{Step 3:} 
Now let $k\in 2\NN-1$ be arbitrary, and assume that $\lambda_i$ and $u_i$ have already been determined for $0\le i<k$ such that for $i$ odd we have $\lambda_i=0$ and $u_i = ((u_i)_1,(u_i)_2)^t$ where $(u_i)_1$ and $(u_i)_2$ are even polynomials, while for $i$ even $u_i = ((u_i)_1,(u_i)_2)^t$ where $(u_i)_1$ and $(u_i)_2$ are odd polynomials.
Then
\begin{align*}
(T_0^w-\lambda_0)   u_ke^{-\tfrac{\omega y^2}{2}}&=-(T_k^w-\lambda_k) u_{0}e^{-\tfrac{\omega y^2}{2}}\\
 &\quad-\sum_{i\in [1,k-2]\cap2\ZZ+1} T_i^w u_{k-i}e^{-\tfrac{\omega y^2}{2}}-\sum_{i\in [2,k-1]\cap2\ZZ} (T_i^w-\lambda_i) u_{k-i}e^{-\tfrac{\omega y^2}{2}}
\end{align*}
where all terms on the right are even functions by the induction hypothesis and Lemma \ref{lem:parity}, with the exception of $\lambda_k u_0e^{-\omega y^2/2}=\lambda_k(\phi_{n,\omega},0)^t$ which is odd. To solve the equation for $(u_k)_1$ the right-hand side must be orthogonal to $\ker ((T_0^w)_{11}-\lambda_0)=\Span\{\phi_{n,\omega}\}$ which then gives $\lambda_k=0$ and 
\begin{equation*}
    \begin{split}
    (u_k)_j =- e^{\tfrac{\omega y^2}{2}}((T_0^w)_{jj} -\lambda_0)^{-1} \Bigg[&\sum_{i \in [2,k-1]\cap 2\ZZ} \Big((T_i^w-\lambda_i)u_{k-i}e^{-\tfrac{\omega y^2}{2}}\Big)_{j}\\
        &+\sum_{i \in [1,k] \cap 2\ZZ+1} \Big(T_i^wu_{k-i}e^{-\tfrac{\omega y^2}{2}}\Big)_{j}\Bigg],\quad j=1,2,
            \end{split}
\end{equation*} 
which makes $u_k$ an even polynomial by the same arguments as before.

\emph{Step 4:} Under the same hypothesis as in step 3, but now with $k \in 2\mathbb N$, we define 
\begin{align*}   \lambda_k = \Bigg\langle (T_k^w)_{11}\phi_{n,\omega}+&\sum_{i \in [2,k]\cap 2\ZZ} \Big((T_i^w-\lambda_i)u_{k-i} e^{-\tfrac{\omega y^2}{2}}\Big)_1\\+&\sum_{i \in [1,k-1] \cap 2\ZZ+1} \Big(T_i^w u_{k-i} e^{-\tfrac{\omega y^2}{2}}\Big)_1,\phi_{n,\omega} \Bigg \rangle.\end{align*}

In analogy with the case $k=2$, this allows us to define
\begin{equation*}
    \begin{split}
 (u_k)_j =-e^{\tfrac{\omega y^2}{2}} ((T_0^w)_{jj} -\lambda_0)^{-1}\Bigg[&\sum_{i \in [2,k]\cap 2\ZZ} \Big((T_i^w-\lambda_i)u_{k-i}e^{-\tfrac{\omega y^2}{2}}\Big)_j\\
 &+\sum_{i \in [1,k-1] \cap 2\ZZ+1} \Big(T_i^w u_{k-i}e^{-\tfrac{\omega y^2}{2}}\Big)_j\Bigg],\quad j=1,2,
\end{split}
\end{equation*}
which makes $u_k$ an odd {polynomial}, and this closes the recurrence scheme.
\end{proof}

Note that for the proof of Proposition \ref{thm:quasimodes2} to work, the assumption that $\mu_1-\mu_2\notin (4\ZZ+2)\omega$ is crucial. As mentioned in the introduction, this assumption is violated for both the pseudodifferential Harper model \eqref{eq:HarperPDO} and the low-energy model \eqref{eq:tm20}, see the proofs of Theorems \ref{cor:periodicquasimodes2} and \ref{cor:periodicquasimodes} in Section \ref{sec:chiral}.

\subsection{Low-lying spectral analysis}\label{ss:completeness}

To prove Theorem \ref{thm:expansionsintro} 
in full generality
we will adapt a technique of Barry Simon \cite{simon1983semiclassical}.  The idea is to first show that the spectrum is stable in a certain sense, and then use this fact to obtain asymptotic expansions of the eigenvalues and eigenvectors, which by truncation give approximate eigenvalues and quasimodes. We recall that we in this context regard $T^w(y,D)$ as an operator on  $\RR$ with dense domain in $L^2(\RR)$. It will be convenient to also be able to express the operator $T^w(y,D)$ in the variables $x,\xi$ in order to make use of the semiclassical symbolic calculus. 
We may without loss of generality assume that $\xi_0=0$ in \eqref{eq:normalformoperator}, so to avoid additional notation we will simply write
\begin{equation}\label{eq:normalformpullback2}
P^w(x,hD)=\gamma^*\circ T^w(y,D)\circ(\gamma^{-1})^*,\qquad \gamma(x)=h^{-\frac12}x,
\end{equation}
and regard $P^w(x,hD)$ as an operator on $\RR$ with dense domain in $L^2(\RR)$; in other words, we assume that  the well is located at $(0,\xi_0)=(0,0)$.

The low-lying eigenvalues of $T^w(y,D)$ that we are interested in correspond to the bottom of the point spectrum of $P^w(x,hD)$ in \eqref{eq:normalformpullback} resulting from the well at $(0,\xi_0)$, so certain care has to be taken to avoid potential contribution from other wells or other components of the zero set of $\det(P_0)$ (compare \eqref{eq:componentszeroset} and \eqref{eq:levelsetHarper}--\eqref{eq:levelset2Harper} below). 
To make this more precise, let $\chi\in C^\infty(T^*\RR)$ be a cutoff function, independent of $h$ and supported in a small neighborhood of $(0,\xi_0)$, such that $0\le \chi\le 1$ and $\chi\equiv1$ near $(0,\xi_0)$, and define the {\it massive Weyl operator}
\begin{equation}\label{eq:massive}
P_\mathrm{mass}^w(x,hD)=P^w(x,hD)+(1-\chi^w(x,hD))\id_{\CC^{2\times2}}.
\end{equation}
By Definition \ref{def:degeneratewell}, 
there is a constant $C\ge0$ such that 
$P_0(x,\xi)\ge -C$ for all $(x,\xi)$ in the sense of semi-bounded operators,
where $P_0$ is the principal symbol of $P^w(x,hD)$.
Since we can always multiply $1-\chi$ by a sufficiently large multiple of $C$  if necessary, we may without loss of generality assume that $\chi$ can be chosen so that
\begin{equation}\label{eq:lowerboundwithchi}
P_0(x,\xi)+1-\chi(x,\xi)\ge \min(1,\tfrac12((\xi-\xi_0)^2+\omega^2x^2))\id_2.
\end{equation}
In particular, $P_0(x,\xi)+1-\chi(x,\xi)$ is positive definite away from $(0,\xi_0)$. 
When $P\in S^2(T^*\mathbb R)$ we can, by using \eqref{eq:lowerboundS2} and arguing in a similar way, make sure that
\begin{equation}\label{eq:lowerboundwithchi2}
(P_\mathrm{mass}^w(x,hD)u,u)\ge  ((V+1-\chi^w)u,u)-Ch(u,u)
\end{equation}
where $V(x)+1-\chi(x,\xi)\in S(1)$ is positive definite away from $(0,\xi_0)$ and satisfies 
\begin{equation}\label{eq:lowerboundwithchi3}
V(x)+1-\chi(x,\xi)\ge \min(1,Cx^2)\id_2
\end{equation}
for some $C>0$.
In particular, both when $P\in S(1)$ and when $P\in S^2(T^*\mathbb R)$ it follows that $P_\mathrm{mass}^w$ is microlocally elliptic away from $(0,\xi_0)$, and it is easy to check that the approximate eigenvalues of order $\mathcal{O}(h)$ of $P^w$ which correspond to quasimodes microlocalized at $(0,\xi_0)$ are precisely the approximate eigenvalues of order $\mathcal{O}(h)$ of the massive Weyl operator $P_\mathrm{mass}^w$.

Having introduced $P_\mathrm{mass}^w$ we then let $T_\mathrm{mass}^w$ be the operator
\begin{equation}\label{eq:massiveT}
T_\mathrm{mass}^w(y,D)=T^w(y,D)+1-G^w(y,D), \quad G(y,\eta)=\chi(h^{\frac12}y,h^{\frac12}\eta),
\end{equation}
so that $P_\mathrm{mass}^w(x,hD)=\gamma^*\circ T_\mathrm{mass}^w(y,D)\circ (\gamma^{-1})^*$. In particular
\begin{equation*}
\chi^w(x,hD_x)=\gamma^*\circ G^w(y,D_y)\circ(\gamma^{-1})^*,
\end{equation*}
and $P_\mathrm{mass}^w(x,hD)$ and $T_\mathrm{mass}^w(y,D)$ have the same spectrum. Also, since $\chi\in C_0^\infty(T^*\RR)$ in \eqref{eq:massive} is independent of $h$, and $G(y,\eta)=\chi(h^{\frac12}y,h^\frac12\eta)$ by  \eqref{eq:massiveT}, we see that
\begin{equation}\label{eq:suppcondG}
(y,\eta)\in\supp(1-G)\quad\Longrightarrow\quad \lvert h^{1/2}\eta\rvert\ge \delta_0 >0
\end{equation}
for some constant $\delta_0$.

Recall that $\varphi_n=\varphi_{n,\omega}$, $n\ge0$, are the harmonic oscillator basis functions given by \eqref{eq:harmoscbasis}, where we omit $\omega$ to shorten notation. From \eqref{eq:Texpansion}, we notice that $T^w=hT_0^w+h^{3/2}R_0^w$ where $R_0\in S(\langle(y,\eta)\rangle^3)$ uniformly for $0<h<1$, and
\begin{equation*}
T_0^w(\varphi_n,0)^t=e_n^{(1)}(\varphi_n,0)^t,
\qquad T_0^w(0,\varphi_n)^t=e_n^{(2)}(0,\varphi_n)^t,
\end{equation*}
where
\[ e_n^{(1)} = (2n+1)\omega+\mu_1 , \quad e_n^{(2)}=(2n+1)\omega+\mu_2,\quad n \in \mathbb N_0.\]
Let $(e_n)_{n \in \mathbb N}$ denote a monotonically increasing ordering of the two sets of eigenvalues.
The spectrum of $T_\mathrm{mass}^w$ is covered in the following sense:

\begin{thm}\label{thm:stability}
Let $\lambda_n(h)$ be the $n$:th eigenvalue, counting multiplicity, of $T_{\operatorname{mass}}^w$ and let $e_n$ be the $n$:th eigenvalue, counting multiplicity, of $T_0^w$, viewed as densely defined operators on $L^2(\RR)$. Then for $n$ fixed and $h$ small, $T_{\operatorname{mass}}^w$ has at least $n$ eigenvalues and
$$
\lim_{h\to0^+} \lambda_n(h)/h=e_n.
$$
\end{thm}

Note that if $\mu_1-\mu_2\notin 2\omega\ZZ$ then $T_0^w$ has only simple eigenvalues and all elements in $(e_n)_{n\in\NN}$ are distinct. If $\mu_1-\mu_2= 2\omega N$ for some $0\ne N\in\ZZ$, then $T_0^w$ has eigenvalues both of multiplicity one and two, and $(e_n)_{n\in\NN}$ contains both some elements that are distinct, and some that appear twice. If $N=0$ then all eigenvalues have multiplicity two, and all elements in $(e_n)_{n\in\NN}$ appear twice.
Theorem \ref{thm:stability} covers all situations. However, to avoid cumbersome notation involving $N$ we will assume that $(\mu_1,\mu_2)=(-\omega,\omega)$ in the sequel. 
The eigenvalues $(e_n)_{n\in\NN}$ of $T_0^w$ are then given by
\begin{equation}\label{eq:eigenvalueordering}
e_{2n+j}=e_n^{(j)}=(2n+1)\omega+(-1)^j\omega,\quad n\in\NN_0,\quad j=1,2.
\end{equation}
In particular, $e_1=0$ is a simple eigenvalue while $e_{2m}=e_{2m+1}$ for $m\ge1$.

Theorem \ref{thm:stability} has an analog for scalar self-adjoint Schr\"odinger operators on the line, and as mentioned we will adapt a proof by Simon \cite[Theorem 1.1]{simon1983semiclassical} to our situation.
One difference is that we shall use a microlocal cutoff function instead of a local one which allows applications to operators $P^w$ with bounded symbols (such as the pseudodifferential Harper model) when the domain of $P^w$ is all of $L^2$, while the domains of the operators $T_0^w$ and $R_0^w$ in the expansion $T^w=hT_0^w+h^{3/2}R_0^w$ are strictly smaller.

To this end, fix $J\in C_0^\infty(\RR)$ with $0\le J\le 1$ and $J(y)=1$ (resp.~0) if $\lvert y\rvert\le 1$ (resp.~$\lvert y\rvert\ge 2$), and let  
\begin{equation}\label{def:J1}
J_1(y,\eta;h)=J(h^{1/10}y)J(h^{1/10} \eta).
\end{equation}

\begin{lemm}\label{lem:remainderboundL2}
If $R_0\in S(\langle(y,\eta)\rangle^3)$ uniformly for $0<h<1$ then 
$$
\lVert h^{3/2}J_1^w R_0^w J_1^w\rVert =\mathcal{O}(h^{6/5})
$$
in $L^2(\RR;\CC^2)$.
\end{lemm}

\begin{proof}
Since $\lvert y\rvert,\lvert \eta\rvert\le 2 h^{-1/10}$ on the support of $J_1$ we have
$$
h^{3/2}\lvert\partial_y^\alpha\partial_\eta^\beta R_0(y,\eta)\rvert\le h^{3/2} C_{\alpha\beta}(1+2h^{-1/10})^3\le C_{\alpha\beta}' h^{6/5},\quad 0<h<1.
$$
Hence, $h^{3/2}J_1^w R_0^w J_1^w:L^2(\RR;\CC^2)\to L^2(\RR;\CC^2)$ is $\mathcal{O}(h^{6/5})$ by the Weyl calculus.
\end{proof}

To shorten notation below we will always understand $G^w$ and $J_1^w$ to mean non-semiclassical Weyl quantizations in the variable $y$ (as in, e.g., $G^w(y,D_y)$), while $\chi^w$
is understood as the semiclassical Weyl quantization $\chi^w(x,hD_x)$ in the variable $x$.

We begin by establishing an upper bound.

\begin{prop}\label{prop:upper}
With notation and assumptions as in Theorem \ref{thm:stability}, for $n$ fixed and $h$ small, $T_{\operatorname{mass}}^w$ has at least $n$ eigenvalues and
\begin{equation}\label{eq:upper}
\varlimsup_{h\to0^+} \lambda_n(h)/h\le e_n.
\end{equation}
\end{prop}

\begin{proof}
Define 
\begin{equation}\label{eq:psin}
\psi_{2n+k}(y)=
J_1^w(y,D)\varphi_n(y)\begin{pmatrix}0&1\\1&0\end{pmatrix}^{k-1}(1,0)^t,\quad n\in\NN_0,\quad k=1,2.
\end{equation}
We claim that
\begin{equation}\label{eq:bsimon1}
(\psi_n,\psi_m)=\delta_{nm}+\mathcal{O}_{nm}(h^\infty)
\end{equation}
where $\delta_{nm}$ is the Kronecker delta. Clearly $(\psi_{2n},\psi_{2m+1})=0$ for all $n$ and $m$, and for pairings where both indices are either even or odd the claim follows from the definitions of $\varphi_n$ and $J_1$.
In fact, for $2n+k,2m+k$ both even or both odd we get 
$$
(\psi_{2n+k},\psi_{2m+k})=(\varphi_n,\varphi_m)-((1-J_1^w)\varphi_n,\varphi_m)-(J_1^w\varphi_n,(1-J_1^w)\varphi_m)
$$
where the first term on the right equals $\delta_{nm}$. 
We have 
$$
(1-J_1^w)\varphi_n(y)=\frac{1}{2\pi}\int e^{i(y-s)\eta}(1-J_1((y+s)/2,\eta))\varphi_n(s)\,ds\,d\eta
$$
where $\lvert h^{1/10}\eta\rvert\ge1$ if $1-J_1((y+s)/2,\eta)\not\equiv0$ due to \eqref{def:J1} and the definition of $J$. A standard integration by parts using $ih^{1/10}(h^{1/10}\eta)^{-1}\partial_s e^{i(y-s)\eta}=e^{i(y-s)\eta}$ then shows that
$$
\lvert((1-J_1^w)\varphi_n,\varphi_m)\rvert\le C_kh^{k/10}\lVert \varphi_n\rVert_{H^k(\RR)}\lVert \varphi_m\rVert_{L^2(\RR)}
$$
for $k\ge0$. Since $J_1^w$ is bounded on $L^2$ the same arguments show that
$$
\lvert(J_1^w\varphi_n,(1-J_1^w)\varphi_m)\rvert\le C_k\lVert \varphi_n\rVert_{L^2(\RR)}h^{k/10}\lVert \varphi_m\rVert_{H^k(\RR)}
$$
for $k\ge0$ which proves the claim. 

We also claim that
\begin{equation}\label{eq:bsimon2}
h^{-1}(T_\mathrm{mass}^w\psi_n,\psi_m)=e_n(\psi_n,\psi_m)+\mathcal{O}_{nm}(h^{1/5}).
\end{equation}
We prove this when $n,m$ are both odd (the case when they are both even is similar and when one is even and one is odd it is trivial). We have
\begin{align*}
h^{-1}(T^w\psi_{2n+1},\psi_{2m+1})&=(T_0^wJ_1^w(\varphi_n,0)^t,J_1^w(\varphi_m,0)^t)+h^\frac12(J_1^wR_0^wJ_1^w(\varphi_n,0)^t,(\varphi_m,0)^t),
\end{align*}
where the second term on the right is $\mathcal{O}(h^{1/5})$ by Lemma \ref{lem:remainderboundL2}.
To analyze the first term on the right we note that
$$
T_0^w J_1^w=J_1^wT_0^w+2[D_y,J_1^w]D_y+[D_y,[D_y,J_1^w]]+2\omega^2[y,J_1^w]y+\omega^2[y,[y,J_1^w]].
$$
Since $J_1^wT_0^w(\varphi_n,0)^t=e_n^1\psi_{2n+1}=e_{2n+1}\psi_{2n+1}$ by \eqref{eq:eigenvalueordering} we find, by using the Weyl calculus to compute the commutators, that
\begin{align*}
(T_0^wJ_1^w(\varphi_n,0)^t,J_1^w(\varphi_m,0)^t)&=e_{2n+1}(\psi_{2n+1},\psi_{2m+1})\\& \quad
+2((D_yJ_1)^wD_y\varphi_n,J_1^w\varphi_m) +((D_y^2J_1)^w\varphi_n ,J_1^w\varphi_m)\\& \quad
-2\omega^2((D_\eta J_1)^wy\varphi_n,J_1^w\varphi_m) +\omega^2((D_\eta^2J_1)^w\varphi_n ,J_1^w\varphi_m)
\end{align*}
where it is easy to see that the last four terms on the right are $\mathcal{O}_{nm}(h^\infty)$ since $\partial_y^kJ_1(y,\eta)=0$ when $\lvert y\rvert\le h^{-1/10}$ and $\partial_\eta^kJ_1(y,\eta)=0$ when $\lvert \eta\rvert\le h^{-1/10}$  for $k\ge1$. 
Hence,
\begin{equation*}
h^{-1}(T^w\psi_{2n+1},\psi_{2m+1})=e_{2n+1}(\psi_{2n+1},\psi_{2m+1})+\mathcal{O}_{nm}(h^{1/5}).
\end{equation*}
By arguments similar to those used to obtain \eqref{eq:bsimon1} we see that we may replace $T^w$ by $T_\mathrm{mass}^w=T^w+(1-G^w)\id_{\CC^{2\times 2}}$ by absorbing the term in the remainder, which gives \eqref{eq:bsimon2}.
In fact, the symbol $1-G$ is bounded on $\RR$ and on $\supp(1-G)$ we have $\lvert h^{1/2}\eta\rvert\ge\const$ by \eqref{eq:suppcondG} so an integration by parts using $ih^\frac12(h^\frac12\eta)^{-1}\partial_s e^{i(y-s)\eta}=e^{i(y-s)\eta}$ gives
\begin{equation}\label{eq:negligible}
\lvert((1- G^w)\psi_n,\psi_m)\rvert\le C_kh^{k/2}\lVert \psi_n\rVert_{H^k(\RR)}\lVert \psi_m\rVert
\end{equation}
for $k\ge0$. This proves the claim.

As in \cite{simon1983semiclassical} (see also \cite[Chapter 11]{cycon2009schrodinger}), we conclude from \eqref{eq:bsimon1}, \eqref{eq:bsimon2} and the Rayleigh-Ritz principle proved in Lemma \ref{thm:rayleighritz} that $T_\mathrm{mass}^w$ has at least $n$ eigenvalues, and that if $\lambda_n(h)$ denotes the $n$:th eigenvalue counting multiplicity, then \eqref{eq:upper} holds.
\end{proof}

We now turn to a lower bound, which combined with Proposition \ref{prop:upper} gives Theorem \ref{thm:stability}.

\begin{prop}\label{prop:lowerbound}
With notation and assumptions as in Theorem \ref{thm:stability}, we have
$$
\varliminf_{h\to0^+}\lambda_n(h)/h\ge e_n.
$$
\end{prop}

For the proof it will be more convenient to work with the unitarily equivalent $P_\mathrm{mass}^w(x,hD)$ rather than $T_\mathrm{mass}^w(y,D)$, where we will use a pseudodifferential version of the IMS localization formula (so dubbed by Barry Simon \cite{simon1983semiclassical} after Ismagilov, Morgan, Simon and I. M. Sigal).
To state it we let
\begin{equation}\label{eq:chi1}
\chi_1(x,\xi)=J_1(h^{-\frac12}x,h^{-\frac12}\xi)=J(h^{-2/5}x)J(h^{-2/5}\xi),
\end{equation}
so that $\chi_1$ is supported for $\lvert x\rvert,\lvert \xi\rvert\le 2h^{2/5}$ with $\chi_1\equiv 1$ when $\lvert x\rvert,\lvert \xi\rvert\le h^{2/5}$.\footnote{The precise value $2/5$ of the exponent is not important -- what is needed is that $\chi_1$ is supported in a ball or radius $\sim h^\nu$ for some $\frac13<\nu<\frac12$.}
With respect to the standard rescaling we have on operator level that 
$$
J_1^w(y,D_y)=(\gamma^{-1})^*\circ\chi_1^w(x,hD_x)\circ\gamma^*.
$$
We observe that $\chi_1\in S_{2/5}^{0,-\infty}(T^*\RR)$, so $\chi_1^w(x,hD)\in \Psi_{2/5}^{0,-\infty}(\RR)$.
Next, define $\chi_0\in C^\infty(T^*\RR)$ by the condition that 
\begin{equation}\label{eq:pou}
(\chi_0(x,\xi))^2+(\chi_1(x,\xi))^2=1.
\end{equation}
By definition we then have
\begin{equation}\label{eq:suppchi0}
(x,\xi)\in\supp{\chi_0}\quad \Longrightarrow\quad\lvert x\rvert,\lvert\xi\rvert\ge h^{2/5}.
\end{equation}

\begin{lemm}[IMS]\label{lem:IMS}
Let $\chi_0$ and $\chi_1$ be as above. Then there are $X_0^w(x,hD)\in\Psi_{2/5}^{0,0}(\RR)$ and $X_1^w(x,hD)\in\Psi_{2/5}^{0,-\infty}(\RR)$ such that 
$X_j=\chi_j$ modulo $S^{-\infty,-\infty}(T^*\RR)$ and
$$\displaystyle P_{\operatorname{mass}}^w=\sum_{k=0}^1 X_k ^w  P_{\operatorname{mass}}^w  X_k^w+\mathcal{O}_{L^2(\RR)\to L^2(\RR)}(h^{6/5}).$$
\end{lemm}

As with $\chi^w$ we shall, to shorten notation, always understand $X_0^w$ and $ X_1^w$ to mean semiclassical Weyl quantizations in the variable $x$ (as in, e.g., $X_0^w(x,hD)$).

Recall from \eqref{eq:Texpansion} and \eqref{eq:normalformpullback2} that $P^w=\gamma^* \circ(hT_0^w+h^{3/2}R_0^w)\circ(\gamma^{-1})^*$ and let us write $H_0^w=\gamma^*\circ (hT_0^w)\circ(\gamma^{-1})^*$, so that $H_0^w$ has a complete set of eigenfunctions $\{\phi_k(h)\}_{n=1}^\infty$ given by $\phi_k(h)=h^{-\frac14}\gamma^*\tilde\varphi_n$, where 
\begin{equation}\label{eq:eigenbasisH0}
\tilde\varphi_{2n+k}(y)=\varphi_n(y)\sigma_1^{k-1}(1,0)^t
\end{equation}
for $n\in\NN_0$ and $k=1,2$.
By Lemma \ref{lem:IMS} we then have
\begin{equation}\label{eq:beforesummingup}
P_\mathrm{mass}^w= X_0^w P_\mathrm{mass}^w  X_0^w+ X_1^w( P_\mathrm{mass}^w-H_0^w) X_1^w+ X_1^w H_0^w X_1^w+\mathcal{O}(h^{6/5}).
\end{equation}
The middle term is $\mathcal O(h^{6/5})$, too:

\begin{lemm}\label{lemm:middleterm}
\label{lem:remainderboundL2periodic}
For $X_1^w$ as in Lemma \ref{lem:IMS} and $H_0^w$ as above, we have
$\displaystyle
\lVert  X_1^w( P_\mathrm{mass}^w-H_0^w) X_1^w\rVert =\mathcal{O}(h^{6/5})
$
in $L^2(\RR;\CC^2)$.
\end{lemm}

We postpone the proofs of Lemmas \ref{lem:IMS} and \ref{lemm:middleterm} to Appendix \ref{sec:auxLemma}, and recall the following sharp semiclassical G\aa rding inequality for systems. This can for example be obtained from H\" ormander's Weyl calculus \cite[Theorem 6.8]{hormander1979weyl}, as we describe below for the reader's convenience.

\begin{lemm}\label{lem:Gårding}
Let $p\in S(1)$ and assume that $p(x,\xi)$ is a positive semi-definite matrix for all $(x,\xi)$. Then
\begin{equation}\label{eq:Gårding}
(p^w(x,hD)u,u)\ge-Ch\lVert u\rVert^2
\end{equation}
for some constant $C$.
\end{lemm}

\begin{proof}[Sketch of proof]
Use the standard rescaling $y=h^{-1/2}x$ and write $p^w(x,hD_x)u(x)=p_h^w(y,D_y)v(y)$ where $p_h(y,\eta)=p(h^\frac12 y,h^\frac12 \eta)$ and $u(x)=v(y)$. Then $p_h\in S(1)$ uniformly for $0<h<1$,  see the discussion on page \pageref{eq:hboundedset}. 
Let $g$ be the metric
$$
g_{x,\xi}(dx,d\xi)=h\lvert dx\rvert^2+h\lvert d\xi\rvert^2
$$
and let $g^\sigma$ be the dual metric of $g$ with respect to the symplectic form $\sigma$. The Weyl calculus then gives that $\sup g_{x,\xi}/g_{x,\xi}^\sigma=h^2$ and it is easy to see that $p_h$ belongs to the Weyl symbol class $S(1,g)$, see \cite[Definition 2.3]{hormander1979weyl}. Since $p_h/h\in S(1/h,g)$ is also positive semi-definite, we can apply \cite[Theorem 6.8]{hormander1979weyl} to $p_h/h$ and obtain a constant $C>0$ such that
\begin{equation*}
(p_h^w(y,D_y)v,v)\ge -Ch\lVert v\rVert^2.
\end{equation*}
A simple calculation then gives \eqref{eq:Gårding}.
\end{proof}

With these preparations at hand, we are now able to give the proof of the lower bound on the eigenvalue asymptotics.

\begin{proof}[Proof of Proposition \ref{prop:lowerbound}]
We first assume that $n\ge2$. It then suffices to prove the proposition when $n$ is even. Indeed, 
suppose it holds for even $n$, and recall from \eqref{eq:eigenvalueordering} that $e_{2m}=e_{2m+1}$ for $m\ge1$. Then
$$
\varliminf_{h\to0^+}\lambda_{2m+1}/h\ge \varliminf_{h\to0^+}\lambda_{2m}/h\ge e_{2m}=e_{2m+1},\quad m\ge1,
$$
which proves the claim. 

We will prove the statement in the proposition with $n$ replaced by $n+1$, so suppose therefore that $n$ is odd, and fix a number $e_n<r<e_{n+1}$ and let $P_n$ be the projection onto the eigenvalues below $rh$ for $H_0^w$ so that $P_n$ has rank $n$. It is then easy to see that
\begin{equation}\label{eq:lowerbound1}
 X_1^w H_0^w X_1^w\ge  X_1^w (H_0^w -hr)P_n  X_1^w+hr( X_1^w)^2.
\end{equation}

Next, note that if $P\in S(1)$ then the symbol of $P_\mathrm{mass}^w$ is semi-bounded from below by $Ch^{4/5}+\mathcal O(h)$ on $\supp\chi_0$ by \eqref{eq:lowerboundwithchi} and \eqref{eq:suppchi0} for some $C>0$.  
By a standard pseudodifferential cutoff argument together with the fact that $X_0^w=\chi_0^w$ mod $\Psi^{-\infty,-\infty}$, an 
 application of the sharp G\aa rding inequality for systems (Lemma \ref{lem:Gårding}) 
 then gives
$$
((P_\mathrm{mass}^w-Ch^{4/5}+\mathcal O(h))X_0^wu,X_0^wu)\ge -\widetilde Ch(X_0^wu,X_0^wu).
$$
For small $h$ we thus have
\begin{equation}\label{eq:lowerbound2}
X_0^wP_\mathrm{mass}^wX_0^w\ge hr (X_0^w)^2
\end{equation}
when $P\in S(1)$.
When $P\in S^2(T^*\RR)$ we instead use \eqref{eq:lowerboundwithchi2} to get
\begin{equation*}
(P_\mathrm{mass}^w X_0^wu,X_0^wu)\ge ((V+1-\chi^w)X_0^wu,X_0^wu)-\widetilde Ch(X_0^wu,X_0^wu),
\end{equation*}
where $V+1-\chi\in S(1)$ satisfies $V(x)+1-\chi(x,\xi)\ge Ch^{4/5}$ for $(x,\xi)\in\supp{\chi_0}$ by \eqref{eq:lowerboundwithchi3} and \eqref{eq:suppchi0} for some $C>0$.
We may thus apply Lemma \ref{lem:Gårding} to conclude that 
$$
((V+1-\chi^w-Ch^{4/5})X_0^wu,X_0^wu)\ge -\widetilde Ch(X_0^wu,X_0^wu)
$$
for some new $\widetilde C$.
This implies that \eqref{eq:lowerbound2} holds for $h$ small also when $P\in S^2(T^*\RR)$.

Summing up using \eqref{eq:beforesummingup} and $( X_0^w)^2+( X_1^w)^2=1$ mod $\Psi^{-\infty,-\infty}$ we get
$$
P_\mathrm{mass}^w\ge hr+R+o(h)
$$
where $R= X_1^w(H_0^w -hr)P_n X_1^w$ has rank at most $n$. We can then pick $\psi$ in the span of the first $n+1$ eigenvectors of $P_\mathrm{mass}^w$ with $\lVert\psi\rVert=1$ and $\psi\in\ker R$. Then
$$
\lambda_{n+1}\ge (P_\mathrm{mass}^w\psi,\psi)\ge hr+o(h)
$$
and, since $r\in(e_n,e_{n+1})$ was arbitrary, this shows that $\varliminf_{h\to0^+}\lambda_{n+1}/h\ge e_{n+1}$.

It remains to consider $n=1$. However, inspecting the arguments above we see that if we fix $r<e_1=0$ then \eqref{eq:lowerbound1}--\eqref{eq:lowerbound2} hold trivially, and $R=0$, so $\lambda_1\ge hr+o(h)$ from which the result easily follows.
\end{proof}

By combining the ideas used in the proof of Corollary \ref{cor:periodicquasimodesnormalform} with the method used to prove Theorem \ref{thm:stability} it is possible to obtain a stability result for the eigenvalues of  $P_{\operatorname{mass}}^w$ also in the periodic setting when $P^w$ is viewed as a densely defined operator on $L^2(\mathbb T)$. Since it might be of independent interest we state such a result here but leave the proof to the interested reader.

\begin{thm}\label{thm:stability2}
Let $\Lambda_n(h)$ be the $n$:th eigenvalue, counting multiplicity, of $P_{\operatorname{mass}}^w$ viewed as a densely defined operator on $L^2(\mathbb T)$. Let $e_n$ be the $n$:th eigenvalue, counting multiplicity, of $T_0^w$ viewed as a densely defined operator on $L^2(\RR)$. Then for $n$ fixed and $h$ small, $P_{\operatorname{mass}}^w$ has at least $n$ eigenvalues and
$$
\lim_{h\to0^+} \Lambda_n(h)/h=e_n.
$$
\end{thm}

\subsection{Asymptotic series}
We now use Theorem \ref{thm:stability} to prove asymptotic expansions of eigenvalues and eigenvectors of $T^w(y,D)$ on $\RR$.
We first consider the massive operator $T_\mathrm{mass}^w=T^w(y, D)+1 -G^w(y,D)$. We assume that the Weyl symbol $T$ satisfies the conditions in Proposition \ref{prop:normalform}. In particular, writing $z=(y,\eta)$ we then have
$$
T(z) =h (T_0(z)+Q_m(z)+R_m(z)),
$$
with matrix norm estimates
\begin{equation}\label{eq:sizeZ}
\lVert Q_m(z) \rVert = \mathcal{O}(h^{1/2}\langle z \rangle^m) \quad\text{ and }\quad \lVert R_m(z) \rVert = \mathcal{O}( (h^{1/2}\langle z \rangle)^{m+1}).
\end{equation}

Recall that $T_0(z)=(\eta^2+\omega^2y^2)\id_2+\omega\diag(-1,1)$ and let $e_n$ be the $n$:th eigenvalue of $T_0^w$ on $\RR$. Choose $\varepsilon$ (depending only on $\omega$) so that for each $m$, either $e_m=e_n$ or $\lvert e_m-e_n\rvert\ge \varepsilon$. Let
\begin{equation}\label{eq:projection}
\Pi_n(h) = \frac{1}{2\pi i} \oint_{\partial B_{\varepsilon}(e_n)} (\zeta-T_\mathrm{mass}^w(y,D)/h)^{-1} d\zeta
\end{equation}
be the projection onto the span of all eigenvectors with eigenvalues such that $\lambda_m(h)/h\to e_n$. (Since $e_1$ is simple while $e_n$ is double for $n\ge2$, we then have $m=n$ or $m=n\pm 1$.) Let $\psi_{2n+k}=J_1^w\varphi_{n}\sigma_1^{k-1}(1,0)^t$ in accordance with \eqref{eq:psin}. We will use a version of \cite[Theorem 2.3]{simon1983semiclassical} proved in the same way which we state here using our notation.

\begin{lemm}\label{lem:perturbation}
$\displaystyle \lVert (1-\Pi_n(h))\psi_n\rVert\to0\text{ as }h\to0$.
\end{lemm}

Let $\tilde\varphi_{2n+k}=\varphi_{n}\sigma_1^{k-1}(1,0)^t$ as in \eqref{eq:eigenbasisH0} so that $T_0^w\tilde\varphi_n=e_n\tilde\varphi_n$ and 
$$
\tilde\varphi_{2n+k}=\psi_{2n+k}+(1-J_1^w)\varphi_{n}\sigma_1^{k-1}(1,0)^t
$$
then $\tilde\varphi_n=\psi_n+\mathcal{O}(h^\infty)$ in $L^2$ so $(1-\Pi_n(h))\tilde\varphi_n\to0$ as $h\to0$ by the lemma. It follows that 
\begin{equation}\label{eq:projectionlimit}
(\tilde\varphi_n,\Pi_n(h)\tilde\varphi_n)
=(\tilde\varphi_n,\tilde\varphi_n)-(\tilde\varphi_n,(1-\Pi_n(h))\tilde\varphi_n)
\to1,\quad h\to0.
\end{equation}

Note that under the assumption in \eqref{eq:eigenvalueordering}, the only simple eigenvalue of $T_0^w$ is $e_1=0$. For other values of $\mu_1,\mu_2$ in \eqref{eq:T0} this is of course not necessarily true. The following statement is written with this more general situation in mind.

\begin{prop}\label{thm:simple}
Let $e_n$ be a simple eigenvalue of $T_0^w$ on $\RR$. Let $\lambda_n(h)$ and $v(n,h)$ be the corresponding eigenvalue and eigenvector of $T_\mathrm{mass}^w$ on $\RR$. Then $\lambda_n(h)\sim he_n+h^{3/2}a^{(0)}_n+h^2a^{(1)}_n+\ldots$ in the sense that 
$$
\lambda_n(h)-he_n-\sum_{i=1}^m h^{(i+2)/2}a_n^{(i)}=\mathcal{O}(h^{(m+3)/2}),
$$
and $v(n,h)=v_0(n)+h^\frac12v_1(n)+hv_2(n)+\ldots$ in the sense that 
$$
v(n,h)-\sum_{i=0}^m h^{i/2}v_i(n)=\mathcal{O}_{\mathscr S}(h^{(m+1)/2}),
$$ 
where $v_0(n)=\varphi_{n',\omega}\sigma_1^{k-1}(1,0)^t$ with $n'\in\NN_0$ and $k\in\{1,2\}$ determined by $n=2n'+k$.
Moreover,  $v(n,h)\in\mathscr S$ and each $v_i(n,y;h)$ is a polynomial in $y$ times $e^{-\omega y^2/2}$.
\end{prop}

\begin{proof}
Let $\varphi\in L^2(\RR;\CC^2)$ be the eigenfunction corresponding to the simple eigenvalue $e_n$ of $T_0^w$ on $\RR$, i.e., $\varphi:=\tilde\varphi_n$ in the notation above. Then $B_{\varepsilon}(e_n)$, with $\varepsilon$ independent of $h$, contains precisely one eigenvalue and 
\[ \Pi(h) = \frac{1}{2\pi i} \oint_{\partial B_{\varepsilon}(e_n)} (\zeta-T_\mathrm{mass}^w(y,D)/h)^{-1} d\zeta\]
is the projection onto an eigenfunction $v(h) = \Pi(h) \varphi/\sqrt{\langle \varphi, \Pi(h) \varphi \rangle}$ corresponding to a single eigenvalue 
$$
\lambda (h)= \frac{\langle h^{-1} T_\mathrm{mass}^w \varphi, \Pi(h) \varphi \rangle}{\langle \varphi, \Pi(h) \varphi \rangle}
$$
of $T_\mathrm{mass}^w/h$ for $h$ small enough, see Theorem \ref{thm:stability}. The denominator is non-vanishing by \eqref{eq:projectionlimit}.
Due to our assumptions on $T$ we clearly have an asymptotic expansion of $T^w\varphi$, so in view of the expressions for $\Pi(h)$, $v(h)$ and $\lambda (h)$ we see that if we obtain an asymptotic expansion for $(T_\mathrm{mass}^w/h-\zeta)^{-1}\varphi$ which is uniform in $\zeta$ then we also get asymptotic expansions for $v(h)$ and $h\lambda (h)$. (The contribution of $h^{-1}(1-G^w) \varphi$ to $\lambda (h)$ is negligible by \eqref{eq:negligible}.)

We then use the standard geometric series 
\begin{equation}\label{eq:geometricseries}
(T_\mathrm{mass}^w/h-\zeta)^{-1} \varphi = \sum_{i=0}^m \psi_i + r_m,
\end{equation}
where $\psi_i =(-1)^i (T_0^w-\zeta)^{-1}((Z^w+(1-G^w)/h)(T_0^w-\zeta)^{-1})^i \varphi$ with $Z=T/h-T_0$, while
$$
r_m=(-1)^{m+1}(T^w/h-\zeta)^{-1} ((Z^w+(1-G^w)/h)(T_0^w-\zeta)^{-1})^{m+1}\varphi.
$$
As above we find by \eqref{eq:negligible} that the terms involving $1-G^w$ in $\psi_1,\ldots,\psi_m$ and $r_m$ all belong to $\mathscr S$ by the pseudodifferential calculus and are $\mathcal{O}(h^\infty)$ there uniformly for $h$ small, i.e., we can redefine $\psi_1,\ldots,\psi_m$ and $r_m$ so that \eqref{eq:geometricseries} holds with
$$
\psi_i =(-1)^i (T_0^w-\zeta)^{-1}(Z^w(T_0^w-\zeta)^{-1})^i \varphi
$$
and
$$
r_m=(-1)^{m+1}(T^w/h-\zeta)^{-1} (Z^w(T_0^w-\zeta)^{-1})^{m+1}\varphi+\tilde r_m,\quad \tilde r_m=\mathcal{O}_{\mathscr S}(h^\infty).
$$

Let $b\ge m+1$, then from \eqref{eq:sizeZ} follows that $Z^w(\langle z\rangle^{-b})^w$ is bounded on $L^2(\RR)$ with norm of size $\sqrt h$. We then rewrite $r_m$ using $Z_k^w:=Z^w(\langle z \rangle^{-kb})^w$ as
\begin{align*}
r_m -\tilde r_m&= (-1)^{m+1} (T^w/h-\zeta)^{-1} Z_1^w \Big[(\langle z \rangle^b)^w  (T_0^w-\zeta)^{-1} Z_2^w\Big]
\times \cdots \\ &\quad\cdots \times\Big[(\langle z \rangle^{mb})^w  (T_0^w-\zeta)^{-1} Z_{m+1}^w\Big]
(\langle z \rangle^{(m+1)b})^w (T_0^w-\zeta)^{-1}\varphi
\end{align*}
where the last factor $(\langle z \rangle^{(m+1)b})^w (T_0^w-\zeta)^{-1}\varphi$ belongs to $\mathscr S(\RR)$ since $\varphi $ is Schwartz.
Moreover, for $1\le i\le m$,
$$
(\langle z \rangle^{ib})^w  (T_0^w-\zeta)^{-1} Z_{i+1}^w=
(\langle z \rangle^{ib})^w  (T_0^w-\zeta)^{-1} \big[Z^w(\langle z\rangle^{-b})^w\big](\langle z\rangle^{-ib})^w
$$ is a bounded (even compact) operator on $L^2(\RR)$ by the pseudodifferential calculus, with norm of size $\sqrt h$. Hence, $r_m=\mathcal{O}_{\mathscr S}(h^{(m+1)/2})$.
If we then define $\psi_i'$ just like $\psi_i$ but with $Z^w$ replaced by $Q_m^w$, we get
\[   \psi_i' =(-1)^i (T_0^w-\zeta)^{-1}(Q_m^w(T_0^w-\zeta)^{-1})^i \varphi,\]
with $\lVert\psi_i'\rVert=\mathcal{O}(h^{i/2})$.
Then, by comparing the difference of $\psi_i$ and $\psi_i' $ we find by using $Z-Q_m =R_m$ that $\psi_i-\psi_i'$ is a finite sum of terms of the form
$$
(-1)^i(T_0^w-\zeta)^{-1}A_1^w(T_0^w-\zeta)^{-1}\cdots A_i ^w(T_0^w-\zeta)^{-1}\varphi
$$
where $A_k^w=R_m^w$ for at least one $1\le k\le i$. By \eqref{eq:lowerbound1}, each term above defines an element in $\mathscr S$ with $L^2$ norm of size $h^{(m+1)/2}$ since $R_m$ is order $h^{(m+1)/2}$. Hence, $ \lVert \psi_i-\psi_i' \rVert = \mathcal{O}(h^{(m+1)/2})$. By setting
$$
v_i(n,y;h)=-\frac{1}{2\pi i} \oint_{\partial B_{\varepsilon}(e_n)} \psi_i'(y;\zeta)/h^{i/2} d\zeta
$$
we obtain the desired expansions of $v(n,h)$.

It remains to prove that each $v_i(n,y;h)$ is a polynomial times $e^{-\omega y^2/2}$. To see that, note that $\langle z\rangle^w$ maps a polynomial times $e^{-\omega y^2/2}$ onto a polynomial times $e^{-\omega y^2/2}$. The same is true for $Q_m^w$ and also for $(T_0^w-\zeta)^{-1}$. Indeed, if $p$ is a polynomial and $(T_0^w-\zeta)^{-1}(p(y)e^{-\omega y^2/2},0)^t=(\psi(y;\zeta),0)^t$ then with $p(y)e^{-\omega \omega y^2/2}=\sum_{n=1}^N a_n\varphi_{n,\omega}(y)$ and $\psi(y;\zeta)=\sum_{n=1}^\infty b_n(\zeta)\varphi_{n,\omega}(y)$ we get
$$
\sum_{n=1}^N a_n\varphi_{n,\omega}=pe^{-\omega y^2}=(T_0^w-\zeta)\sum_{n=1}^\infty b_n(\zeta)\varphi_{n,\omega}=\sum_{n=1}^\infty b_n(\zeta)((2n+1)\omega+\mu_1 -\zeta )\varphi_{n,\omega}
$$
which, for $\zeta\in\partial B_\varepsilon(e_n)$, is only possible if $b_n(\zeta)\equiv 0$ for all $n\ge N'$ with $N'=N'(\mu_1)$ independent of $\zeta$. Hence $\psi$ equals $e^{-\omega y^2/2}$ times a polynomial of degree bounded independently of $\zeta$. It follows that each $\psi_i'$ and therefore also each $v_i(n,y;h)$ is a polynomial times $e^{-\omega y^2/2}$.
\end{proof}

We now turn to degenerate eigenvalues.

\begin{prop}\label{thm:doubleEV}
Let $e_n$ be an eigenvalue of $T_0^w$ on $\RR$ of multiplicity 2 with $e_n=e_{n+1}$ and let $\lambda _n(h),\lambda_{n+1} (h)$ be the eigenvalues of $T_\mathrm{mass}^w$ on $\RR$ which, when divided by $h$, tend to $e_n$. Then for $j\in\{n,n+1\}$ there exists an asymptotic expansion with coefficients $a_j^{(i)}$ of the form $$\lambda _j(h)\sim he_n+h^{3/2}a^{(0)}_j+h^2a^{(1)}_j+\ldots.$$
\end{prop}

\begin{proof}
Let $\tilde\varphi_j$ be as above so they span the eigenspaces of $T_0^w$ on $\RR$, and let $\Pi_n(h)$ be the projection  \eqref{eq:projection} onto the span of all eigenvectors of $T_\mathrm{mass}^w/h$ associated to eigenvalues approaching $e_n$ as $h\to0$. Since $e_n=e_{n+1}$ we thus have $\Pi_n(h)=\Pi_{n+1}(h)$. By \eqref{eq:projectionlimit} together with the proof of Proposition \ref{thm:simple} we see that $\Delta_{ij}(h):=(\tilde\varphi_i,\Pi_n(h)\tilde\varphi_j)=\delta_{ij}+\mathcal{O}(h^\frac12)$ for $i,j\in\{n,n+1\}$. In particular, $\Pi_n(h)\tilde\varphi_n$ and $\Pi_n(h)\tilde\varphi_{n+1}$ are linearly independent for $h$ small. We then let $\Delta^{-\frac12}$ be the square root of the inverse of $\Delta=(\Delta_{ij})_{i,j=1}^2$ which exists for $h$ small. Since $\Delta_{ij}$ has an asymptotic expansions by the proof of Proposition \ref{thm:simple}, and since the eigenvalues of a Hermitian matrix have asymptotic expansions in $h^{1/2}$ provided that the elements of the matrix do (see \cite[Lemma 5.2]{simon1983semiclassical}), it follows that $\Delta^{-\frac12}$ also has an asymptotic expansion.
(We diagonalize $\Delta^{-1}=UEU^*$ with $E$ diagonal consisting of the eigenvalues of $\Delta^{-1}$ which are positive for $h$ small. Then $\Delta^{-\frac12}=U\sqrt EU^*$.)

Write $C(h):=\Delta^{-\frac12}H\Delta^{-\frac12}$ with $H=(H_{ij})_{i,j=n}^{n+1}$ where $H_{ij}=(h^{-1}T_\mathrm{mass}^w\tilde\varphi_i,\Pi_n(h)\tilde\varphi_j)$ has an asymptotic expansion by the proof of Proposition \ref{thm:simple}. Hence, $C(h)$ has an asymptotic expansion, and thus the eigenvalues of $C(h)$ do as well, which we claim are precisely $\lambda_n$ and $\lambda_{n+1}$. 

We consider two cases for $h$ small but fixed: 

1) $\lambda_n(h)=\lambda_{n+1}(h)$. Then $H=\lambda_n\Delta$ so $C(h)=\lambda_n(h)\id_2$ which proves the claim in this case.

2) $\lambda_n(h)\ne\lambda_{n+1}(h)$. 
Since $\Pi_n(h)$ has rank 2 we can find orthonormal $v(n,h),v_{}(n+1,h)$ such that $\ran\Pi_n=\Span\{v(n),v(n+1)\}$ and $T_\mathrm{mass}^wv_j=\lambda_j v_j$ and thus
\begin{equation}
\label{eq:change_of_basis}
\begin{pmatrix} v(n)\\ v({n+1})\end{pmatrix}=D\begin{pmatrix} \Pi_n(h)\tilde\varphi_n\\ \Pi_n(h)\tilde\varphi_{n+1}\end{pmatrix}
\end{equation}
for some invertible transition matrix $D$. It is straightforward to check that $\id_2= D\Delta D^*$ so $\Delta^{-1}=D^*D$ and thus $\Delta^{-\frac12}=\sqrt{D^* D}$. Writing $D=(d_{ij})_{i,j=1}^2$ and using that $\lambda_i = ( T^w_\mathrm{mass} v_i,v_i)$ and \eqref{eq:change_of_basis}, we get
$$
\lambda_n=\lvert d_{11}\rvert^2 H_{nn}+2\Re(d_{11}\overline{d_{12}}H_{n(n+1)})+\lvert d_{12}\rvert^2 H_{(n+1)(n+1)}
$$
$$
\lambda_{n+1}=\lvert d_{21}\rvert^2 H_{nn}+2\Re(d_{21}\overline{d_{22}}H_{n(n+1)})+\lvert d_{22}\rvert^2 H_{(n+1)(n+1)}.
$$
Similarly, using $0 = ( T^w_\mathrm{mass} v(n),v({n+1})) = ( T^w_\mathrm{mass} v_{n+1},v_{n})$ and \eqref{eq:change_of_basis},
we obtain $\diag(\lambda_n,\lambda_{n+1})=DHD^*$ so
$$
H=D^{-1}\diag(\lambda_n,\lambda_{n+1})(D^{-1})^*.
$$
Hence, with $U=\sqrt{D^*D}D^{-1}$ we have since $\Delta^{-\frac12}=\sqrt{D^*D}$ is self-adjoint that
\begin{equation}\label{eq:C}
C(h)=\Delta^{-\frac12}H\Delta^{-\frac12}=U\diag(\lambda_n,\lambda_{n+1})U^*.
\end{equation}
Now observe that 
\[UU^* = \sqrt{D^*D}D^{-1}( D^*)^{-1}\sqrt{D^*D} = (D^*D)^{1/2}  (D^*D)^{-1/2}  (D^*D)^{-1/2}  (D^*D)^{1/2} = \id\] and also  
\[U^*U= (D^*)^{-1}\sqrt{D^*D} \sqrt{D^*D}D^{-1} = (D^*)^{-1}(D^*D)D^{-1} = \id.\]
So $U$ is unitary and therefore $U^*=U^{-1}$, which in view of \eqref{eq:C} means that $C(h)$ has eigenvalues $\lambda_n,\lambda_{n+1}$.
\end{proof}

\begin{prop}\label{thm:eigenvectorsofdoubleEV}
Let $e_n$ be an eigenvalue of $T_0^w$ on $\RR$ of multiplicity 2 with $e_n=e_{n+1}$ and let $\lambda _n(h),\lambda_{n+1} (h)$ be the eigenvalues of $T_\mathrm{mass}^w$ on $\RR$ which, when divided by $h$, tend to $e_n$. If $v(n,h)$ and $v(n+1,h)$ are the corresponding eigenvectors of $T_\mathrm{mass}^w$ then $v(n),v({n+1})\in\mathscr S$ and have asymptotic series in $\sqrt h$ to any order, and each term in the expansions is a polynomial in $y$ times $e^{-\omega y^2/2}$.
\end{prop}

\begin{proof}
We recall that $\mu_1-\mu_2=-2\omega$, which, since $e_n=e_{n+1}$, means that $n$ is even so $n=2n'+2$ for some integer $n'$. With $\tilde\varphi_{n}=\tilde\varphi_{2n'+2}=\varphi_{n'}(0,1)^t$ and $\tilde\varphi_{n+1}=\tilde\varphi_{2(n'+1)+1}=\varphi_{n'+1}(1,0)^t$ we then have $T^w_0 \tilde\varphi_{n+j}=e_{n+j}\tilde\varphi_{n+j}$, $j=0,1$, so $\{\tilde\varphi_{n},\tilde\varphi_{n+1}\}$ is an orthonormal basis for the eigenspace of $T_0^w$ associated to the double eigenvalue $e_n$. If $\Pi_n$ is given by \eqref{eq:projection} then $\Pi_n=\Pi_{n+1}$ and as above we have that $\Pi_n\tilde\varphi_n$ and $\Pi_{n}\tilde\varphi_{n+1}$ are linearly independent for $h$ small.

We consider two cases:

1) The asymptotic series for $\lambda_n(h)$ and $\lambda_{n+1}(h)$ provided by Proposition \ref{thm:doubleEV} are not identical. We can then define spectral projections $P_n$ and $P_{n+1}$ of rank 1 onto the span of $v(n,h)$ and $v_{}(n+1,h)$, respectively. Since $\Pi_n$ is the projection onto the span of $\{v(n,h),v_{}(n+1,h)\}$ and $\Pi_n\tilde\varphi_n$ and $\Pi_{n}\tilde\varphi_{n+1}$ are linearly independent for $h$ small, we must have $P_n\tilde\varphi_i\ne0$ and $P_{n+1}\tilde\varphi_j\ne0$ for some $i,j\in\{n,n+1\}$. Indeed, if for example $P_n\tilde\varphi_n=P_n\tilde\varphi_{n+1}=0$ then $\Pi_n\tilde\varphi_n=P_{n+1}\tilde\varphi_n$ and $\Pi_n\tilde\varphi_{n+1}=P_{n+1}\tilde\varphi_{n+1}$ are linearly dependent, a contradiction. For the same reason we cannot have $P_{n+1}\tilde\varphi_n=P_{n+1}\tilde\varphi_{n+1}=0$. Hence, for $j=n,n+1$ we have $v_j\in\Span\{P_j\varphi_j\}$ for some eigenvector $\varphi_j\in\{\tilde\varphi_n,\tilde\varphi_{n+1}\}$ of $T_0^w$. We now obtain an asymptotic expansion of $P_j\varphi_j$ by arguments similar to those in the proof of Proposition \ref{thm:simple}, which gives the desired expansion of  $v_j$.

2) The asymptotic series for $\lambda_n(h)$ and $\lambda_{n+1}(h)$ are identical. In this case any vector in $\ran\Pi_n$ is an approximate eigenvector of both $\lambda_n(h)$ and $\lambda_{n+1}(h)$ to any order.
Since $\Pi_n\tilde\varphi_n$ and $\Pi_{n}\tilde\varphi_{n+1}$ is a basis for $\ran\Pi_n$ for $h$ small, and both $\Pi_n\tilde\varphi_n$ and $\Pi_{n}\tilde\varphi_{n+1}$ have asymptotic expansions by the proof of Proposition \ref{thm:simple} we obtain the expansion in this case as well.

Finally we note by the above that the eigenvectors are spectral projections of $\tilde\varphi_j$, so as in the proof of Proposition \ref{thm:simple} we find that $v_j\in\mathscr S$ and each term in the expansion is a polynomial times $e^{-\omega y^2/2}$. 
\end{proof}

We can now give

\begin{proof}[Proof of Theorem \ref{thm:expansionsintro}]
As before we assume that $\mu_j=(-1)^j\omega$ for notational simplicity.
For $n\in\NN_0$ and $j\in\{1,2\}$ let $e_{2n+j}=e_{n}^j$ be the eigenvalues of $T_0^w$ arranged as in \eqref{eq:eigenvalueordering}, with $e_1$ simple and $e_n$ double for $n\ge2$. Let $\lambda_1(h)$ be the eigenvalue of $T_\mathrm{mass}^w$ tending, after division by $h$, to $e_1$ as $h\to0$, and let $v_1\in\mathscr S$ be the corresponding eigenvector. Also let $\lambda_{2n},\lambda_{2n+1}$ be the eigenvalues of $T_\mathrm{mass}^w$ tending, after division by $h$, to $e_{2n}=e_{2n+1}$ as $h\to0$, and let $v_{2n},v_{2n+1}\in\mathscr S$ be the corresponding eigenvectors. Now, using integration by parts as in the proof of \eqref{eq:negligible} we find that $T_\mathrm{mass}^w v_{2n+j}=T^w v_{2n+j}+\mathcal{O}_{\mathscr S}(h^\infty)$ since $v_{2n+j}\in\mathscr S$. Hence,
$$
(T^w-\lambda_{2n+j})v_{2n+j}=\mathcal O_{\mathscr S}(h^\infty),
$$
and since $T^w=hT_0^w+h^{3/2}R_0^w$ contains a factor $h$ we then find
by Propositions \ref{thm:simple}--\ref{thm:eigenvectorsofdoubleEV} that
$$
\lambda_{2n+j}v^{(j)}(n)=T^w \sum_{i=0}^{2\ell}h^{i/2}v_i^{(j)}(n)+\mathcal{O}_{\mathscr S}(h^{\ell+3/2})
$$
where $v_i^{(j)}(n)$ is a polynomials times $e^{-\omega y^2/2}$,
and
$$
\lambda_{2n+j}v^{(j)}(n)=\sum_{i=0}^{2\ell} h^{(i+2)/2}\lambda_i^{(j)}(n) \sum_{i'=0}^{2\ell}h^{i'/2}v_{i'}^{(j)}(n)+\mathcal{O}_{\mathscr S}(h^{\ell+3/2})
$$
with $\lambda_0^{(j)}(n)$ and $v_0^{(j)}(n)$ having leading asymptotics as in the statement. Hence,
$$
T^w\sum_{i=0}^{2\ell}h^{i/2}v_i^{(j)}=\lambda^{(j)}(n)\sum_{i=0}^{2\ell}h^{i/2}v_i^{(j)}(n)+\mathcal{O}_{\mathscr S}(h^{\ell+3/2})
$$
where $\lambda^{(j)}(n)=\sum_{i=0}^{2\ell} h^{(i+2)/2}\lambda_i^{(j)}(n)$ mod $\mathcal O(h^{\ell+3/2})$, and the result follows.
\end{proof}

\section{Wells in chiral strained moir\'e lattices}
\label{sec:chiral}

Here we shall apply Corollary \ref{cor:periodicquasimodesnormalform} to the low-energy model $H_\ch^w(k_x)$ in \eqref{eq:tm20} and the pseudodifferential operator $H_{\operatorname{\Psi DO}}$ in \eqref{eq:HarperPDO}. To do so we must first show that each model can be written in the appropriate normal form, which we will do by verifying the assumptions in Proposition \ref{prop:normalform}.

\subsection{Wells for the chiral low-energy Hamiltonian}\label{ss:lowenergy}
Let us start with $H_\ch^w(k_x)$. In the chiral limit $w_0=0$ we get
\begin{equation}\label{eq:semiclassicalH}
H_\ch^w(k_x) = \begin{pmatrix} 0 & hD+k_x -i k_\perp  & 0 & w_1  U^-(x) \\ hD+k_x +i k_\perp  & 0 & w_1 U^+(x) & 0  \\ 0 & w_1 U^+(x) & 0 & hD+k_x -i k_\perp    \\  w_1 U^-(x) & 0 & hD+k_x +i k_\perp  & 0 \end{pmatrix}
\end{equation}
with $H(k_x)$ understood to be a densely defined operator on  $L^2(\mathbb T)$.

\begin{lemm}\label{lem:reduction0}
Consider the chiral limit $w=(0,w_1)$. Then $H_\ch^w(k_x)$ in \eqref{eq:semiclassicalH} is unitarily equivalent to the system 
\begin{equation}\label{eq:reduction}
 \mathscr L_\ch =\begin{pmatrix} 0& D_{\operatorname{c}} \\D_{\operatorname{c}}^* & 0 \end{pmatrix},\quad D_{\operatorname{c}}=\frac12\begin{pmatrix} i & 1\\ -i &  1 \end{pmatrix} \begin{pmatrix} hD+\mathbf k &w_1U^+(x) \\ w_1U^-(x)&hD+\mathbf k \end{pmatrix}\begin{pmatrix} -i & i\\ 1 &  1 \end{pmatrix}
\end{equation}
where $\mathbf k=k_x+ik_\perp$.
Then $\mathscr L_\ch ^2=\diag(D_\ch D_\ch ^*,D_\ch ^*D_\ch )$ and if $D_\ch D_\ch ^*u = \lambda u$ then $\mathscr L_\ch  v=\pm\sqrt\lambda v$ for $v=(v_1,v_2)^t$ with $v_1 =u$ and $v_2 = \pm \lambda^{-1/2} D_{\operatorname{c}}^* u$.
\end{lemm}

\begin{proof}
Let 
\begin{equation*}
\mathscr U=\begin{pmatrix} 0 & U\\ U & 0 \end{pmatrix},\quad
U=\frac{1}{\sqrt 2}\begin{pmatrix} i & 1\\ -i &  1 \end{pmatrix}.
\end{equation*}
Then $UU^\ast=U^\ast U=\id_2$ and $\mathscr U\mathscr U^\ast=\mathscr U^\ast \mathscr U=\id_4$. By first conjugating by $\diag(1,\sigma_1,1)$ and then conjugating by $\mathscr U$ we see that $H$ is equivalent to
$\mathscr L_\ch $ in \eqref{eq:reduction}. If $D_\ch D_\ch ^*u = \lambda u$ then with $v=(v_1,v_2)^t=(u, \pm \lambda^{-1/2} D_{\operatorname{c}}^* u)^t$ we get
\begin{equation*}
 \mathscr L_\ch v=\begin{pmatrix} 0& D_{\operatorname{c}} \\D_{\operatorname{c}}^* & 0 \end{pmatrix}\begin{pmatrix} u \\ \pm\lambda^{-1/2} D_{\operatorname{c}}^* u \end{pmatrix}=\begin{pmatrix} \pm \lambda^{1/2} u \\D_\ch ^* u\end{pmatrix}=\pm\sqrt\lambda v,
\end{equation*}
as claimed.
\end{proof}

If $\lambda$ is an eigenvalue of $\mathscr L_\ch $ then clearly $\lambda^2$ is an eigenvalue of $(\mathscr L_\ch ^2)_{11}$. In view of the converse correspondence between eigenvalues of $(\mathscr L_\ch ^2)_{11}$ and $\mathscr L_\ch $ given by Lemma \ref{lem:reduction0} we can therefore study the spectrum of $(\mathscr L_\ch ^2)_{11}$ in place of $\mathscr L_\ch $.  We then use the following description of the Weyl symbol of $\mathscr L_\ch ^2$.

\begin{lemm}\label{lem:reduction1}
Let $\mathscr L_\ch $ be given by \eqref{eq:reduction}. Then the square $\mathscr L_\ch ^2$ is the Weyl quantization of the symbol $\sigma(\mathscr L_\ch ^2)=\sigma_0(\mathscr L_\ch ^2)+h\sigma_1(\mathscr L_\ch ^2)+h^2\sigma_2(\mathscr L_\ch ^2)$ where the principal symbol $\sigma_0(\mathscr L_\ch ^2)$ has block-diagonal form
$$
\sigma_0(\mathscr L_\ch ^2)=\begin{pmatrix} P_{11}&0\\0&P_{22}\end{pmatrix}+k_\perp ^2\id_{\CC^{4\times4}}+2k_\perp g\diag(1,-1,1,-1)
$$
with
\begin{equation*}
 P_{11}(x,\xi)=  \begin{pmatrix} \xi^2 + f^2+g^2&  
2i\xi f - 2fg  \\ -2i\xi f-2fg &  \xi^2 + f^2+g^2 \end{pmatrix},\quad P_{22}(x,\xi)=  \begin{pmatrix} \xi^2 + f^2+g^2&  
2i\xi f +2fg  \\ -2i\xi f +2fg  &  \xi^2 + f^2+g^2 \end{pmatrix}
\end{equation*}
where $f(x)=w_1(1-\cos(2\pi x))$ and $g(x)=w_1\sqrt 3\sin(2\pi x)$, 
and with lower order terms
\[ \sigma_1(\mathscr L_\ch ^2)(x,\xi) = 2\xi (k_x/h)\id_4-g'(x)
\operatorname{diag}(\sigma_3,-\sigma_3)-2f(x)(k_x/h) \diag(\sigma_2,\sigma_2)
\]
and $\sigma_2(\mathscr L_\ch ^2)=(k_x/h)^2\id_4$.
\end{lemm}

\begin{proof}
Writing $\mathbf{k}=k_x+ik_\perp$ we find by \eqref{eq:reduction} that $\mathscr L_\ch ^2=\diag(D_\ch D_\ch ^*,D_\ch ^*D_\ch )$ where
\begin{align*}
D_{\operatorname{c}}D_{\operatorname{c}}^* &=\begin{pmatrix} hD +\mathbf{k} +ig&if \\ -if&hD+\mathbf{k} -ig \end{pmatrix}\begin{pmatrix} hD+\bar{\mathbf{k}}  -ig&if \\ -if&hD+\bar{\mathbf{k}} +ig \end{pmatrix}
\\&=\begin{pmatrix} (hD+k_x)^2 +f^2+(g+k_\perp )^2& Q^w-2fg \\ -Q^w-2fg&(hD+k_x)^2 +f^2+(g-k_\perp )^2 \end{pmatrix}+[ig,hD]\sigma_3
\end{align*}
and
\begin{equation}\label{eq:Q}
Q^w(x,hD)u=i((hD+k_x) (f u)+f (hD+k_x)u).
\end{equation}
By using the Weyl calculus and noting that $k_x=\mathcal O(h)$ by \eqref{eq:BFunion} it is now straightforward to check that $\mathscr L_\ch ^2$ is an operator having Weyl symbol as described in the statement. 
\end{proof}

We shall now study the existence of degenerate wells for $\mathscr L_\ch ^2$ in the sense of Definition \ref{def:degeneratewell}.
It is easy to see that the eigenvalues of $\sigma_0(\mathscr L_\ch ^2)_{jj}(x,\xi)$ for $j=1,2$ are given by 
\begin{equation*}
\lambda_\pm(x,\xi)=(\xi^2+f^2+g^2+k_\perp ^2)\pm2\sqrt{\xi^2f^2+g^2(f^2+k_\perp ^2)},
\end{equation*}
where $f(x)=w_1(1-\cos(2\pi x))$ and $g(x)=w_1\sqrt 3\sin(2\pi x)$. The eigenvalues coalesce (i.e., $\lambda_+=\lambda_-$) at $(\frac12,0)$ and $(0,\xi)$ for all $\xi\in\RR$.
If $k_\perp \ne0$ then $\lambda_+$ never vanishes. 
If $k_\perp =0$ then $\lambda_+$ vanishes only at $(0,0)\in T^*\mathbb T$, and the zero set of $\lambda_-$ is
\begin{equation}\label{eq:componentszeroset}
\{(x,\xi):\xi=\pm\sqrt{f(x)^2-g(x)^2},\ f^2-g^2\ge0\}
\end{equation}
which has two connected components: the origin and a closed curve in $T^*\mathbb T$ connecting $(\frac13,0)$ with $(-\frac13,0)$ which scales with $w_1$, see Figure \ref{fig:detP}.
Since 
\begin{equation}\label{eq:isolatedzero}
f(x)^2-g(x)^2=-12w_1^2\pi^2x^2+\mathcal{O}(x^4)
\end{equation}
the point $(0,0)$ is an isolated zero of 
$\lambda_-$.
In particular, if $k_\perp \ne0$ then the eigenvalues are distinct near all points in the characteristic set of $\lambda_-$, and if $k_\perp =0$ then the eigenvalues are distinct near all points in the characteristic set of $\lambda_-$ except at the isolated zero $(0,0)$.

\begin{figure}
\includegraphics[width=0.45\textwidth]{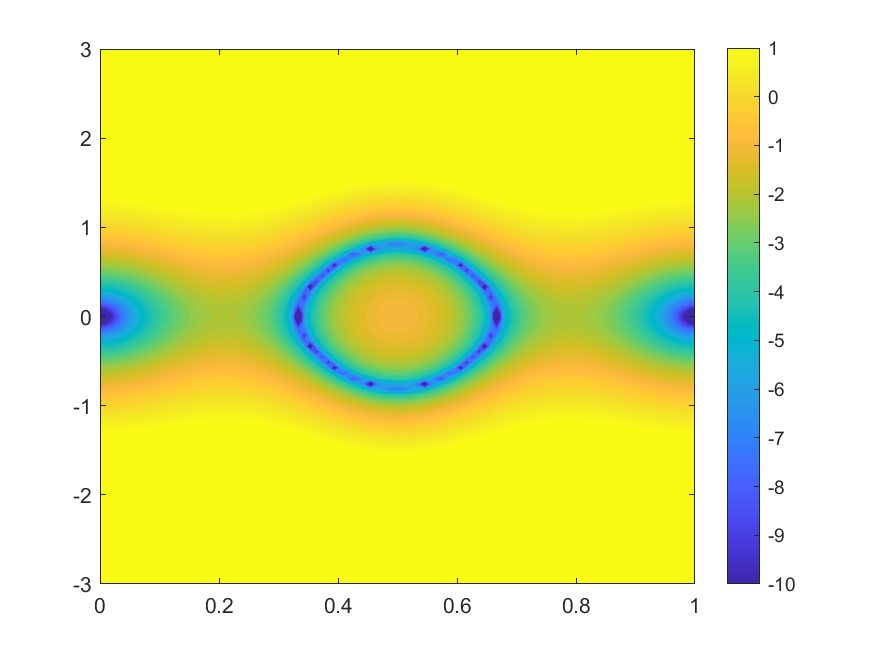}
\includegraphics[width=0.45\textwidth]{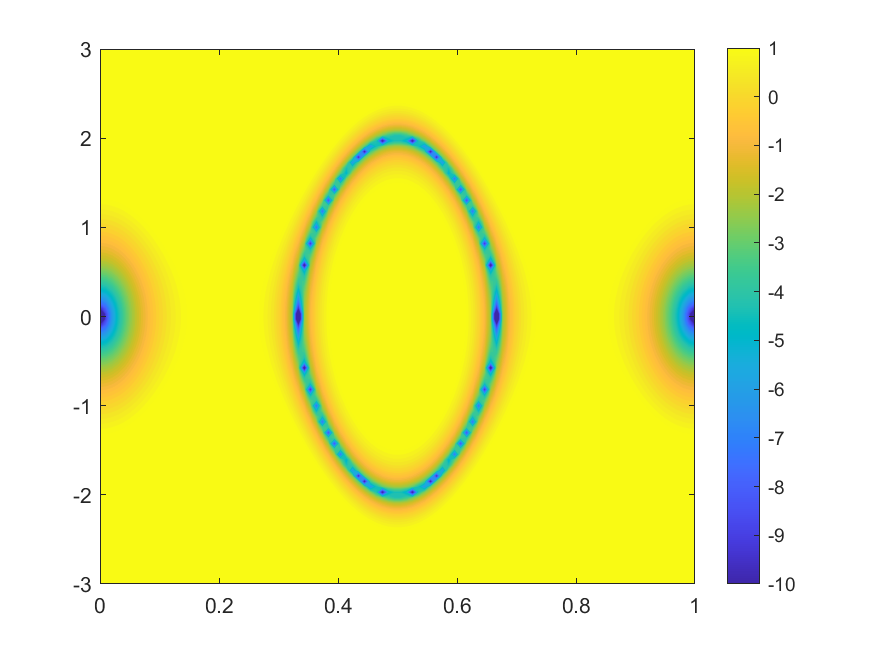}
\caption{Contour plot of the logarithm of the determinant of $\sigma_0(\mathscr L_\ch ^2)_{11}(x,\xi)$ over one period in the $x$ direction, showing the zero set consisting of the origin together with a closed curve in $T^*\mathbb T$. Here, $k_\perp=0$ and $w_1=2/5$ (left) and $w_1=1$ (right). \label{fig:detP}}
\end{figure}

Hence, for quasimomenta $k_\perp\ne0$ there are no degenerate wells, so we shall now restrict our attention to $k_\perp =0$ in which case there is a degenerate well at $(0,0)$:

\begin{prop}\label{prop:degwellH}
Let $\mathscr L_\ch ^2$ be as in Lemma \ref{lem:reduction1} with $k_\perp=0$. Then
$\sigma(\mathscr L_\ch ^2)_{11}(x,\xi)=\sum_{j=0}^2 h^jP_j(x,\xi)$ with $P_j\in S^{2-j}(T^*\mathbb T)$.
Moreover, $P_0$ has a degenerate well at $(0,0)$ with
$$
P_0(x,\xi)=(\xi^2+12\pi^2w_1^2x^2)\id_2+\mathcal O(\lvert (x,\xi)\rvert^3)
$$
and $P_1(0,0)=-2\pi\sqrt 3 w_1\diag(1,-1)$.
\end{prop}

\begin{proof}
That the symbol of $(\mathscr L_\ch ^2)_{11}$ has the stated expansion follows from Lemma \ref{lem:reduction1}. It is also clear that at $(0,0)$ the subprincipal symbol has the stated diagonal from, so we only need to show that the principal symbol $P_{11}$ has a degenerate well at $(0,0)$, where $P_{11}$ is as in Lemma \ref{lem:reduction1}. It is straightforward to check that $P_{11}(x,\xi)$ is positive semi-definite for all $(x,\xi)\in T^*\mathbb T$, and by Taylor's formula we have, with $f(x)=w_1(1-\cos(2\pi x))$ and $g(x)=w_1\sqrt 3\sin(2\pi x)$, that
\begin{equation*}
\begin{split}
f(x)^2 + g(x)^2  &= \sum_{j \ge 1} \alpha_{j} x^{2j} \text{ with }\alpha_1 = (2\pi \sqrt 3w_1)^2,  \\
-2f(x)g(x)  &= \sum_{j \ge 1} \beta_{j} x^{2j+1}\text{ with }\beta_1 = -8 \sqrt{3}\pi^3w_1^2
\end{split}
\end{equation*}
and $2i\xi f(x)=\sum_{j\ge1}\frac{2iw_1 (2\pi)^{2j} (-1)^{j+1}}{(2j)!}\xi x^{2j}$. This gives the result.
\end{proof}

We also record that the chiral Hamiltonian satisfies \eqref{eq:lowerboundS2}, where we use the observation in Remark 2 on page \pageref{rmk2} to momentarily view $(\mathscr L_\ch^2)_{11}$ as an operator on $\RR$. 

\begin{lemm}\label{lem:loweroperatorbound}
If $k_\perp=0$ then
$$
((\mathscr L_\ch^2)_{11}u,u)\ge 
 ((g^2-2fg\sigma_1)u,u)-Ch(u,u)
$$
where $f(x)=w_1(1-\cos(2\pi x))$ and $g(x)=w_1\sqrt 3\sin(2\pi x)$, and 
$g^2\operatorname{id}_{\CC^{2\times 2}}-2fg\sigma_1=12w_1^2\pi^2x^2\operatorname{id}_{\CC^{2\times 2}}+\mathcal O(x^3)$.
\end{lemm}

\begin{proof}
Let $u=(u_1,u_2)^t\in L^2(\mathbb R)$. By \eqref{eq:reduction} we have (for $k_\perp=0$) that $((\mathscr L_\ch^2)_{11}u,u)=(D_{\operatorname{c}}D_{\operatorname{c}}^*u,u)$ where
\begin{align*}
D_{\operatorname{c}}D_{\operatorname{c}}^* &=\begin{pmatrix} hD+k_x +ig&if \\ -if&hD+k_x-ig \end{pmatrix}\begin{pmatrix} hD+k_x -ig&if \\ -if&hD+k_x+ig \end{pmatrix}
\\&=((hD+k_x)^2 +f^2+g^2)\id_{\CC^{2\times 2}}-2fg\sigma_1+Q^wi\sigma_2-[hD,ig]\sigma_3
\end{align*}
with $Q^w$ given by \eqref{eq:Q}.
Here, $[hD,ig]=hg'=2h\pi\sqrt 3\cos(2\pi x)$ so  $(-[hD,ig]\sigma_3u,u)\ge -Ch(u,u)$. 
Noting that
\begin{align*}
(Q^wi\sigma_2u,u)&=(((hD+k_x)if+if(hD+k_x))i\sigma_2u,u)\\&=-2\Re(f(hD+k_x)\sigma_2u,u)\ge-((hD+k_x)^2 u+f^2u,u)
\end{align*}
we conclude that
$$
(D_{\operatorname{c}}D_{\operatorname{c}}^*u,u)\ge ((g^2-2f g\sigma_1)u,u)-Ch(u,u).
$$
We have $fg=\mathcal O(x^3)$ and $g^2=12w_1^2\pi^2x^2+\mathcal O(x^4)$ which gives the result.
\end{proof}

We can now show that Corollary \ref{cor:periodicquasimodesnormalform} implies existence of quasimodes for the chiral low-energy model.

\begin{proof}[Proof of Theorem \ref{cor:periodicquasimodes}]
By Proposition \ref{prop:degwellH} and Lemma \ref{lem:loweroperatorbound} we find that $\sigma(\mathscr L_\ch^2)$ satisfies the assumptions of Proposition \ref{prop:normalform} with $\xi_0=0$ and $\mu_j=(-1)^j\omega$, $j=1,2$, and $\omega=2\pi \sqrt 3w_1$. We then apply Corollary \ref{cor:periodicquasimodesnormalform} to $(\mathscr L_\ch^2)_{11}$ and obtain approximate eigenvalues $\lambda^j(n) $ and quasimodes $u^j(n)$ as stated, such that $((\mathscr L_\ch^2)_{11}-\lambda^j(n))u^j(n)=\mathcal{O}(h^{\ell+\frac32})$ in $L^2(\mathbb T)$ for any $\ell\in\NN_0$.

Fix $j\in\{1,2\}$ and $n\in\NN_0$ and omit them from the notation. In the notation of Lemma \ref{lem:reduction0} with $\psi=(\psi_1,\psi_2)^t=(u,\pm\lambda^{-\frac12}D_\ch^*u)^t$ we then have $D_\ch D_\ch^*=(\mathscr L_\ch^2)_{11}$ and the correspondence
\begin{equation*}
 \mathscr L_\ch \psi=\begin{pmatrix} 0& D_{\operatorname{c}} \\D_{\operatorname{c}}^* & 0 \end{pmatrix}\begin{pmatrix} u \\ \pm\lambda^{-1/2} D_{\operatorname{c}}^* u \end{pmatrix}=\begin{pmatrix} \pm \lambda^{-1/2} (\mathscr L_\ch^2)_{11}u\\\pm\sqrt \lambda \psi_2\end{pmatrix}=\pm\sqrt\lambda \psi+\mathcal{O}_{L^2(\mathbb T)}(h^{\ell+1}),
\end{equation*}
where the last identity follows from $(\mathscr L_\ch^2)_{11}u=\lambda u+\mathcal{O}_{L^2(\mathbb T)}(h^{\ell+\frac32})$ and $\lambda^{-1/2}=\mathcal{O}(h^{-1/2})$. Since $\lVert \psi\rVert_{L^2(\mathbb T;\CC^4)}\ge \lVert \psi_1\rVert_{L^2(\mathbb T;\CC^2)}= 1+\mathcal{O}(h^{1/2})$ we still have $\mathscr L_\ch \psi=\pm\sqrt\lambda\psi+\mathcal{O}(h^{\ell+1})$ after normalizing and renaming $\psi/\lVert \psi\rVert_{L^2(\mathbb T)}$ to $\psi$. Also, since $\WF_h(u)=\{(0,0)\}$ we have $\WF_h(\psi)=\{(0,0)\}$ by the microlocal property of pseudodifferential operators. By Lemma \ref{lem:reduction0},  $H_\ch^w(k_x)$ is unitarily equivalent to $\mathscr L_\ch$, which gives approximate eigenvalues and quasimodes of $H_\ch^w(k_x)$ as in the statement of the theorem.

To prove the last part, let $H_\ch^w(k_x)$ be given by \eqref{eq:semiclassicalH} with $k_\perp=0$. Since $k_x\in[0,2\pi h)$ we may write $k_x=h\xi_0$ with $\xi_0\in[0,2\pi)$.
If we make the symplectic change of variables $(x,\zeta)=(x,\xi+k_x)=(x,\xi+h\xi_0)$ then the proof of Proposition \ref{prop:normalform} shows that
$$
H_\ch^w(k_x)=e^{-ix\xi_0}q^w(x,hD)e^{ix\xi_0},
$$
where
$$
q(x,\zeta)=\sigma(H_\ch^w(k_x))(x,\xi)=\sigma(H_\ch^w(k_x=0))(x,\zeta).
$$
This shows that the approximate eigenvalues of $H_\ch^w(k_x)$ obtained in the first part of the proof are independent of $k_x$. Hence, the last statement of the theorem follows from \eqref{eq:BFunion}. Note also that multiplying by $e^{ix\xi_0}$ does not affect the wavefront set of the associated quasimodes, since $\WF_h(e^{ix\xi_0})=\RR^n\times\{0\}$, see \cite[Section 8.4]{zworski}. 
\end{proof}

\subsection{Wells for the chiral Harper model}\label{ss:Harper}

We now establish the existence of degenerate wells for the operator $H_{\operatorname{\Psi DO}}=a^w(x,hD)$ in \eqref{eq:HarperPDO}. To increase similarity with the presentation in the previous subsection we introduce the auxiliary operator $b^w(x,hD)$, where
\begin{equation}
\label{eq:b}
b(x,\xi)=2\mathbf t(k_{\perp}) \cos(2\pi\xi)+\mathbf t_0+V_w(x).
\end{equation}
Observe that $a(x,\xi)=b(\xi,-x)$. As operators on $\RR$ (cf.~Remark 2 on page \pageref{rmk2}) we then have $a^w(x,hD)=\mathcal F_h^{-1} b^w(x,hD)\mathcal F_h$ where $\mathcal F_h$ is the semiclassical Fourier transform, see \cite[Theorem 4.9]{zworski}. We shall then use the following result, analogous to Lemma \ref{lem:reduction0} and with identical proof. 

\begin{lemm}\label{lem:reduction0Harper}
In the chiral limit, $b^w(x,hD)$ is unitarily equivalent to a Hamiltonian on off-diagonal block form, $$\mathscr H_\ch =\begin{pmatrix} 0 & D_\ch  \\  D_\ch ^* & 0 \end{pmatrix},$$ where the symbol of $D_\ch $ is given by
$$
 D_{\operatorname{c}}(x,\xi) = \frac12\begin{pmatrix} i & 1\\ -i &  1 \end{pmatrix}\begin{pmatrix} 2\cos(2\pi \xi) e^{2\pi ik_\perp }+1 &  w_1 U^+(x) \\  w_1 U^-(x) &2\cos(2\pi \xi) e^{2\pi ik_\perp }+1 \end{pmatrix}\begin{pmatrix} -i & i\\ 1 &  1 \end{pmatrix}.
$$
Then $\mathscr H_\ch ^2=\diag(D_\ch D_\ch ^*,D_\ch ^*D_\ch )$ and if $D_\ch D_\ch ^*u = \lambda u$ then $\mathscr H_\ch  v=\pm\sqrt\lambda v$ for $v=(v_1,v_2)^t$ with $v_1 =u$ and $v_2 = \pm \lambda^{-1/2} D_{\operatorname{c}}^* u$.
\end{lemm}

Let $\mathscr H_{\operatorname{c}}$ be as in Lemma \ref{lem:reduction0Harper}.
Using the lemma it is easy to see that 
\begin{align*}
 \operatorname{det}( D_{\operatorname{c}})(x,\xi) &=  (2\cos(2\pi \xi)e^{2\pi ik_\perp }+1)^2- w_1^2 U^+(x)U^-(x)\\
  & = (2\cos(2\pi \xi) \cos(2\pi k_\perp ) +1)^2  -4\cos^2(2\pi \xi) \sin^2(2\pi k_\perp ) -w_1^2 U^+(x)U^-(x)\\&\quad +4i\cos(2\pi \xi)\sin(2\pi k_\perp ).
\end{align*}
  Thus in case that $k_\perp  \notin\frac12\ZZ$, then for the imaginary part to vanish we require $\xi \in  \frac12\ZZ+\frac{1}4$. For then the real part to vanish as well, we require 
  $ 1 = w_1^2 U^+ U^-(x),$
  which, since the range of $U^+U^-$ is $[-2,4]$, admits a solution once $w_1 \ge 1/2.$ 
  
  Conversely, for $k_\perp \in\ZZ$ there is the special solution $x=0$, $\xi=\pm \frac{1}{3} + \mathbb Z$ which exists independent of $w_1$, together with the level set 
\begin{equation}\label{eq:levelsetHarper}
\{(x,\xi):(2\cos(2\pi \xi)  +1)^2=w_1^2 U^+(x)U^-(x)\}.
\end{equation}
For $k_\perp =\frac12+\ZZ$ there is the special solution $x=0$, $\xi=\pm \frac{1}{6} + \mathbb Z$ together with the level set 
\begin{equation}\label{eq:levelset2Harper}
\{(x,\xi):(2\cos(2\pi \xi)  -1)^2=w_1^2 U^+(x)U^-(x)\}.
\end{equation}
As we shall see, there are no degenerate wells unless $k_\perp \equiv 0$ or $k_\perp \equiv\frac12$ mod $\ZZ$ in which case there are degenerate wells precisely at these special solutions.
  
\begin{lemm}\label{lem:Harpersquareprincipalsymbol}
Let $\mathscr H_{\operatorname{c}}$ be as in Lemma \ref{lem:reduction0Harper}. Then $\mathscr H_\ch ^2=\sigma(\mathscr H_\ch ^2)^w(x,hD)$
where
$\sigma_0(\mathscr H_\ch ^2)=\diag( P_{11},P_{22})+2g\Im(\varUpsilon_{k_\perp })\diag(1,-1,1,-1)$
with
\begin{align*}
 P_{11}(x,\xi)&=  \begin{pmatrix} \lvert\varUpsilon_{k_\perp }\rvert^2 + f^2+g^2&  
2(i\Re\varUpsilon_{k_\perp } - g)f  \\ 2(-i\Re\varUpsilon_{k_\perp } - g)f  &  \lvert\varUpsilon_{k_\perp }\rvert^2 + f^2+g^2 \end{pmatrix},\quad P_{22}(x,\xi)=  \sigma_2P_{11}(x,\xi)\sigma_2,
\end{align*}
where $f(x)=w_1(1-\cos(2\pi x))$, $g(x)=w_1\sqrt 3\sin(2\pi x)$, and $\varUpsilon_{k_\perp }(\xi)=2\cos(2\pi\xi)e^{2\pi i k_\perp }+1$.
\end{lemm}

\begin{proof}
We have $\mathscr H_\ch ^2=\diag(D_\ch D_\ch ^*,D_\ch ^*D_\ch )$ where
\begin{align*}
D_\ch D_\ch ^*&=\begin{pmatrix}\varUpsilon^w_{k_\perp }\varUpsilon^w_{-k_\perp }+f^2+g^2& i(\varUpsilon_{k_\perp }^wf+f\varUpsilon_{-k_\perp }^w)-2fg\\
-i(\varUpsilon_{k_\perp }^wf+f\varUpsilon_{-k_\perp }^w)-2fg & \varUpsilon^w_{-k_\perp }\varUpsilon^w_{k_\perp }+f^2+g^2\end{pmatrix}+\operatorname{diag}(Q_{-k_\perp },-Q_{-k_\perp }),\\
D_\ch ^*D_\ch &=\begin{pmatrix}\varUpsilon^w_{-k_\perp }\varUpsilon^w_{k_\perp }+f^2+g^2& i(\varUpsilon_{-k_\perp }^wf+f\varUpsilon_{k_\perp }^w)+2fg\\
-i(\varUpsilon_{-k_\perp }^wf+f\varUpsilon_{k_\perp }^w)+2fg & \varUpsilon^w_{-k_\perp }\varUpsilon^w_{k_\perp }+f^2+g^2\end{pmatrix}+\operatorname{diag}(-Q_{k_\perp },Q_{k_\perp }),
\end{align*}
where  $Q_{k_\perp }=i(g\varUpsilon_{k_\perp }^w-\varUpsilon_{-k_\perp }^wg)$, and $f$, $g$ and $\varUpsilon_{k_\perp }$ are as in the statement. Since $\varUpsilon_{-k_\perp }(\xi)$ is the complex conjugate of $\varUpsilon_{k_\perp }(\xi)$, the Weyl calculus now gives that $\sigma_0(D_\ch D_\ch ^*)=P_{11}+2g\Im(\varUpsilon_{k_\perp })\diag(1,-1)$. Similarly, $\sigma_0(D_\ch ^*D_\ch )=P_{22}+2g\Im(\varUpsilon_{k_\perp })\diag(1,-1)$.
\end{proof}

By Lemma \ref{lem:Harpersquareprincipalsymbol} we have
$$
\det(\sigma_0(\mathscr H_\ch ^2)_{jj}-\lambda)=(\lambda- \lvert\varUpsilon_{k_\perp }\rvert^2 + f^2+g^2)^2-4g^2(\Im(\varUpsilon_{y_k}))^2-4f^2(g^2+(\Re(\varUpsilon_{k_\perp }))^2)
$$
so the eigenvalues of $\sigma_0(\mathscr H_\ch ^2)_{jj}$ for $j=1,2$ are given by
$$
\lambda_\pm(x,\xi)=\lvert\varUpsilon_{k_\perp }\rvert^2 + f^2+g^2\pm 2\sqrt{f^2g^2+f^2(\Re(\varUpsilon_{k_\perp }))^2+g^2(\Im(\varUpsilon_{y_k}))^2}.
$$
If $k_\perp \notin\frac12\ZZ$ then 
$$
\varUpsilon_{k_\perp }(\xi) =2\cos(2\pi\xi)\cos(2\pi k_\perp )+1+2i\cos(2\pi\xi)\sin(2\pi k_\perp )
$$
has vanishing imaginary part only when $\xi=\frac14+\frac12\ZZ$ in which case $\varUpsilon_{k_\perp }(\xi)\equiv1$ for all $\xi$. Thus, $\lambda_+(x,\xi)$ never vanishes when $k_\perp \notin\frac12\ZZ$.
If $k_\perp \in\ZZ$ then $\lambda_+(x,\xi)=0$ precisely when $x=0$ and $\xi=\pm\frac13$ mod $\ZZ$. By the analysis preceding the lemma, the characteristic set of $\lambda_-$ is the level set \eqref{eq:levelsetHarper}.
Since $w_1^2U^+U^-=f^2-g^2$ we find in view of \eqref{eq:isolatedzero} that this level set has several connected components in $\RR^2/\ZZ^2$: the points $(0,\pm\frac13)$ and one or two closed curves depending on $w_1$, see Figure \ref{fig:zerosetHarper}.

\begin{figure}
\begin{centering}
\includegraphics[width=0.3\textwidth]{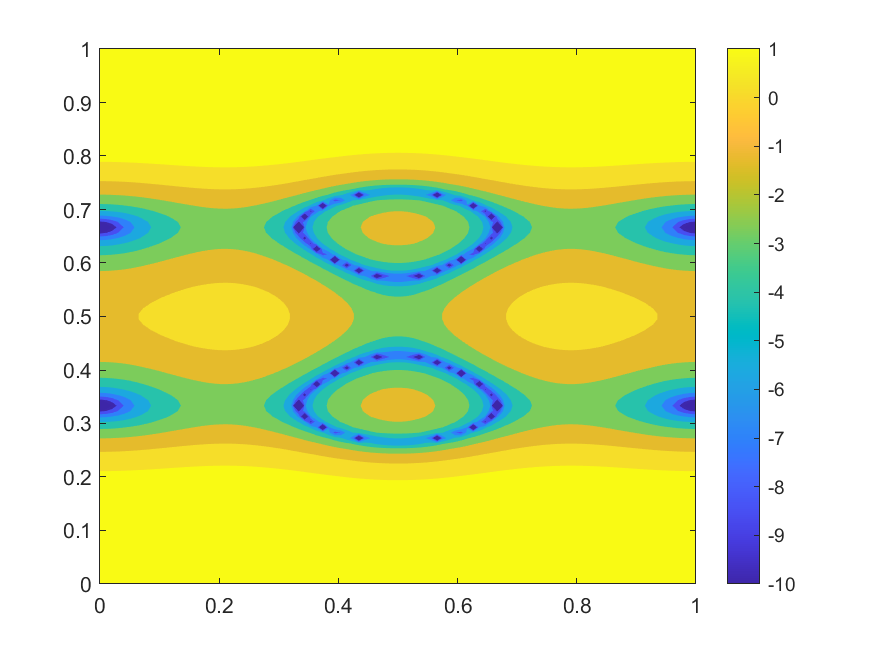}
\includegraphics[width=0.3\textwidth]{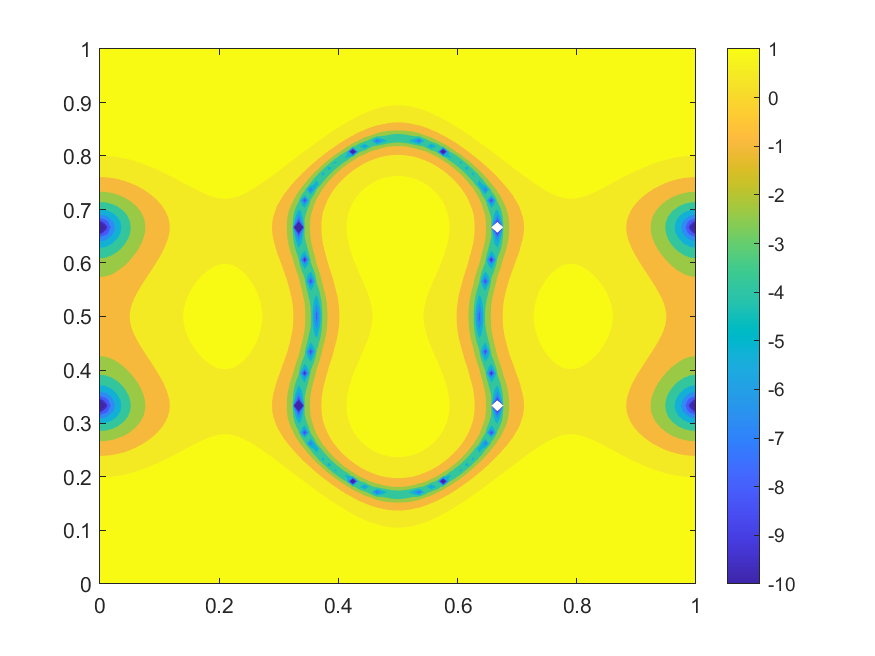}
\includegraphics[width=0.3\textwidth]{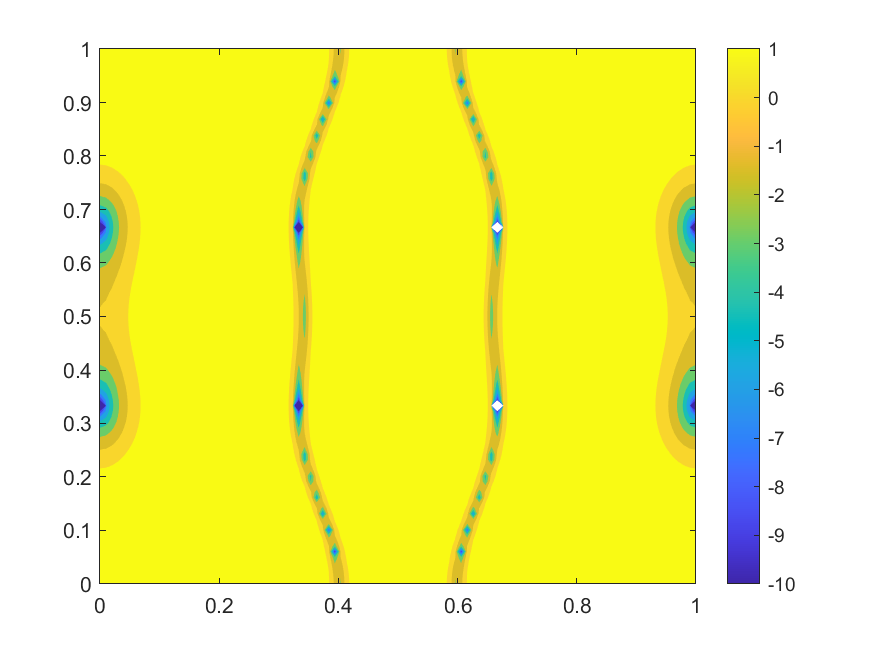}\\
\includegraphics[width=0.3\textwidth]{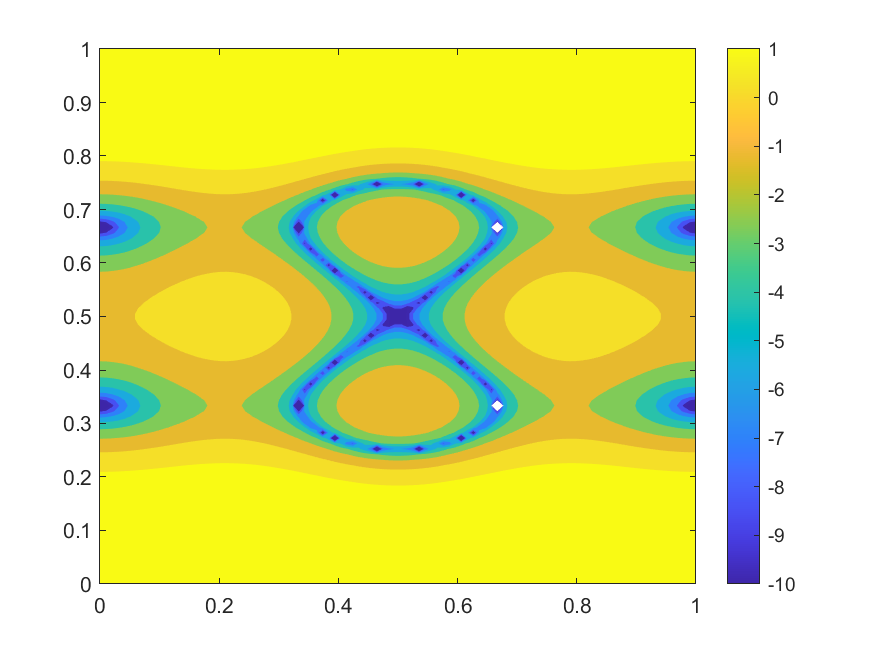}
\includegraphics[width=0.3\textwidth]{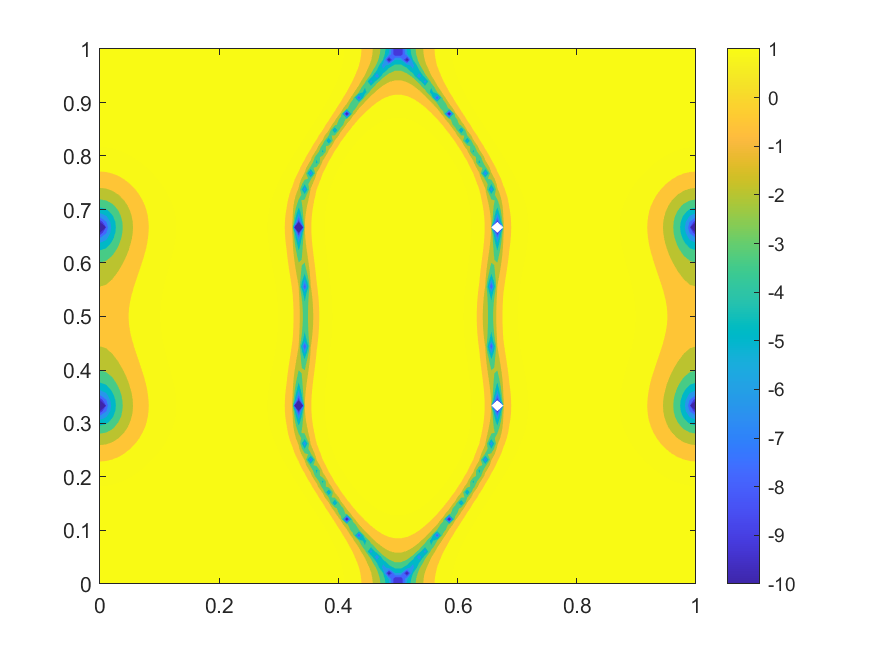}
\end{centering}
\caption{Contour plot of the logarithm of the determinant of $\sigma_0(\mathscr H_\ch ^2)_{11}(x,\xi)$ over one period in the $x$ and $\xi$ directions, showing the zero set consisting of the points $(0,\pm\frac13)$ and one or two closed curves in $T^*\mathbb T$ mod $\ZZ^2$. Here, $k_\perp=0$ and the top panels show the zero set for $w_1=2/5$ (left), $w_1=1$ (middle) and $w_1=2$ (right), while the bottom panels show the special values $w_1=1/2$ (left) and $w_1=3/2$ (right) where the number of closed curves in the zero set switches between one and two.\label{fig:zerosetHarper}}
\end{figure}

If $k_\perp =\frac12+\ZZ$ then $\lambda_+(x,\xi)=0$ precisely when $x=0$ and $\xi=\pm\frac16$ mod $\ZZ$. In this case the characteristic set of $\lambda_-$ is the level set \eqref{eq:levelset2Harper} which again has three connected components in $\RR^2/\ZZ^2$. In particular, when $k_\perp \in\frac12\ZZ$ the eigenvalues are distinct near all points in the characteristic set of $\lambda_-$ except at a set of discrete points.

We now restrict to the case $k_\perp \equiv 0$ mod $\frac12\ZZ$. Since we then have $\varUpsilon_{k_\perp }=\varUpsilon_{-k_\perp }$ the result of Lemma \ref{lem:Harpersquareprincipalsymbol} takes on a simpler form. 
We first compute the full symbol of $\mathscr H_\ch ^2$ where we include a restatement of Lemma \ref{lem:Harpersquareprincipalsymbol} for convenience.

\begin{lemm}\label{lem:reduction3}
Let $\mathscr H_{\operatorname{c}}$ be as in Lemma \ref{lem:reduction0Harper} and $k_\perp =0$ or $k_\perp=\frac12$.  Then $\mathscr H_\ch ^2=\sigma(\mathscr H_\ch ^2)^w(x,hD)$
where
$\sigma_0(\mathscr H_\ch ^2)=\diag( P_{11},P_{22})$
with
\begin{align*}
 P_{11}(x,\xi)&=  \begin{pmatrix} \varUpsilon^2 + f^2+g^2&  
2(i\varUpsilon - g)f  \\ 2(-i\varUpsilon - g)f  &  \varUpsilon^2 + f^2+g^2 \end{pmatrix},\quad P_{22}(x,\xi)=  \begin{pmatrix} \varUpsilon^2 + f^2+g^2&  
2(i\varUpsilon + g)f  \\ 2(-i\varUpsilon + g)f  &  \varUpsilon^2 + f^2+g^2 \end{pmatrix}
\end{align*}
where $f(x)=w_1(1-\cos(2\pi x))$, $g(x)=w_1\sqrt 3\sin(2\pi x)$, and $\varUpsilon(\xi)=2\cos(2\pi\xi)(-1)^{2 k_\perp}+1$, 
and with lower order terms 
\[ \sigma(\mathscr H_\ch ^2)-\sigma_0(\mathscr H_\ch ^2)= \sigma([ig,\varUpsilon^w(hD)])\operatorname{diag}(1,-1,-1,1)-a\diag(\sigma_2,\sigma_2)\]
where $a(x,\xi)=\sigma(\varUpsilon^wf+f\varUpsilon^w)(x,\xi)-2\varUpsilon(\xi)f(x)$.
\end{lemm}

\begin{proof}
The result follows by inspecting the proof of Lemma \ref{lem:Harpersquareprincipalsymbol} and using $\varUpsilon_{k_\perp }=\varUpsilon_{-k_\perp }$ together with properties of the Weyl calculus. 
\end{proof}

Consider $k_\perp\in\ZZ$ and perform a symplectic change of variables $\xi= \zeta\pm 1/3$. Then 
$$
\varUpsilon(\xi)
=2\cos(2\pi \zeta\pm 2\pi/3)+1=1-\cos(2\pi \zeta)\mp\sqrt 3\sin(2\pi\zeta)=U^\mp(\zeta)/w_1,
$$
so $(\varUpsilon(\xi))^2=(\varUpsilon(\zeta\pm 1/3))^2=12\pi^2\zeta^2+\mathcal{O}(\zeta^3)$. When $k_\perp=\frac12+\ZZ$ the symplectic change of variables $\xi= \zeta\pm 1/6$ gives 
$$
\varUpsilon(\xi)
=-2\cos(2\pi \zeta\pm \pi/3)+1=1-\cos(2\pi \zeta)\pm\sqrt 3\sin(2\pi\zeta)=U^\pm(\zeta)/w_1,
$$
so again $(\varUpsilon(\xi))^2=12\pi^2\zeta^2+\mathcal{O}(\zeta^3)$. 
The operator $\mathscr H_\ch^2$ in Lemma \ref{lem:reduction3} therefore has wells at $(0,\pm\frac13)$ mod $\ZZ^2$, or at  $(0,\pm\frac16)$ mod $\ZZ^2$, depending on if $k_\perp=0$ or if $k_\perp=\frac12$ mod $\ZZ$, and these wells are degenerate since the eigenvalues of $\sigma_0(\mathscr H_\ch^w)_{jj}$ coalesce there. 
Before turning to quasimodes concentrated near each corresponding well we first compute the lower order terms of the symbol.

\begin{lemm}\label{lem:lowerorder}
Assume that $k_\perp=0$ or $k_\perp=\frac12$. Let $g(x)=w_1\sqrt 3\sin(2\pi x)$, $\varUpsilon(\xi)=2\cos(2\pi\xi)(-1)^{2k_\perp}+1$ and $a(x,\xi)=\sigma(\varUpsilon^wf+f\varUpsilon^w)(x,\xi)-2\varUpsilon(\xi)f(x)$. Then for $N=0,1,\ldots,$ we have
$$
\sigma([ig,\varUpsilon^w(hD)])(x,\xi)=(-1)^{2k_\perp}4\sqrt 3w_1\cos(2\pi x)\sin(2\pi \xi)\sum_{n=0}^N\frac{h^{2n+1}}{(2n+1)!}(2\pi^2)^{2n+1}(-1)^{n}
$$
modulo an error in $\mathcal{O}_{S(1)}(h^{2N+3})$ as $h\to0$, and
$$
a(x,\xi)=-(-1)^{2k_\perp}4w_1\cos(2\pi x)\cos(2\pi\xi)\sum_{n=1}^N\frac{h^{2n}}{(2n)!}(2\pi^2)^{2n}(-1)^n+\mathcal{O}_{S(1)}(h^{2N+2})
$$
as $h\to0$.
\end{lemm}

\begin{proof}
We have $[ig,\varUpsilon^w]=i(g\#\varUpsilon-\varUpsilon\#g)^w$ where
$$
g\#\varUpsilon(x,\xi)=\sum_{k=0}^K\frac{i^kh^k2^{-k}}{k!}(D_\xi D_y-D_xD_\eta)^k(g(x)\varUpsilon(\eta))\bigg\rvert_{\eta=\xi}+\mathcal{O}_{S(1)}(h^{K+1})
$$
for $K=0,1,\ldots,$ see \cite[Theorem 4.12]{zworski} and \cite[Theorem 4.18]{zworski}. Thus, in the difference $g\#\varUpsilon-\varUpsilon\#g$ all the terms for even $k$ cancel, which gives
$$
i(g\#\varUpsilon-\varUpsilon\#g)=-2i\sum_{n=0}^N\frac{h^{2n+1}2^{-(2n+1)}}{(2n+1)!}\partial_x^{2n+1}g(x)D_\xi^{2n+1}\varUpsilon(\xi)+\mathcal{O}_{S(1)}(h^{2N+3}).
$$
Computing the derivatives shows that the symbol of $[ig,\varUpsilon^w]$ has the stated form. 

Since $a=\varUpsilon\# f+f\#\varUpsilon-2\varUpsilon f$, similar computations also show that 
\begin{align*}
a(x,\xi)=2\sum_{n=1}^N\frac{h^{2n}2^{-2n}}{(2n)!}\partial_x^{2n}f(x)D_\xi^{2n}\varUpsilon(\xi)+\mathcal{O}_{S(1)}(h^{2N+2})
\end{align*}
so after computing the derivatives we obtain the result.
\end{proof}

\begin{prop}\label{prop:degwellHarper}
Let $\mathscr H_\ch ^2=\mathscr H_\ch ^2(x,hD)$ be as in Lemma \ref{lem:reduction3} with $k_\perp=0$ or $k_\perp=\frac12$.  Then
$\sigma(\mathscr H_\ch ^2)_{11}(x,\xi)\sim\sum_{j\ge0}h^jP_j(x,\xi)$ with $P_j\in S(1)$. Moreover, for each $\xi_0=\pm \frac13(\frac12)^{2k_\perp} \!\!\! \mod 1$, $P_0$ has a degenerate well at $(0,\xi_0)$ with
$$
P_0(x,\xi)=12\pi^2((\xi-\xi_0)^2+w_1^2x^2)\id_2+\mathcal O(\lvert (x,\xi)-(0,\xi_0)\rvert^3)
$$
and $P_1(0,\xi_0)=12\pi^2w_1c(\xi_0)\diag(1,-1)$, where $c(\pm \frac13(\frac12)^{2k_\perp})=\pm(-1)^{2k_\perp}$.
\end{prop}

\begin{proof}
The asymptotic expansion of $\sigma(\mathscr H_\ch^2)$ follows from Lemmas \ref{lem:reduction3} and \ref{lem:lowerorder}, and by the discussion preceding Lemma \ref{lem:lowerorder} we see that $P_0$ has the stated form. Next, with $\xi_0=\pm \frac13(\frac12)^{2k_\perp}$ we note that 
$$
\sin(2\pi(\zeta+\xi_0))=\pm\frac{\sqrt 3}2\cos(2\pi\zeta)-(-1)^{2k_\perp}\frac12\sin(2\pi\zeta)
$$
so with $c(\pm \frac13(\frac12)^{2k_\perp})=\pm(-1)^{2k_\perp}$ it follows from Lemma \ref{lem:lowerorder} that $a=\sigma(\varUpsilon^wf+f\varUpsilon^w)-2\varUpsilon f=\mathcal{O}_{S(1)}(h^{2})$ and
$$
\sigma([ig,\varUpsilon^w(hD)])(0,\xi_0)=c(\xi_0)12\pi^2w_1 h+\mathcal{O}(h^{3})
$$
which in view of Lemma \ref{lem:reduction3} shows that $P_1$ also has the stated form.
\end{proof}

We can now prove existence of quasimodes for the chiral Harper model.

\begin{proof}[Proof of Theorem \ref{cor:periodicquasimodes2}]
For each $\xi_0=\pm \frac13(\frac12)^{2k_\perp}+n_0$, $n_0\in\ZZ$, we find by Proposition \ref{prop:degwellHarper} that $(12\pi^2)^{-1}\sigma(\mathscr H_\ch^2)(x,\xi)$ satisfies the assumptions of Proposition \ref{prop:normalform} with $\mu_j=\pm(-1)^{2k_\perp+j-1}\omega$, $j=1,2$, and $\omega=w_1$. Let $Uv(x)=h^{-\frac14}v(h^{-\frac12}x)$ so that \eqref{eq:normalformpullbackintro} reads $P^w(x,hD)=U\mathcal TU^*$ with $U$ unitary. We then apply Theorem \ref{thm:expansionsintro} to $P^w(x,hD)=(\mathscr H_\ch^2)_{11}$ and obtain approximate eigenvalues $\lambda^{(j)}(n)$ (independent of the integer $n_0$ in the definition of $\xi_0$) such that for any $\ell\in\NN_0$, $\lambda^{(j)}(n) = h \sum_{i = 0}^{2\ell} h^{i/2} \lambda_{i}^{(j)}(n)+\mathcal O(h^{\ell+3/2})$ with 
$$
h^{(j)}_0(n)=((2n+1)\pm(-1)^{2k_\perp+j-1})w_1,
$$
together with quasimodes $v^{(j)}(n,y)$ such that, with $u^{(j)}(n,x)=Uv^{(j)}(n,x)$, we have 
$$
((12\pi^2)^{-1}\mathscr H_\ch^2-\lambda^j(n))u^{(j)}(n)=\mathcal{O}_{\mathscr S}(h^{\ell+\frac32}),
$$
where $\lVert u^{(j)}(n)\rVert_{L^2(\mathbb R)}=1+\mathcal{O}(h^\frac12)$ and $\WF_h( u^{(j)}(n))=\{(0,\xi_0)\}$, $n\in\NN_0$. Multiplying $\lambda^{(j)}(n)$ by $12\pi^2$ and renaming the $\lambda_i$ we get approximate eigenvalues of $(\mathscr H_\ch^2)_{11}$ of the stated form. By Lemma \ref{lem:reduction0Harper}, $\mathscr H_\ch$ is unitarily equivalent to $\mathcal F_h a^w(x,hD)\mathcal F_h^{-1}$ with $a$ given by \eqref{eq:a}, so by repeating the second part of the proof of Theorem \ref{cor:periodicquasimodes} we find that $a^w(x,hD)\mathcal F_h^{-1}w_\pm^{(j)}(n)=\pm\sqrt{\lambda^{(j)}(n)}\mathcal F_h^{-1}w_\pm^{(j)}(n)+\mathcal O_{\mathscr S}(h^{\ell+1})$ for some $w_\pm^{(j)}(n)\in\mathscr S(\RR;\CC^4)$ with $\lVert w_\pm^{(j)}(n)\rVert_{L^2(\mathbb R)}=1+\mathcal{O}(h^\frac12)$ and $\WF_h(w_\pm^{(j)}(n))=\{(0,\xi_0)\}$, $n\in\NN_0$. Clearly, $\lVert \mathcal F_h^{-1}w_\pm^{(j)}(n)\rVert_{L^2(\mathbb R)}=1+\mathcal{O}(h^\frac12)$ and $\WF_h(\mathcal F_h^{-1}w_\pm^{(j)}(n))=\{(\xi_0,0)\}$, so to obtain the stated quasimodes $\psi_\pm^{(j)}(n)\in L^2(\TT;\CC^4)$ for $H_{\operatorname{\Psi DO}}(0,w_1)=a^w(x,hD):L^2(\TT)\to L^2(\TT)$ we now simply repeat the arguments used to prove Corollary \ref{cor:periodicquasimodesnormalform}. We omit the details. 
\end{proof}

\begin{appendix}

\section{Auxiliary results}
\label{sec:aux_results}

\label{sec:auxLemma}

Here we provide some results used in the main text, starting with a Rayleigh-Ritz principle stated for the massive (non-semiclassical) Weyl operator $T_\mathrm{mass}^w(y,D)$.

\begin{lemm}[Rayleigh-Ritz principle]\label{thm:rayleighritz}
Let $T^w(y,D;h)$ be a self-adjoint operator semi-bounded from below, $T^w\ge -Ch$, and let $G(y,\eta;h)\in C_0^\infty(T^*\RR)$ with $0\le G\le1$. Set 
$$
T_\mathrm{mass}^w=T^w+(1-G^w)\id_{\CC^{2\times2}}
$$
and assume that there exists a set $\{\psi_n\}_{n\in\NN}$ of functions satisfying \eqref{eq:bsimon1} and \eqref{eq:bsimon2}. Then, as a densely defined operator on $L^2(\RR)$, $T_\mathrm{mass}^w$ has at least $n$ eigenvalues $\lambda_1(h)\le\ldots\le \lambda_n(h)$ counting multiplicity and $\varlimsup_{h\to0^+} \lambda_n(h)/h\le e_n$, with $e_n$ being the number appearing in \eqref{eq:bsimon2}.
\end{lemm}

\begin{proof}
Let
$$
\mu_n(h)=\sup_{\zeta_1,\ldots,\zeta_{n-1}} Q(\zeta_1,\ldots,\zeta_{n-1};h),
$$
where 
$$
Q(\zeta_1,\ldots,\zeta_{n-1};h)=\inf\{(T_\mathrm{mass}^w\psi,\psi):\psi\in D(T_\mathrm{mass}^w),\ \lVert\psi\rVert=1,\ \psi\in[\zeta_1,\ldots,\zeta_{n-1}]^\perp\}.
$$

{\it Step 1:} Since $T_\mathrm{mass}^w$ is self-adjoint and semi-bounded from below, the min-max principle (see \cite[Theorem XIII.1]{MR0493421}) implies that $T_\mathrm{mass}^w$ has at least $n$ eigenvalues $\lambda_1(h),\ldots,\lambda_n(h)$ counting multiplicity and either $\mu_n(h)=\lambda_n(h)$ or $\mu_n(h)=\inf\sigma_{\operatorname{ess}}(T_\mathrm{mass}^w)$.

{\it Step 2:} 
Let $\varepsilon>0$ be fixed but arbitrary, and for each $h$, choose $\zeta_1^h,\ldots,\zeta_{n-1}^h$ so that
\begin{equation}\label{eq:Qplusve}
\mu_n(h)\le Q(\zeta_1^h,\ldots,\zeta_{n-1}^h;h)+\epsilon.
\end{equation}
From \eqref{eq:bsimon1} it follows that for $h$ small, $\psi_1,\ldots,\psi_n$ span an $n$-dimensional space, so for sufficiently small $h$ we can find a linear combination $\psi$ of $\psi_1,\ldots,\psi_n$ such that $\psi\in[\zeta_1^h,\ldots,\zeta_{n-1}^h]^\perp$. By \eqref{eq:bsimon2} we then get
$$
Q(\zeta_1^h,\ldots,\zeta_{n-1}^h;h)\le he_n+\mathcal{O}(h^{6/5}).
$$ 
From \eqref{eq:Qplusve} and the fact that $\epsilon$ was arbitrary we find that
\begin{equation}\label{eq:usingQplusve}
\mu_n(h)\le  he_n+\mathcal{O}(h^{6/5}). 
\end{equation}

{\it Step 3:} Fix another $0<\epsilon<1$ and let $\Omega_\epsilon\subset\CC$ be an $\epsilon$-neighborhood of the negative half-line $\RR_-=\{x\in\RR:x<0\}$. Let $z\in\Omega_\epsilon$. Since $T^w\ge -Ch$ it follows that $T^w+1-z$ is invertible for $z\in\Omega_\epsilon$ and $h$ small, and by the pseudodifferential calculus $G^w$ is compact, see \cite[Theorem 4.28]{zworski}. We then have 
$$
T_\mathrm{mass}^w-z =(T^w+1-z)(1-K(z)),
$$
where $K(z)= (T^w+1-z)^{-1}G^w$ is compact. Also, since $0\le G\le1$ we have $\lVert K(z_0)\rVert\le \lVert (T^w+1-z_0)^{-1}\rVert<1$ for $z_0\in\Omega_\epsilon$ with $\Re z_0\ll -1$, so $1-K(z_0)$ is invertible. By analytic Fredholm theory it then follows that the resolvent $(T_\mathrm{mass}^w-z)^{-1}$ is meromorphic in $\Omega_\epsilon$ (see e.g., \cite[Theorem D.4]{zworski})
which implies that $T_\mathrm{mass}^w-z$ is Fredholm for $z\in\Omega_\epsilon$, see e.g., \cite[Theorem 9.6]{taylor1966theorems}. By definition we then have $\sigma_{\operatorname{ess}}(T_\mathrm{mass}^w)\bigcap\Omega_\epsilon=\emptyset$.

{\it Step 4:} By step 3 we have $\inf\sigma_{\operatorname{ess}}(T_\mathrm{mass}^w)\ge \varepsilon>0$ so  \eqref{eq:usingQplusve} implies that
$$
\mu_n(h)\le he_n+\mathcal{O}(h^{6/5})<\varepsilon\le \sigma_{\operatorname{ess}}(T_\mathrm{mass}^w)
$$
for $h$ small, so $\mu_n(h)\ne \sigma_{\operatorname{ess}}(T_\mathrm{mass}^w)$. From step 1 we conclude that $\mu_n(h)=\lambda_n(h)$, where $\lambda_n(h)$ is the $n$:th eigenvalue of $T_\mathrm{mass}^w$ counting multiplicity. By \eqref{eq:usingQplusve} we then get
$$
\lambda_n(h)/h=\mu_n(h)/h\le e_n+\mathcal{O}(h^{1/5})\quad\Longrightarrow\quad \varlimsup_{h\to0^+} \lambda_n(h)/h\le e_n,
$$
which completes the proof.
\end{proof}

Next, we shall provide proofs of the IMS formula (Lemma \ref{lem:IMS}) and Lemma \ref{lem:remainderboundL2periodic}. For this we shall use a pseudodifferential partition of unity. Let $J$ be as in Section \ref{sec:wellsnormalform}, that is, $J\in C_0^\infty(\RR)$ with $0\le J\le 1$ and $J(y)=1$ (resp.~0) if $\lvert y\rvert\le 1$ (resp.~$\lvert y\rvert\ge 2$). Let $\chi_1$ be given by \eqref{eq:chi1}, that is,
\begin{equation*}
\chi_1(x,\xi)=J(h^{-2/5}x)J(h^{-2/5}\xi).
\end{equation*}
Recall also from \eqref{eq:pou} that $(\chi_0(x,\xi))^2+(\chi_1(x,\xi))^2=1$.

\begin{lemm}\label{lem:pou}
Let $\chi_0$ and $\chi_1$ be as above. Then there are $X_0^w\in\Psi_{2/5}^{0,0}(\RR)$ and $X_1^w\in\Psi_{2/5}^{0,-\infty}(\RR)$ such that 
$X_j=\chi_j$ modulo $S^{-\infty,-\infty}(T^*\RR)$
and
$$
(X_0^w)^2+(X_1^w)^2=\id+R,\quad R\in\Psi^{-\infty,-\infty}(\RR),\quad (X_j^w)^*=X_j^w.
$$
\end{lemm}

\begin{proof}
We adapt the proof of \cite[Lemma 3.2]{sjostrand2007fractal} to symbols in $S_\delta^{m,k}$. Using the Weyl calculus we can write
$$
(\chi_0^w)^2+(\chi_1^w)^2=\id+r_1^w,\quad r_1\in S_{2/5}^{-2/5,-\infty}(T^*\RR)
$$
where we have taken advantage of \eqref{eq:pou} and the fact that the Poisson bracket $\{\chi_j,\chi_j\}$ vanishes, so the symbol of $(\chi_j^w)^2$ is $\chi_j^2$ modulo $S_{2/5}^{-2/5,-\infty}(T^*\RR)$. For $h$ small we set
$$
X_j^{1}=(1+r_1^w)^{-1/4}\chi_j^w(1+r_1^w)^{-1/4}.
$$
Then $X_0^1\in \Psi_{2/5}^{0,0}$ and $X_1^1\in \Psi_{2/5}^{0,-\infty}$ and since $r_1\in S_{2/5}^{-2/5,-\infty}$ we have $\sigma(X_j^1)=\chi_j$ modulo $S_{2/5}^{-2/5,-\infty}$. It is also easy to see that
$$
(X_0^1)^2+(X_1^1)^2=\id+r_2^w, \quad r_2\in S_{2/5}^{-3/5,-\infty}(T^*\RR),\quad (X_j^1)^*=X_j^1,
$$
by using the fact that $[(1+r_1^w)^{-\frac12},\chi_j^w]\in \Psi_{2/5}^{-3/5,-\infty}$ and $[(1+r_1^w)^{-\frac14},(\chi_j^w)^2]\in \Psi_{2/5}^{-3/5,-\infty}$. The result therefore follows by iterating this procedure.
\end{proof}

Recall that the massive term $1-\chi^w$ in \eqref{eq:massive} is defined by a cutoff function $\chi\in C_0^\infty(T^*\RR)$ independent of $h$.
For $h$ small we then have
\begin{equation*}
\supp\partial \chi_k\bigcap\supp(1-\chi)=\emptyset,\quad k=0,1,
\end{equation*}
where $\partial \chi_k$ is shorthand for the first order derivatives of $\chi_k$. In view of Lemma \ref{lem:pou} we get
\begin{equation}\label{eq:suppcondX}
X_k^w(1-\chi^w)\in\Psi^{-\infty,-\infty}(\RR),\quad k=0,1,
\end{equation}
by the Weyl calculus.

\begin{proof}[Proof of Lemma \ref{lem:IMS}]
We have
$$
( X_k^w)^2P_{\operatorname{mass}}^w+P_{\operatorname{mass}}^w( X_k^w)^2-2 X_k^w P_{\operatorname{mass}}^w X_k^w
=[ X_k^w,[ X_k^w,P_{\operatorname{mass}}^w]]
$$
so if we show that $[ X_k^w,[ X_k^w,P_{\operatorname{mass}}^w]]=\mathcal{O}(h^{6/5})$ on $L^2(\RR)$, then the result follows by summing over $k$ since $( X_0^w)^2+( X_1^w)^2=1$ mod $\Psi^{-\infty,-\infty}$ by Lemma \ref{lem:pou}. When proving the estimate we may replace $P_{\operatorname{mass}}^w$ by $P^w$ since $X_k^w(1-\chi^w)=\mathcal{O}(h^\infty)$ in $L^2$ by \eqref{eq:suppcondX}.

By Lemma \ref{lem:pou} we have $\partial  X_k\in S_{2/5}^{2/5,-\infty}$, where $\partial X_k$ is shorthand for the first order derivatives of $ X_k$. 
By assumption we have $P\in S^{0,k}(T^*\RR)$, $k\ge0$. Hence, $ X_k\# P-P\# X_k\in S_{2/5}^{-3/5,-\infty}\subset h^{3/5}S_{2/5}(1)$ by the symbol calculus, the main point being that it is bounded. In fact, inspecting \cite[Theorem 4.12]{zworski} and \cite[Theorem 4.18]{zworski} we see that
$$
\sigma([ X_k^w,P^w])= X_k\# P-P\# X_k
=-ih\{ X_k,P\}+\mathcal{O}_{S_{2/5}(1)}(h^{3(1-2/5)}).
$$
Since $\partial P_0=\mathcal{O}(h^{2/5})$ on the support of $\partial X_k$ by
Definition \ref{def:degeneratewell} we find that $\sigma([ X_k^w,P^w])=\mathcal{O}_{S_{2/5}(1)}(h)$. By \cite[Theorem 4.18]{zworski} we then get
$$
\sigma([ X_k^w,[ X_k^w,P^w]])=-ih\{ X_k,\sigma([ X_k^w,P^w])\}+h\mathcal{O}_{S_{2/5}(1)}(h^{3(1-4/5)})=\mathcal{O}_{S_{2/5}(1)}(h^{6/5}),
$$
which implies that $[ X_k^w,[ X_k^w,P^w]]=\mathcal{O}(h^{6/5})$ on $L^2(\RR)$, see \cite[Theorem 4.23]{zworski}.
\end{proof}

\begin{proof}[Proof of Lemma \ref{lem:remainderboundL2periodic}]
As in the proof of Lemma \ref{lem:IMS} we may replace $P_\mathrm{mass}^w$ by $P^w$ without changing the estimate. Since $ X_1^w(P^w-H_0^w) X_1^w=\gamma^*\circ (h^{3/2}J_1^w R_0^w J_1^w)\circ (\gamma^{-1})^*$ is unitarily equivalent to $ h^{3/2}J_1^wR_0^w J_1^w$ the result then follows from Lemma \ref{lem:remainderboundL2}.
\end{proof}

\section{Degenerate wells in the anti-chiral limit}\label{sec:antichiral}

Here we give a brief presentation of degenerate potential wells for the pseudodifferential Harper model in the anti-chiral limit.
The anti-chiral model allows for various quasimodes at potential wells located at different energy levels, but not necessarily at zero. To see this, we use the following simple lemma which reduces the analysis to scalar pseudodifferential operators.
\begin{lemm}
Let $H_{\operatorname{\Psi DO}}(w)$ be given by \eqref{eq:HarperPDO} in the anti-chiral limit $w=(w_0,0)$, with $k_{\perp} \in \ZZ/2$. Then $H_{\operatorname{\Psi DO}}(w_0,0)$ is unitarily equivalent to a Hamiltonian of diagonal form with symbol, with $+$ for $k_{\perp} \in \ZZ$ and $-$ for $k_{\perp} \in \ZZ+1/2$,
\[ \begin{split} \mathscr H_{\operatorname{ac}}(x,\xi) =\operatorname{diag}&(-(1\pm 2 \cos(2\pi \xi))-w_0U(x), -(1\pm 2 \cos(2\pi \xi))+w_0U(x),\\
&(1\pm 2 \cos(2\pi \xi))-w_0U(x),(1\pm 2 \cos(2\pi \xi))+w_0U(x)).\end{split}\] 
\end{lemm}

\begin{proof}
Let $b(x,\xi)$ be given by \eqref{eq:b} with $w=(w_0,0)$ and $k_{\perp} \in \ZZ/2$, so that $b^w(x,hD)=H_{\operatorname{\Psi DO}}(w_0,0)$. Conjugating with 
\begin{equation}\label{eq:conjach}
\mathscr U =\frac12 \begin{pmatrix*}[r] 1 & - 1 & -1 & 1 \\ -1 & 1 & -1 & 1 \\ -1 & -1 &1& 1\\ 1 & 1 & 1 &1  \end{pmatrix*},
\end{equation}
we find that $\mathscr H_{\operatorname{ac}}(x,\xi):=\mathscr U^*b(x,\xi) \mathscr U$ has the stated form.
\end{proof}
It is now easy to see from the diagonal form, since both the position and momentum variable appears in terms of a cosine 
\begin{equation*}
\cos(2\pi t) = \pm 1 \mp 2\pi^2 (t- t_0)^2 + \mathcal O((t - t_0)^4)
\end{equation*}
with $+$ in case of $t_0 \in \ZZ$ and $-$ for $t_0 \in \ZZ+1/2$, that the Hamiltonian $\mathscr H_{\operatorname{ac}}(x,\xi)$ admits quasimodes at various energy levels by the scalar Bohr-Sommerfeld rule given by Theorem \ref{theo:DiSJ}.
(Observe however that the spectral gap condition \eqref{eq:gap_condition} is not satisfied since the eigenvalues of $\mathscr H_\ach(x,\xi)$ coalesce at $x=\pm\frac13$, $\xi=\pm\frac13(\frac12)^{2k_\perp}$ mod $\ZZ^2$.) We state an example of one such result here.\footnote{One can also get, similarly to this result, wells with opposite sign for the anti-chiral Hamiltonian.}

\begin{thm}\label{thm:Harperantichiral}
Let $H_{\operatorname{\Psi DO}}(w)$ be given by \eqref{eq:HarperPDO} in the anti-chiral limit $w=(w_0,0)$, with $k_{\perp} \in \ZZ/2$. Then for $j=1,2,3,4$ there are functions $F^{(j)}(\tau,h)\sim\sum_{n=0}^\infty h^nF^{(j)}_n(\tau)$ with
$$
F_0^{(j)}(\tau)\equiv F_0(\tau)=\frac{\tau}{8\pi^2 w_0}+\mathcal O(\tau^2), \quad F_1^{(j)}=\tfrac{1}{2}, \quad j=1,2,3,4,
$$
with $\lambda^{(j)}_k(h)$ such that $F^{(j)}(\lambda_k^{(j)}(h),h)=kh$, together with quasimodes $u^{(j)}\in L^2(\mathbb T;\CC^4)$ which are nontrivial only in component $j$,
such that $(H_{\operatorname{\Psi DO}}(w_0,0)-z_j(k,h))u^{(j)}=\mathcal O_{L^2}(h^\infty)$, where $z_j(k,h)=c_j+\lambda^{(j)}_k(h)+\mathcal O(h^\infty)$ for $k\in\NN$, and
\begin{equation}\label{eq:cj}
c_1=-3-3w_0,\quad c_2=-3-w_0,\quad c_3=-1-3w_0,\quad c_4=-1-w_0.
\end{equation}
In particular, $H_{\operatorname{\Psi DO}}(w_0,0)$ has approximate eigenvalues $z_j(k,h)=c_j+8\pi^2 w_0(k+\frac12)h+\mathcal O(h^2)$ with $k\in\NN_0$ for $j=1,2,3,4$.
\end{thm}

\begin{proof}[Sketch of proof]
Due to the diagonal form of $\mathscr H_\ach$ it suffices to study scalar operators by choosing quasimodes $u\in L^2(\mathbb T;\CC^4)$ that are zero in all but one component. By Taylor's formula
$$
(\mathscr H_\ach(x,\xi))_{jj}=c_j+4\pi^2((\xi-\xi_j)^2+w_0(x-x_j)^2)+\mathcal O(\lvert (x,\xi)-(x_j,\xi_j)\rvert^4)
$$
where $c_j$ are as in \eqref{eq:cj}, and where $x_1=x_3=0$, $x_2=x_4=\frac12$ mod $\ZZ$, and $\xi_1=\xi_2=0$, $\xi_3=\xi_4=\frac12$ mod $\ZZ$ if $k_\perp\in\ZZ$, while $\xi_3=\xi_4=0$, $\xi_1=\xi_2=\frac12$ mod $\ZZ$ if $k_\perp\in\ZZ+\frac12$.

Consider $k_\perp\in\ZZ$ and $j=1$, and view $\mathscr H_\ach^w$ temporarily as an operator on $\RR$. We then add a massive term so that we can apply Theorem \ref{theo:DiSJ}: Choose $\chi\in C_0^\infty(T^*\RR)$ with $0\le\chi\le 1$ and $\supp\chi$ contained in a small neighborhood of $(0,0)$ so that 
$$
p(x,\xi):=3(1+w_0)+(\mathscr H_\ach(x,\xi))_{11}+(1-\chi(x,\xi))\ge0
$$
with equality only at $(x_1,\xi_1)=(0,0)$. Make the symplectic change of variables $(\tilde x,\tilde \xi)=(w_0^{-1/4}x,w_0^{1/4}\xi)$ and set $\tilde p(x,\xi)=p(\tilde x,\tilde \xi)$. Then
$$
\tilde p(x,\xi)=\frac{8\pi^2w_0}2(\xi^2+x^2)+\mathcal O(\lvert(x,\xi)\rvert^4)
$$
so by Theorem \ref{theo:DiSJ} there is a function $F(\tau,h)\sim\sum_{n=0}^\infty h^nF_n(\tau)$ with $F_0(\tau)=\tau(8\pi^2w_0)^{-1}$ and $F_1=1/2$ (since there are no lower order terms in $h$ in $\mathscr H_\ach(x,\xi)$; in particular, the subprincipal symbol of $\mathscr H_\ach^w$ is zero at $(x_1,\xi_1)$), such that for all $\delta>0$, the eigenvalues of $\tilde p^w(x,hD)$ in $(-\infty,h^\delta)$ are given as in Theorem \ref{theo:DiSJ} with $k\in\NN$. If $\lambda\in (-\infty,h^\delta)$ is an eigenvalue of $\tilde p^w(x,hD)$ with eigenvector $u\in L^2(\RR)$, and we write $\tilde u(x)=u(w_0^{1/4} x)$, it is easy to check that
$$
\lambda \tilde u(x)=\lambda u(w_0^{1/4}x)=\tilde p^w u(w_0^{1/4}x)=p^w \tilde u(x)
$$
so $\lambda$ is also an eigenvalue of $p^w(x,hD)$. By \cite[Lemma 14.10]{DimSjo} the eigenvector $u$ of $\tilde p^w$ is microlocalized to $(0,0)$, so $(1-\chi^w)u=\mathcal O(h^\infty)$. It follows that $(\mathscr H_\ach^w)_{11}$, as an operator on $\RR$, has eigenvalues near $-3(1+w_0)$ of the form $z(k,h)=-3(1+w_0)+\lambda^{(1)}_k(h)+\mathcal O(h^\infty)$. 
The same is true for $(\mathscr H_\ach^w)_{11}$ as an operator on $\mathbb T$, which can be seen by using a periodization argument similar to the one used in the proof of Corollary \ref{cor:periodicquasimodesnormalform}. Since translation is a symplectic change of variables, the same arguments can be applied to the other components of $\mathscr H_\ach^w$, both when $k_\perp\in\ZZ$ and when $k_\perp\in\ZZ+\frac12$.
\end{proof}

\begin{rem}\label{rem:achLEM}
In contrast to the pseudodifferential Harper model discussed in Theorem \ref{thm:Harperantichiral}, there are no wells in the anti-chiral low-energy model. Indeed, let $H_\ach^w(k_x)$ be given by \eqref{eq:tm20} with $w=(w_0,0)$ and $k_\perp=0$. Conjugating by $\mathscr U$ in \eqref{eq:conjach} shows that $H_\ach^w(k_x)$ is unitarily equivalent to the semiclassical operator $\mathscr L_\ach^w(x,hD)$ with symbol
\begin{align*}
\mathscr L_{\operatorname{ac}}(x,\xi) &=\operatorname{diag}(-\xi-w_0U(x), -\xi+w_0U(x),
\xi-w_0U(x),\xi+w_0U(x))\\&\quad+k_x\operatorname{diag}(-1,-1,1,1)
\end{align*} 
where $k_x=\mathcal O(h)$.
Each component of the principal symbol is a scalar symbol of real principal type so there are no point-localized states near zero energy, see \cite[Theorem 12.4]{zworski}.
\end{rem}

\end{appendix}

\section*{Acknowledgments}
This paper arose from a question about existence of flat bands in moir\'e systems of dimension not equal to two (dimension two being exemplified by twisted bilayer graphene), and we thank Ashvin Vishwanath for having suggested the topic.
The article benefited substantially
from comments and suggestions by the anonymous
referees and conversations with Nils Dencker and Maciej Zworski.
Jens Wittsten was partially supported by The Swedish Research Council grants 2019-04878 and 2023-04872.

\section*{Declarations}

\subsection*{Data availability}
Data sharing not applicable to this article.

\subsection*{Conflict of interest}
The authors have no competing interests to declare that are relevant to the content of this article.

\bibliographystyle{aomalpha}
\bibliography{references}

\end{document}